\documentclass{amsart}
\usepackage[dvipsnames]{xcolor}
\usepackage{hyperref}
\usepackage{soul}
\usepackage{amsfonts,amscd,amssymb,amsmath,amsthm,mathtools}
\usepackage{graphicx,caption,subcaption,mathrsfs,appendix}
\usepackage{xparse,enumerate}
\usepackage{tikz}
\usepackage{pgfplots}

\usetikzlibrary{decorations.pathmorphing}

\usepackage[foot]{amsaddr}

\begin{document}


\newtheorem{theorem}{Theorem}[section]
\newtheorem{lemma}[theorem]{Lemma}
\newtheorem{corollary}[theorem]{Corollary}
\newtheorem{conjecture}[theorem]{Conjecture}
\newtheorem{cor}[theorem]{Corollary}
\newtheorem{proposition}[theorem]{Proposition}
\theoremstyle{definition}
\newtheorem{definition}[theorem]{Definition}
\newtheorem{example}[theorem]{Example}
\newtheorem{claim}[theorem]{Claim}
\newtheorem{remark}[theorem]{Remark}

\newenvironment{pfofthm}[1]
{\par\vskip2\parsep\noindent{\sc Proof of\ #1. }}{{\hfill
$\Box$}
\par\vskip2\parsep}
\newenvironment{pfoflem}[1]
{\par\vskip2\parsep\noindent{\sc Proof of Lemma\ #1. }}{{\hfill
$\Box$}
\par\vskip2\parsep}


\newcommand{\R}{\mathbb{R}}
\newcommand{\C}{\mathbb{C}}
\newcommand{\T}{\mathbb{T}}
\newcommand{\D}{\mathbb{D}}
\newcommand{\G}{\mathcal{G}}
\newcommand{\Z}{\mathbb{Z}}
\newcommand{\Q}{\mathbb{Q}}
\newcommand{\E}{\mathbb E}
\newcommand{\N}{\mathbb N}

\newcommand{\Def}{\overset{\Delta}{=}}

\newcommand{\supp}{\operatorname{supp}}
\newcommand{\sgn}{\operatorname{sgn}}


\newcommand{\Prob}{\Pr}
\newcommand{\Var}{\operatorname{Var}}
\newcommand{\Cov}{\operatorname{Cov}}
\newcommand{\dtv}{d_{\operatorname{TV}}}
\newcommand{\Exp}{\mathbb{E}}
\newcommand{\expect}{\mathbb{E}}
\newcommand{\1}{\mathbf{1}}
\newcommand{\prob}{\Pr}
\newcommand{\weakto}{\Rightarrow}
\newcommand{\pr}{\Pr}
\newcommand{\filt}{\mathscr{F}}
\DeclareDocumentCommand \one { o }
{%
  \IfNoValueTF {#1}
  {\mathbf{1}  }
  {\mathbf{1}\left\{ {#1} \right\} }%
}
\newcommand{\Bernoulli}{\operatorname{Bernoulli}}
\newcommand{\Binomial}{\operatorname{Binom}}
\newcommand{\Binom}{\Binomial}
\newcommand{\Poisson}{\operatorname{Poisson}}
\newcommand{\Exponential}{\operatorname{Exp}}


\newcommand{\Ai}{\operatorname{Ai}}
\newcommand{\tr}{\operatorname{tr}}
\renewcommand{\det}{\operatorname{det}}

\newcommand{\Image}{\operatorname{Image}}
\newcommand{\Span}{\operatorname{Span}}


\DeclareDocumentCommand \norm { O{\cdot} } { \left\|{ #1 }\right\| }
\DeclareDocumentCommand \nuclear { O{\cdot} } { \left\|{ #1 }\right\|_{\nu} }
\DeclareDocumentCommand \Hhalf { O{\cdot} } { \left\|{ #1 }\right\|_{H^{1/2}} }
\DeclareDocumentCommand \rawev { O{i} O{N} } { \lambda_{ {#1 }}^{({#2})} }

\newcommand{\GF}{\mathbf{G}}
\newcommand{\F}{\mathbf{F}}
\newcommand{\BSA}{\mathscr{B}}
\newcommand{\CUEF}{\mathbf{U}_{\corE{N}}}
\newcommand{\CUEFr}{\mathbf{U}_{\corE{N,r}}}
\NewDocumentCommand {\BF} { } {\nu}
\newcommand{\WN}{\mathbf{Z}}
\newcommand{\RWN}{\mathbf{R}}
\newcommand{\dH}{d_{\mathbb{H}}}
\newcommand{\dT}{d_{T}}
\newcommand{\ddH}{\tilde{d}_{\mathbb{H}}}
\newcommand{\HM}{\mathfrak{m}}
\DeclareDocumentCommand \ROT { O{\theta} } { 
  {Q}_{ {#1} }
}
\newcommand{\Meso}{\mathscr{E}}
\newcommand{\MFI}{{Y_p}}
\newcommand{\FI}{{W_p}}
\def\corO{}
\def\corE{}

\DeclareDocumentCommand{\FMC}{ g }{
  \IfNoValueTF {#1}
  {
    \mathbb{F}
  }
  {
    \mathbb{F}_{ {#1} }
  }
}
\DeclareDocumentCommand \EvI { g o } { 
  \EVENT{\mathcal{A}}{ }[][#2]
}
\DeclareDocumentCommand \Ev { G{\BF,p,t_0} o } { 
  \EVENT{\mathcal{A}}{\ell}[#1][#2]
}
\DeclareDocumentCommand \EvU { G{\rho,t} o } { 
  \EVENT{\mathcal{A}}{u}[#1][#2]
}
\NewDocumentCommand {\EvL}{ O{\BF} O{\theta} } { 
  \EVENT{\mathcal{A}}{\ell2}[#1][#2]
}
\DeclareDocumentCommand \EvUU { G{\rho,t} o } { 
  \EVENT{\mathcal{A}}{u2}[#1][#2]
}
\DeclareDocumentCommand \BS { o o } { 
  \EVENT{\mathcal{B}}{1}[#1][#2]
}
\DeclareDocumentCommand \BL { G{t_0} o } { 
  \EVENT{\mathcal{B}}{2}[#1][#2]
}
\DeclareDocumentCommand \BLL { o o } { 
  \EVENT{\mathcal{B}}{3}[#1][#2]
}
\DeclareDocumentCommand \ES { G{t_0} o } { 
  \EVENT{\mathcal{E}}{1}[#1][#2]
}
\DeclareDocumentCommand \ESU { o } { 
  \EVENT{\mathcal{E}}{2}[][#1]
}
\DeclareDocumentCommand \EL { G{\rho,t} o } { 
  \EVENT{\mathcal{E}}{u}[#1][#2]
}
\DeclareDocumentCommand \ELL { G{p,t_0} o } { 
  \EVENT{\mathcal{E}}{\ell}[#1][#2]
}
\DeclareDocumentCommand \ELM { o o } { 
  \EVENT{\mathcal{E}}{\HM}[#1][#2]
}
\DeclareDocumentCommand \DL { o O{k} } { 
  \EVENT{\mathcal{D}}{#2}[\rho][#1]
}
%
\DeclareDocumentCommand{\EVENT}{ m m o o }
{
  \IfNoValueTF {#3}
  {
    {#1}_{#2}^{{\BF}}
  }
  {
    \IfNoValueTF {#4}
    { {#1}_{#2}^{{#3}} }
    { {#1}_{#2}^{{#3}}({#4}) }
  }
}

\DeclareDocumentCommand \Bias { g o }{
  \IfNoValueTF {#1}
  {
    \IfNoValueTF {#2}
    {
      \mathfrak{B}
    }
    {
      \mathfrak{B}_{#2}
    }
  }
  {
    \IfNoValueTF {#2}
    {
      \mathfrak{B}(#1)
    }
    {
      \mathfrak{B}_{#2}(#1)
    }
  }
}

\title[CUE field]{The maximum of the CUE field}
\author{Elliot Paquette}
\address{Department of Mathematics, Weizmann Institute of Science}
\email{elliot.paquette@gmail.com}
\author{Ofer Zeitouni}
\address{Department of Mathematics, Weizmann Institute of Science and Courant institute, NYU}
\email{ofer.zeitouni@weizmann.ac.il}
\thanks{The work of both authors was supported by  grants
	from the Israel Science Foundation and from the US-Israel Binational 
	Science Foundation.
  EP gratefully acknowledges the support of NSF Postdoctoral Fellowship DMS-1304057.
  This project has received funding from the European Research Council (ERC) under the European Union's Horizon 2020 research and innovation programme (grant agreement No. 692452).
}
\date{\today}
\maketitle
\begin{abstract}
	{
Let $U_N$ denote a Haar Unitary matrix of dimension $N,$ and consider the 
	field
	$\CUEF(z) = \log |\det(1-zU_N)|$ for $z\in \mathbb{C}$. Then,
    $$\frac{\max_{|z|=1} \CUEF(z) -
    \log N 
    + \frac{3}{4} \log\log N}{
    \log\log N}
    \to 0 $$
  in probability. This provides a verification up to second order 
  of a conjecture of Fyodorov, Hiary and Keating, improving on the recent
  first order verification of Arguin, Belius and Bourgade.  
  }
\end{abstract}
\section{Introduction and Overview}
\label{sec:def}

Let $U_N$ denote a Haar Unitary matrix of dimension $N,$ and consider the 
	field
\[
  \CUEF(z) = \log |\det(1-zU_N)|,
\]
defined for all $z \in \mathbb{D} =
\left\{ z \in \mathbb{C}~:~|z| < 1 \right\}.$  
{
	The supremum of this function, which by the maximum principle
	equals $\max_{z\in \mathbb{T}} \CUEF(z)=:\CUEF^*$, where $\mathbb{T}=\{z\in \mathbb{C}: |z|=1\}$,}
is the subject of the following conjecture.
\begin{conjecture}[Fyodorov-Hiary-Keating \cite{FK,FH}]
  There is a random variable $\xi$ so that as $N \to \infty,$
  \[
	  {
		  \CUEF^*
	  }
	  -
    \log N 
    + \frac{3}{4} \log\log N
    \weakto
    \xi,
  \]
  with $\weakto$ denoting convergence in distribution.
  \label{conj}
\end{conjecture}
\noindent The exact distribution of $\xi$ is conjectured as well (see \cite{FH} for further details).

Arguin, Belius and Bourgade \cite{ABB} 
have recently obtained the leading order in  Conjecture \ref{conj}, that is,
they show that 
    $\CUEF^*/
    \log N \to 1$ as $N\to\infty$, in probability. 
    In this paper, we
    improve this to the convergence of the 
first two terms in the expansion, namely we prove:
\begin{theorem}
  As $N \to \infty,$
  \[
    \frac{
	    \CUEF^*
    -\log N 
    + \frac{3}{4} \log\log N}{
    \log\log N}
    \to 0
  \]
  in probability.
  \label{thm}
\end{theorem}
\corO{After this paper was posted and submitted for publication,
  Chhaibi,  Madaule and Najnudel \cite{CNM}
  have obtained a significant improvement
  of Theorem \ref{thm}. Namely, they prove that the sequence
  $\{\CUEF^*+\log N-\frac34\log\log N\}$ is actually tight, 
  and even consider other ensembles of random unitary matrices.
  Their proof is based on a different representation of $\CUEF$ than ours.
}

\subsection{Background, motivation \corO{and extensions}}
The field $\{\CUEF(z)\}$ is nearly Gau\-ssian in many respects.
	For example, as part of their work concerning relations between 
	the eigenvalues of 
	random unitary matrices and the Riemann zeta function,
	Keating and Snaith \cite{KS} proved that 
	for fixed $\theta$, the random variable
	$\CUEF(e^{i\theta})/\sqrt{( \log N)/2}$ 
	converges in distribution to a standard Gaussian. This convergence
	was 
	improved in various ways \cite{HKO,Wieand}, and is related
	to the convergence of linear statistics of eigenvalues of
	random unitary matrices, going back to the foundational work of
	Diaconis and Shashahani \cite{DS,DE}, because 
	$\CUEF(z)=\sum_{h=1}^N \log |1-ze^{i\theta_h}|$ where 
	$\left\{ e^{i\theta_h} \right\}_{h=1}^N$ are the eigenvalues of $U_N.$

	As we will make precise below, the limiting correlation of the
	field $\{\CUEF(z)\}_{z\in \mathbb{T}}$ is that of a 
	\textit{logarithmically correlated Gaussian field}. 
	Conjecture \ref{conj} is then the prediction that 
	the maximum of $\CUEF(z)$  over the unit circle behaves like
	the maximum of a logarithmically correlated Gaussian field, 
	for which convergence results are available, see
	\cite{bramson83} for branching Brownian motion,
	\cite{BDG,madaule} for more general such fields including the
	two dimensional discrete Gaussian free field, and
	\cite{DRZ} for a universality result concerning
	the distribution of the maximum of logarithmically correlated Gaussian
	field. See also \cite{FBo} \corO{(which served as motivation to \cite{FK,FH})} 
	for early predictions and computations
	concerning extremes of Gaussian correlated fields in the context of
	the paradigm of \textit{freezing transition}. \corO{The fields studied in
	  \cite{FBo} are precisely those appearing as limits of 
	$\CUEF$.}
	A word of caution is however that while the minimum of Gaussian
	log-correlated fields has the same distribution as minus the
	maximum, obviously $\min_{z\in \mathbb{T}} 
\CUEF(z)=-\infty$.

	The attempt to use techniques developed in the context of extrema of 
	logarithmically correlated fields, and in particular a modified version of
	the second moment method, to the study of $\CUEF^*$ is natural and in
	fact is behind the study \cite{ABB}. Besides employing a different
	regularization scheme (based on a truncation of a Fourier series in \cite{ABB}, and on a more geometric notion of rays here), our approach employs two additional components that allow us to apply the second moment method in
	greater accuracy. First, in the upper bound, 
	we use directly a certain monotonicity 
	property from \cite{Johansson}, see 
	Proposition \ref{prop:domination}. Second, we employ certain identities, that can be traced back to \cite{Baxter} and that were used in \cite{Johansson},
	to compute exponential moments of linear combinations of $\CUEF(z)$ for
	different $z$'s. 
    These in particular allow us to make some comparisons between $\CUEF(z)$ and $\GF(z)$ at optimal scales 
    ($1-|z| =\Theta(N^{-1})$)
    (see Proposition~\ref{prop:clt} and Corollary~\ref{cor:exact}).
    A challenging question that we could not resolve and thus leave open is whether these
	techniques can be improved to yield a complete proof of Conjecture
	\ref{conj}. \corO{As discussed above, the very recent \cite{CNM} is a further step in that
	direction. We note that while the approach of 
	\cite{CNM} also rests on a truncated second moment
argument, the approximation they use is different than both ours
and that of \cite{ABB}, and it rests on an efficient use of modified
Verblunsky coefficients.}
	
	\corO{An important motivation behind the study of eigenvalues of random unitary
matrices is the conjectured link with the zeros of the Riemann zeta function
\cite{KS}; this link also served as motivation to the  the Fyodorov-Bouchaud study and to the Fyodorov-Hiary-Keating
conjecture. When translated back to the RZF setup, the latter reads as the conjecture
that the maximum of the RZF on a ``typical'' bounded interval of the critical 
line, at height
$T$, is roughly $\exp(\log\log T-\frac34 \log\log\log T+O(1))$. Of course, our
results do not shed any light on the latter conjecture. It is
worthwhile to mention
in this context  the work of Arguin, Belius and Harper \cite{ABH}, who analyze
a different ``random'' model for the RZF. }

\corO{We finally note several possible extensions of our work. First, an 
anonymous referee pointed out to us the equality
$$ \sup_{|z|\leq 1} |P_N'(z)|=\frac{N}{2} \sup_{|z|=1} |P_N(z)|,$$
valid for any polynomial of degree $N$ whose zeros have modulus equal to $1$
\cite[Pg. 512]{Lax}. This applies in particular to the derivative of the
determinant of $U_N$, and Theorem \ref{thm} thus gives an estimate 
on the latter.}

\corO{In another direction, one can also consider the imaginary
part of the $\log \det(I-zU_N)$ on the unit circle (taking the 
limit as $|z|\nearrow 1$). Our methods apply also in that
setup. We do not provide further details, since this extension is
already considered in \cite{ABB} and \cite{CNM}.}

\corO{Finally, it is natural to consider 
similar questions concerning the maximum of the log-determinant
of random Hermitian matrices, such as the GUE
 (over a compact subset of the real line). At the
level of the leading term, such results are now available in \cite{LP}.}

\subsection{Deterministic relaxation}
\label{sec:deterministic}

Our approach is to study the maximum of $\CUEF(z)$ on the unit circle by studying the field on its interior.  For the purposes of bounding the maximum from below, this is particularily convenient, because $\CUEF(z)$ is a harmonic function on $\mathbb{D}.$  Hence, we have that almost surely
\[
  \sup_{|z| < 1} \CUEF(z) = \max_{z\in \mathbb{T}} \CUEF(z),
\]
and so any value of $\CUEF(z)$ on the interior of the disk serves as a lower bound to the maximum.

On the other hand, because $\CUEF(z)$ can be written as
\(
  \CUEF(z) = \sum_{h=1}^N \log| 1 -ze^{i\theta_h}|,
\)
where $\left\{ e^{i\theta_h} \right\}_{h=1}^N$ are the eigenvalues of $U_N,$ we have the following estimate for the maximum.
\begin{lemma}
  For any $M>0,$ there is an $N_0(M)$ sufficiently large so that for all 
  {integer}
  $N > N_0(M),$ and any $\left\{ \theta_h \right\}_1^N \subset \mathbb{T},$ the function $F(z) = \sum_{h=1}^N \log| 1 -ze^{i\theta_h}|$ 
  satisfies
  \[
    \max_{|z|=1-MN^{-1}} F(z)
    \geq \max_{|z|=1} F(z) - M.
  \]
  \label{lem:deterministic}
\end{lemma}
\begin{proof}
  If $0 < r_1 < r_2<1,$ then the maximum of $\log |1-r_2\omega| - \log |1-r_1\omega|$ over $\omega \in \T$ is attained at $\omega = -1.$  Hence,
  \begin{equation}
    \sum_{h=1}^N \log|1-r_2e^{i\theta_h}|
    \leq \sum_{h=1}^N \log|1-r_1e^{i\theta_h}|
    +N\log\left( \frac{1+r_2}{1+r_1} \right).
    \label{eq:radialslide}
  \end{equation}
  Letting $z_{*}$ be the optimizer of $\max_{|z|=1} F(z),$ we can estimate
  \[
    F\left( (1-{M}{N}^{-1})z_* \right)
    \geq F(z_*) - N\cdot\log\left(\frac{2}{2-{M}{N}^{-1}}\right) 
    \geq F(z_*) - \frac{M}{2-{M}{N}^{-1}}, 
    \]
  from which the claim follows.
\end{proof}

In fact, because of the correlation structure of the field $\CUEF,$ the maximum should be determined by $F(z)$ on a grid of angular spacing of order $N^{-1}.$  Most of the work will be to estimate the maximum on such a grid (in reality on a grid of spacing $(N\log N)^{-1}$), from which it will be possible to extend to a dense mesh by a global union bound.

\subsection{Gaussian field}
\label{sec:GF}
As discussed above, the field $\CUEF$ is very nearly Gaussian in many respects.
Hence, we introduce a \corE{real-valued} Gaussian field $\GF$ whose covariance is given by the limiting covariances of $\CUEF.$  Let $\GF$ be a centered Gaussian field on $\mathbb{D}$ with $\GF(0) = 0$ almost surely and covariance
\begin{equation}
  \label{eq:gfcov}
  \Exp \GF(z) \GF(y) = -\frac{\log| 1-z\bar{y}|}{2}.
\end{equation}

An alternative description of this field is given as 
\[
    \GF(z) = \langle \log | 1-(z)(\cdot)|, \mathcal{W} \rangle_{\T},
\]
where for continuous functions $f_1,f_2$
\[
    \langle f_1,f_2 \rangle_{\T} = \frac{1}{2 \pi i} \int_{\T} f_1(z)
		f_2^{\corO{*}}(z)\,\frac{dz}{z}.
\]
Here we take $\mathcal{W}$ to be the Gaussian field on $\T$ which is the weak limit of
\[
    \mathcal{W}(\omega) = \operatorname{w-lim}_{N \to \infty} \sum_{h=-N}^N \frac{\sqrt{h}}{\sqrt{2}}\cdot Z_h \omega^h
\]
for i.i.d.\,complex normals $\left\{ Z_h \right\}_{h=1}^\infty,$ each having independent real and imaginary parts of variance $1$ and satisfying the symmetry condition $Z_{-h}=\overline{Z}_h.$
From this, one can recover \eqref{eq:gfcov} using the Fourier series of $\log|1-z\omega|.$  Moreover, for any function $f:\T\to \C$ with sufficiently rapidly decaying Fourier coefficients, we have that
\begin{equation}
    \Exp[
      \langle f, \mathcal{W} \rangle_{{\T}}
      \langle g, \mathcal{W} \rangle_{{\T}}
    ]
    = \sum_{h \in \mathbb{Z}} \frac{h}{2} \hat f(h) \hat g(-h).
    \label{eq:Fourier}
\end{equation}

Using existing machinery, such as \cite{DRZ}, it is relatively straightforward to show that Conjecture~\ref{conj} would hold for the maximum of $\GF(z)$ restricted to the disk $|z| = 1-N^{-1}.$  
This is because $\GF(z)$ is a canonical example of a log-correlated Gaussian field, of which perhaps the most central example is branching random walk.  Indeed, identifying a branching structure is one of the key tools to answering questions about the maxima of such fields.  The branching structure for $\GF(z)$ enters through hyperbolic geometry.

Specifically, we recall the hyperbolic metric $\dH$ on $\mathbb{D},$ under which $\mathbb{D}$ is often referred to as the Poincar\'e disk model of the hyperbolic plane.  For any point $z \in \mathbb{D},$ the distance of $z$ to $0$ is given by
\begin{align}
  &\dH(0,z) = \log \left( \frac{1+|z|}{1-|z|} \right). \nonumber \\
  \intertext{We also recall the hyperbolic disk automorphism:}
  &T_y(z) := \frac{z-y}{1-z\bar{y}},
  \label{eq:diskautomorphism}
\end{align}
which is an isometry of the Poincar\'e disk taking $y$ to $0.$  
{For two arbitrary points $y,z \in \mathbb{D},$ we can then write}
\begin{align}
  \dH(y,z) &= \dH(0,T_y(z))
  = \log \left( \frac{1+|T_y(z)|}{1-|T_y(z)|} \right)\nonumber \\
  \intertext{The variance of the difference of the field $\GF$ at two points $y,z \in \mathbb{D}$ is given by}
  \Var
  \left( \GF(z) - \GF(y)	\right)
  &= \frac{1}{2}\log
  \left(
  \frac{|1-z\bar{y}|^2}{
    (1-|z|^2)
    (1-|y|^2)
  }
  \right).\nonumber \\
  \intertext{It is now a straightforward calculation to see that this could also be expressed as}
  \Var\left( \GF(z) - \GF(y)	\right)
  &=\frac{1}{2}\log 
  \left( 
  \frac{1}{1-|T_y(z)|^2}
  \right).\nonumber \\
  \intertext{In particular, we can write this in terms of the hyperbolic distance between $y$ and $z$ by the formula}
  \Var\left( \GF(z) - \GF(y)	\right)
  &=\log\left( \cosh\left( \tfrac{\dH(z,y)}{2} \right) \right) 
  \label{eq:varcosh}\\
  &=\tfrac{\dH(z,y)}{2} - \log 2 + O(e^{-\dH(z,y)}),
  \nonumber \\
  \intertext{with the approximation uniform in all $z,y \in \D.$  We can also write the covariance in a similar way:}
  \Exp \GF(z)\GF(y) 
  &= -\frac12\left[ \Var\left( \GF(z) - \GF(y)	\right) - \Var(\GF(z)) - \Var(\GF(y)) \right] \nonumber. \\
  \intertext{Using that $G(0)=0$ almost surely, we can express this in terms of \eqref{eq:varcosh} as} 
  \Exp \GF(z)\GF(y) 
  &= \frac12\log\left( 
  \tfrac
  { \cosh( {\dH(0,y)}{2}^{-1} )\cosh( {\dH(0,z)}{2}^{-1} )}
  {\cosh( {\dH(z,y)}{2}^{-1} )}
  \right). \label{eq:covcosh}
\end{align}

It is possible to discretize hyperbolic space in such a way that the branching structure now appears naturally through the geometric structure of the discretization (see \cite[Section 14]{Kenyon}).
We however will not use any such discretization directly.  Instead, it will be enough for us to know that the covariance structure can be compared directly to branching random walk. 

Let $\left\{ \zeta_i \right\}_0^\infty$ be points on the positive real axis 
with 
\begin{equation}
  \label{eq-zetadef}
  \corO{  \zeta_0 = 0,  \quad \dH(\zeta_i,\zeta_j) = |i-j|.} 
\end{equation}
{The points $\zeta_i$ (and their rotated version $e^{i\theta} \zeta_i$ with 
angles $\theta\sim 2\pi k/N$, $k\in \Z$) will provide us with a convenient
skeleton along which the
field $\CUEF(z)$ behaves roughly as a Gaussian branching random walk.}
For $\theta \in \R,$ we wish to estimate the distance $\dH(\zeta_i, e^{i\theta}\zeta_j).$  Indeed the following is a quick calculation.
\begin{lemma}
  Uniformly in $h,j \in \mathbb{N}$ and $\theta \in [-\pi,\pi]$
  \[
    \dH(\zeta_h, e^{i\theta}\zeta_j)
    = h + j - 2\min\{-\log|\sin\tfrac{\theta}{2}|,h,j\} + O(1).
  \]
  When $k = \min\{h,j\} > -\log |\tfrac\theta2|$ the error term can be estimated by $Ce^{-k}|\theta|^{-1}$ for some sufficiently large absolute constant $C>0.$ For the covariances of $\GF,$ it follows that
  \[
    \Exp \GF(\zeta_h)\GF(e^{i\theta}\zeta_j)
    =\tfrac{1}{2}\min\{-\log|\sin\tfrac{\theta}{2}|,h,j\}
    - \tfrac{\log 2}{2} + O(1),
  \]
  where again the error term can be estimated by $C\min\{e^{-k}\theta^{-1},1\}.$
  \label{lem:branch}
\end{lemma}
\begin{proof}
  For a hyperbolic triangle with side lengths $a,b,c$ with $\theta$ the angle opposite $a,$ the hyperbolic law of cosines says that
  \[
    \cosh a = 
    \frac{\cosh(b+c)}{2}(1-\cos\theta)
    +\frac{\cosh(b-c)}{2}(1+\cos\theta).
  \]
  We apply this with $a=\dH(\zeta_h,e^{i\theta}\zeta_j),$ $b=h$ and $c=j.$
  The remainder is a straightforward case-by-case analysis, noting that when  $k = \min\{h,j\} > -\log |\sin\tfrac\theta2|,$ the first term dominates, and otherwise the second term dominates.  Using \eqref{eq:covcosh}, this estimate can be transferred to the covariances, since for $x \geq 0,$
  \[
      \log(\cosh(\tfrac{x}{2})) = \tfrac{x}{2} - \log 2 + O(e^{-x}).
  \] 
\end{proof}
Hence 
  the covariance structure of 
  $\left( \GF(\zeta_j),\GF(e^{i\theta} \zeta_j) \right)_{j=1}^\infty$ is,
  up to universally 
  bounded 
  additive errors,
the same as that of two Gaussian simple random walks which have identical increments until step $\log |\tfrac2\theta|$ and have independent increments afterwards.  As a corollary of this, we have that there is an absolute constant $C>0$ so that for any three points $x,y,z \in \D,$
\begin{equation}
    \label{eq:correlationbound}
    \left|
    \Exp[\GF(x)\left( \GF(y) - \GF(z) \right)]
    \right|
    \leq \frac{\dH(y,z)}{2} + C.
\end{equation}



\subsection{Barrier method overview}
\label{sec:secondmoment}

The approach we take to estimating the maximum of $\CUEF$ is an adaptation of one developed to estimate the maximum of branching Brownian motion~\cite{Bramson78} (see also \cite{ABR09}, \cite{Aidekon} and \cite{BDZ14}
for the more closely related case of branching random walk). This method is also ubiquitous in the study of the extremes of log-correlated Gaussian fields, see
e.g.
\cite{BZ10}, \cite{madaule}, \cite{DRZ}.

In light of Lemma~\ref{lem:deterministic} \corO{(with $M=2$)}, we roughly need to estimate $\CUEF$ on the points 
\[
  \left\{ e^{2\pi i hN^{-1}}(1-\corO{2} N^{-1}) \right\}_{h=1}^N.
\]
Because of correlations in $\CUEF$, techniques that treat $\CUEF$ at these points  as independent variables fail to capture the behavior of the maximum of $\CUEF$.  Roughly speaking, if $\CUEF$ is unusually large at a single point, it will be unusually large at many nearby points.  The extent to which this is true is so great that a union bound fails to give the correct upper bound on $\CUEF.$  A standard second moment method argument, which would be used to give a lower bound for the maximum of $\CUEF,$ fails even more spectacularly.

To fix this in the case of branching Brownian motion, a key insight of \cite{Bramson78} is to work on an event where all particles are constrained to lie below some time-evolving barrier.  In our situation, this means we do the following.
Let 
\corO{
  \begin{equation}
    \label{eq-defn}
    n = \lfloor \dH(0, 1-\corO{2} N^{-1}) \rfloor \sim \log N.  
  \end{equation}
  Recall the points $\zeta_i$, see \eqref{eq-zetadef}.
}
With $t \approx -\tfrac34\log n$, to be defined later, define the subset of fields $\F : \D \to \R$
\begin{equation}
  \label{eq:Ev0}
  \EvI{} = \left\{ \F(\zeta_i) < i + C\log n,~\forall~1 \leq i \leq n, \text{ and } 
n+ t - 1 < \F(\zeta_n) < n + t \right\},
\end{equation}
and define $\EvI{ }[\omega]$ for 
  $\omega=e^{i\theta} \in \T$ to be the pushforward of 
  $\EvI{}$ under the map that rotates the field by $\theta.$
\begin{remark}
  We will need to modify the definition of $\EvI{}$ for technical reasons.  See \eqref{eq:Ev} for the events we will ultimately use.
\end{remark}

It is straightforward to show that with high probability (that is, with probability going to $1$ as $N\to \infty$), we have that
\[
  \CUEF(\zeta_j e^{2\pi i hN^{-1}}) < j + C\log n
\]
for all $1 \leq j \leq n$ and all $h \in [N] = \left\{ 1,2, \dots, N \right\}.$  This we do by computing exponential moments, applying Markov's inequality and employing a simple chaining argument.  Hence, the barrier introduced in \eqref{eq:Ev0} is in a sense typical.

Next, we show that the event $\CUEF \in \EvI{}$ has probability nearly equal to that which one would get if it were the case that $\left\{\CUEF(\zeta_i)\right\}_0^n$ were a random walk, i.e.\,
\begin{equation}
  \label{eq:Evrough}
  \Pr(\CUEF \in \EvI{}) = \frac{e^{-n-2t}}{n^{3/2}}e^{o(\log n)}.
\end{equation}
Indeed in the case of random walk the extra error term $e^{o(\log n)}$ 
can be much improved.  For what we seek to prove here, this estimate will be sufficient.

We then define the counting variable
\[
    Z = \sum_{h \in [N]} \one[{\CUEF \in \EvI{}[e^{2\pi i hN^{-1}}]}].
\]
An upper bound on the maximum of $\CUEF$ on $\{\zeta_n e^{2\pi i hN^{-1}}\}$ now follows by estimating $\Exp Z$ and summing over $t.$
A lower bound will proceed by a second moment method applied to $Z,$ i.e.\,estimating
\[
    \Pr\left[ Z > 0 \right] \geq \frac{
    (\Exp Z)^2
    }{
   \Exp(Z^2)
    }
    =
    \frac{
    N\Pr(\CUEF \in \EvI{})^2
    }{
        \sum_{\omega}
        \Pr\left( \CUEF \in \EvI{}[1] \cap \EvI{}[\omega] \right)
    },
\]
with the sum over all $\omega \in \left\{ e^{2\pi i hN^{-1}} : h \in [N] \right\}.$

  To control the second moment,
  we must show that the correlation between the
  two events $\EvI{}[1]$ and $\EvI{}[\omega]$ with $|\omega-1|\geq N^{-1}$ 
  is again similar to that of branching random walk, which translates to an estimate of the form
\[
    \Pr\left( \CUEF \in \EvI{}[1] \cap \EvI{}[\omega] \right) \leq  
    \Pr(\CUEF \in \EvI{}[1])^2 e^{-\log (|\omega-1|)}e^{o(\log n)}.
\]
Some care is needed in that for $\omega$ which are \emph{very} separated from $1$, meaning a $1-o(1)$ fraction of the phases $e^{2\pi i hN^{-1}},$ we need the stronger estimate
\[
    \Pr\left( \CUEF \in \EvI{}[1] \cap \EvI{}[\omega] \right) =  
    \Pr(\CUEF \in \EvI{}[1])^2(1+o(1)).
\]

\subsection{Strong Gaussian approximations}
\label{sec:normalintro}

Estimates like \eqref{eq:Evrough} are nontrivial, even in the case of simple random walk.  Hence, we approach this problem by showing that $\CUEF$ is very nearly $\GF$.  Once in the Gaussian context, a direct comparison to a Gaussian simple random walk is possible using Gaussian comparison inequalities.  

Ideally, we would like to show that
\begin{equation}
  \label{eq:Ev1}
  \Exp \one_{\EvI{}}( \CUEF ) = \Exp \one_{\EvI{}} (\GF) + o(N^{-1}),
\end{equation}
The computation of $\Exp \one_{\EvI{}} (\GF)$ is still nontrivial due to the correlation structure of $\GF$.  Indeed, letting $\zeta(t)$ be a unit speed hyperbolic geodesic, we have that $\GF(\zeta(t))$ is a non-Markovian, smooth Gaussian process.  However, the process $R^{-1/2}\GF(\zeta(Rt))$ has a Brownian limit as $R \to \infty.$  Further, the correlations converge sufficiently quickly that even $\GF(\zeta(t))$ at equally spaced times can be compared to a Gaussian random walk. 

Thus, our main task is to show a sufficiently strong quantitative comparison between $\CUEF$ and $\GF.$  One of the more striking exact identities that holds for $\CUEF$ is 
  the following
type of stochastic monotonicity in $N.$  
\begin{proposition}
  For any $\left\{ \lambda_i \right\}_1^k \subset \R$ and any $\left\{ z_i \right\}_1^k \subset \mathbb{D},$ we have 
  \[
    \Exp e^{ \sum_{i=1}^k \lambda_i \CUEF(z_i) }
    \leq 
    \Exp e^{ \sum_{i=1}^k \lambda_i \GF(z_i) }.
  \]
  \label{prop:domination}
\end{proposition}
\begin{proof}
  This is a special case of 
	  \cite[Lemma 2.9]{Johansson}.
	  In summary, it is shown \corO{that}
		the left-hand side increases monotonically in $N$ to its limit, which is given by the Strong-Szeg\H{o} limit theorem and which is equal to the right-hand side.
\end{proof}
In particular, we have a sub-Gaussian tail bound for any $\CUEF(z_i)$ with implied variance given by the variance of $\GF(z_i).$  This, together with 
Lemma \ref{lem:deterministic}, is enough to give a short proof that for each $\epsilon > 0,$ $\max_{z\in \T} \log |\det(1-zU_N)| < (1+\epsilon)\log N$ with high probability.

Before formulating exact results relevant to \eqref{eq:Ev1}, we will transform the problem.  In particular, a natural method to estimate $\Exp \one_{\EvI{}} (\GF)$ is to perform a change of measure to remove the drift from the event ${\EvI{}}.$  Indeed, if we bias the measure of $\GF$ by a Radon-Nikodym derivative proportional to $e^{c \GF(z)},$ the field under this bias will have the law of $\GF + \mu,$ where $\mu$ is a deterministic function $\mathbb{D} \to \R$ (see Lemma~\ref{lem:means}).

The biasing factor relevant to $\EvI{}$ is $e^{2 \GF(\zeta_n)}.$  
Note that by the definition of $\EvI{}$ \corO{
  and \eqref{eq-defn},} we have that
\[
  e^{2t - 2} 
  \leq
  \frac{\Exp \one_{\EvI{}}(\GF) e^{2\GF(\zeta_n)}}
  {N^2 \Exp \one_{\EvI{}}(\GF)}
  \leq e^{2t}.
\]
On the other hand, we have that
\[
  \frac{
    \Exp \one_{\EvI{}}(\GF) e^{2\GF(\zeta_n)}
  }{
    \Exp e^{2\GF(\zeta_n)}
  } = \Exp \one_{\EvI{}}(\GF + \mu).
\]
Using that $\GF(\zeta_n)$ has variance $\tfrac{1}{2} \log N + O(1),$ we arrive at the bounds
\[
  \Exp \one_{\EvI{}}(\GF)
  \asymp
  \frac{e^{-2t}\Exp \one_{\EvI{}}(\GF + \mu)}{N},
\]
where $\asymp$ denotes equality up to absolute multiplicative constants.  The event $\Exp \one_{\EvI{}}(\GF + \mu)$ now has a probability of polylogarithmic order in $N$.  This will allow us to prove a Gaussian approximation theorem with much larger additive error, provided we can additionally bias $\CUEF$ by exponential factors.  

Indeed this is the case.  To state our approximations, we introduce an additional field on $\D$ which we use for mollification purposes.  Let $\WN$ be a white noise on $\D,$ that is $\left\{ \WN(z) \right\}_{z\in S}$ is jointly centered Gaussian for every finite subset $S \subset \D$, $\WN(z)$ and $\WN(w)$ are independent for all $z \neq w,$ and 
$\Exp \WN(z)^2 = 1$ for all $z \in \D.$ 

We introduce this field primarily for notational convenience.  In reality we are only interested in finite dimensional marginals of $\WN.$ To this end, for any finite collection of points $\mathbf{z} \subset \D,$ we let $\BSA(\mathbf{z}) \subset \D^\R$ be the $\sigma$-algebra in the power set of $\R^\D$ generated by the cylinder sets over $\mathbf{z}$, i.e.\,generated by
\[
  \left\{\left\{ f \in \R^\D~:~f(z) \in B_z~\forall z \in \mathbf{z} \right\}, \{B_z\} \text{ Borel }  \right\}.
\]
In short, $\BSA(\mathbf{z})$ are the functions depending solely on a field at points $\mathbf{z}.$

Our first approximation in this vein is the following.
\begin{proposition}
  \label{prop:clt1}
  For any $K > 0,$
  there are constants $m > 1$ and $C>0$ so that the following hold.
  For any positive integer $d \leq n - m\log n$ 
  and any $F \in \BSA(\left\{ \zeta_1, \zeta_2, \dots, \zeta_d \right\})$
  with $\|F\|_\infty \leq 1$
  \begin{align*}
    \biggl|
      \frac{\Exp \left[ F(\CUEF + \WN) e^{2\CUEF(\zeta_n)} \right]}
      {\Exp \left[ e^{2\CUEF(\zeta_n)} \right]}
      - \frac{\Exp \left[ F(\GF+\WN) e^{2\GF(\zeta_n)}\right]}
      {
          \Exp \left[ e^{2\GF(\zeta_n)}\right]
      } 
      \biggr| < C(\log N)^{-K}.
  \end{align*}
\end{proposition}
\noindent Thus, we show a smoothed total variation approximation of 
$\CUEF(\zeta_i)_1^d$ and $\GF(\zeta_i)_1^d$ under the desired biasing term.  
Proposition \ref{prop:clt1} is a special case of Proposition  \ref{prop:clt} below.

 The estimate in Proposition \ref{prop:clt1}
is sufficient for showing the upper bound in \eqref{eq:Evrough}.  It is however insufficient for attaining the lower bound, on account of only giving information on the initial $n - m\log n$ steps.  Indeed, we would effectively like to take $m$ arbitrarily small to attain the desired bound, but the method we use for normal approximation encounters a natural technical barrier when $m$ is below $1.$

On the other hand, we are able to compute with high precision similar statistics of $\CUEF$ under substantially more elaborate biasing terms, provided they have a specific algebraic form.  To this end, let $\mathbf{y},\mathbf{z} \subset \mathbb{D}$ be finite subsets, and define for a field $\F,$
\begin{equation}
    \label{eq:bias}
  \Bias{\F} = \sum_{z \in \mathbf{z}} 2 \F(z) - \sum_{y \in \mathbf{y}} 2 \F(y).
\end{equation}
We will prove a Gaussian approximation for $\CUEF$ biased by the exponential of $\Bias{\CUEF}.$  Further, we will show that
the approximation holds simultaneously on finite collections of rays
$\left\{ \zeta_i \omega_j \right\}_{\substack{i=1\dots n \\ j=1\dots 
  {
  s
}
}}.$

\begin{proposition}
  \label{prop:clt}
  For any $K >0$ and any $s \in \mathbb{N},$
  there are constants $m > 
  {
    100
  }
  $ and $C>0$ so that the following hold.
  Let $\mathbf{y},\mathbf{z} \subset \mathbb{D}\cdot (1-N^{-1})$ be finite subsets with $|\mathbf{y}| \leq |\mathbf{z}|
  {
    \leq k
  }
  .$  Suppose that there is a Euclidean ball of radius $N^{-1}(\log N)^m$ that contains all points of $\mathbf{z}$ with hyperbolic distance from $0$ greater than $n-m\log n.$  Further suppose the pairwise hyperbolic separation between points of $\mathbf{z}$ is at least $N^{-1}.$  
  Define
  \[
      \Delta
    = \max_{z \in \mathbf{z}} e^{-N\exp(-\dH(0,z))}\prod_{
      \substack{w \in \mathbf{z} \\ 
    w \neq z}} 
    \coth(\dH(w,z)/2).
  \]
  For any positive integer $d \leq n - m\log n,$ any
  $\left\{ \omega_j \right\}_{1}^s \subset \mathbb{T},$
  and any 
  \[
    F \in \BSA(\left\{ \zeta_i\omega_j~:~ 1 \leq i \leq d, 1 \leq j \leq r \right\})
  \]
  with $\|F\|_\infty \leq 1,$
  it follows that
    {
 \[  \left|   \frac{\Exp \left[F(\CUEF + \WN) e^{\Bias{\CUEF}}\right]}
    {\Exp e^{\Bias{\CUEF}}}
    - \frac{\Exp \left[F(\GF + \WN) e^{\Bias{\GF}}\right]}
    {\Exp e^{\Bias{\GF}}}\right|\leq C
     (1+\Delta)^{2k}
    (\log N)^{-K}
    \frac
    {\Exp e^{\Bias{\GF}}}
    {\Exp e^{\Bias{\CUEF}}}.
  \]
}
\end{proposition}
In this formulation, Proposition~\ref{prop:clt1} is a special case of 
\corO{Proposition 
  \ref{prop:clt}, as can be seen by taking $\mathbf{y}=\emptyset$,
  $\mathbf{z}=\{\zeta_n\}$; Note that from
  \eqref{eq-zetadef} and
  \eqref{eq-defn} 
   it follows that $|\zeta_n|\leq (1-N^{-1})$, and the other
  assumptions in Proposition \ref{prop:clt} are immediate in the setup of
  Proposition \ref{prop:clt1}.
}
   This Gaussian approximation will be sufficient to produce estimates such as \eqref{eq:Evrough} (see Section~\ref{sec:fmcintro}), after slightly modifying the event in question.  In our application, we will only use this approximation with $
  s
=1,2.$ 

The effect of biasing a jointly Gaussian vector by a linear functional of that vector is to change the mean of the Gaussian vector.  We will use the following identity repeatedly: 
\begin{lemma}
  Let $\mathbf{z},\mathbf{y} \subseteq
  \mathbf{w}$ be finite subsets of $\D,$ and let $F$ be in $\BSA(\mathbf{w})$ and $\Bias$ be as in \eqref{eq:bias}.  Let 
  \[
    \mu(\zeta) = 
    \sum_{z \in \mathbf{z}} 2\Exp \GF(z)\GF(\zeta)
    -\sum_{y \in \mathbf{y}} 2\Exp \GF(y)\GF(\zeta)\,,
    \quad
      \zeta\in \D.
  \]   
  Let $\lambda \in \R$ be arbitrary, then
  \[
    \frac{\Exp\left[ F(\GF)e^{\lambda \Bias{\GF}} \right]}{\Exp\left[e^{\lambda \Bias{\GF}} \right]}
    =
    \Exp\left[F(\GF+\lambda \mu)\right].
  \]
  \label{lem:means}
\end{lemma}
\begin{proof}
    The vector $\left( \GF(w) \right)_{w \in \mathbf{w}}$ is jointly Gaussian.  Enumerating the points of $\mathbf{w},$ we can write this vector as $\left( G_i \right)_1^k$ with $k=|\mathbf{w}|.$  Let $\Sigma$ be the covariance matrix of this vector.  We can find a vector $v \in \R^k$ representing $\corO{\lambda} \Bias$ in that
    \[
        \Exp\left[ F(\GF)e^{\lambda \Bias{\GF}} \right]
        =
        \int_{\R^k}
        \frac{F(x)e^{v^tx}}{\sqrt{(2\pi)^k |\det \Sigma|}}
        e^{\tfrac{-x^t \Sigma^{-1} x}{2}}\,dx.
    \]
    We can now set $u = \Sigma v$ and complete the square, to get
    \[
        \frac{
            \Exp\left[ F(\GF)e^{\lambda \Bias{\GF}} \right]
        }
        {
            e^{\tfrac{u^t \Sigma^{-1}u}{2}}
        }
        =
        \int_{\R^k}
        \frac{F(x)}{\sqrt{(2\pi)^k |\det \Sigma|}}
        e^{\tfrac{-(x-u)^t \Sigma^{-1} (x-u)}{2}}\,dx.
    \]
    The constant $e^{\tfrac{u^t \Sigma^{-1}u}{2}} = \Exp\left[ e^{\lambda \Bias{\GF}} \right],$ as can be verified by setting $F \equiv 1$ and changing variables in the integral.  Moreover, the right hand side is exactly the claimed expression in the lemma.

\end{proof}
Hence, biasing by twice the endpoint of the ray effectively removes the drift from the events we wish to consider:
\begin{corollary}
  \label{cor:means}
  After biasing the law of $\GF$ by $\frac{e^{2\GF(\zeta_d)}}{\Exp e^{2\GF(\zeta_d)}},$ we have that $(\GF(\zeta_i))_1^d$ have means $i + O(1).$
\end{corollary}
\subsection{Field moment calculus}
\label{sec:fmcintro}

In light of Proposition~\ref{prop:clt1}, we slightly modify the event $\EvI{}$ to reflect that the Gaussian approximation only holds for the initial $n - m\log n$ steps of $\left\{\CUEF(\zeta_j)\right\}_1^n.$  Closer to the events we finally consider, we redefine
\begin{equation*}
  \EvI{} = \left\{ 
      \begin{aligned}
          \F(\zeta_i) < i + C\log n,~\forall~1 \leq i \leq n - m\log n, \text{ and } \\
      n+ t - (\log n)^{1-\delta} < \F(\zeta_n) < n + t 
      +(\log n)^{1-\delta} 
  \end{aligned}
  \right\}.
\end{equation*}

Using Proposition~\ref{prop:clt}, we can tilt the measure of $\CUEF$ so that $\CUEF(\zeta_i)$ has mean approximately $i.$  Let $\FMC$ be the expectation under this tilted measure, which has Radon-Nikodym derivative proportional to $e^{2\CUEF(\zeta_n)}.$  The portion of $\EvI{}$ related to the first $n - m\log n$ steps of $\left\{ \CUEF(\zeta_i) \right\}_1^n$ can be estimated using the Gaussian approximation.  However, it still remains to say that with high probability, the final step $\F(\zeta_n)$ indeed lands in the window specified.  

We expect the increment of the final step 
\(
W = \CUEF(\zeta_n) - \CUEF(\zeta_{n-m\log n})
\)
to be roughly Gaussian with variance order $\log n.$
After conditioning the value of the random variable
$\CUEF(\zeta_{n-m\log n})$ to be around $n-m\log n -\tfrac 34 \log n,$ asking the final step $\CUEF(\zeta_n)$ to be roughly $n-\tfrac34\log n$ is typical under $\FMC.$  Moreover, to quantify this typicality, it would be enough to compute a second moment of $W$ conditional on 
\[
    \CUEF \in \mathcal{B} =\left\{ \F: \F(\zeta_i) < i + C\log n, \forall~1 \leq i \leq n-m\log n\right\}.
\]

Proposition~\ref{prop:clt} nearly provides this machinery.  Instead of allowing the computation of a field moment, that of $W,$ it provides a way to compute exponential moments of the final steps of $\left\{\CUEF(\zeta_i) \right\}$ conditional on $\CUEF \in \mathcal{B}.$
%

For the purpose of estimating something like the second moment of $W,$ a direct comparison to an exponential moment turns out to be useless.  However, we bypass this problem by exploiting the smoothness of the field: namely for $z,y$ with $\dH(z,y) \leq 1,$ $\Var\left( \CUEF(z)-\CUEF(y) \right) \asymp \dH(z,y)^2.$  Hence, for an increment $X = \CUEF(z)-\CUEF(y),$
\[
    \FMC( X ) \leq \FMC(2^{-1}(e^{2X} - 1))
    \approx 2^{-1}\left( e^{2\mu + 2\sigma^2} - 1 \right)
    \approx \mu(1+o(1)),
\]
if we use that $\mu \asymp \dH(z,y)$ and $\sigma^2 \asymp \dH(z,y)^2$ for $z,y \in \D$ with $\dH(z,y) \leq 1.$

    Thus by decomposing $W$ into a sum of many microscopic increments, and estimating using exponential moments, we can estimate the first moment of $W,$ conditioned on $\CUEF \in \mathcal{B}.$  By altering the bias in $\FMC,$ it is also possible to estimate a second moment of $W,$ or in principle higher moments.  The details are provided in Section \ref{sec:fmc} below.

\subsection{Organization {and notation}}
\label{sec:org}

We arrange the paper as follows.  In Section~\ref{sec:main} we give the proof of the main theorem, using the normal approximations and assuming certain calculations about the Gaussian field, such as the barrier estimate.  In Section~\ref{sec:Baxter}, we prove a generalization of Baxter's Toeplitz determinant identities, which we use to estimate the characteristic functions of $\CUEF$ under exponential biases.  In Section~\ref{sec:tv}, we use these identities to prove Proposition~\ref{prop:clt}. {In the appendix, we collect some barrier estimates that are used throughout the paper.}

We use the following notation.  The symbol $a(\dots) \ll b(\dots)$ should be read as meaning ``there is an absolute constant $C>0$ so that $a(\dots) \leq Cb(\dots).$''  If used as a hypothesis, it should be read as ``for any fixed absolute constant $C>0$ so that $a(\dots) \leq Cb(\dots), \dots$''  We use the symbol $\asymp$ to mean $\ll$ and $\gg.$
Besides this, we use the usual $o,O,\omega,\Omega,\Theta$ notation as well.

For $z \in \C \setminus (-\infty,0],$ we always take $\arg z$ to be the principal branch of the argument.  On the cut $(-\infty, 0),$ we let $\arg z = \pi.$
\section{Estimation of the maximum of the CUE field}
\label{sec:main}
\subsection{Upper bound preliminaries}
\label{sec:1p1}
Let $m$ be as in Proposition~\ref{prop:clt}, with 
  $s=2$
and $K=100.$  
Let $n_0 = \lfloor n - m\log n \rfloor.$
Because the normal approximation Proposition~\ref{prop:clt} only holds for $\CUEF(z)$ for $\dH(z,0) \leq n_0$ we redesign the event in \eqref{eq:Ev0} appropriately.  In particular we will relax the requirement that all $\CUEF(\zeta_i ) < i + C_0 \log n$ for $i > n_0.$  

Let $\BF : \N \to \R;$ $\nu(i)$ will represent the barrier below which we constrain the field along rays.  Roughly, it will be $i + \Theta(\log n).$  In the mesoscopic regime, which we consider as $\CUEF(\zeta_i)$ for $i \leq n_0$ and where the Gaussian approximation holds, the events we consider are essentially the same as those one would consider in the setup of branching random walk.  For $\omega \in \T$  and $t_0 \in \R,$ define the sets of fields $\BS$ and $\ES$ by
\begin{align*}
  \BS[\BF][\omega] &= \left\{ \F(\omega\zeta_i) <\BF(i),~\forall~1\leq i \leq n_0
  \right\} \text{ and } \\
  \ES{t_0}[\omega] &= \left\{ n_0 + t_0 - 1 < \F(\omega\zeta_{n_0}) < n_0 + t_0 \right\}.
\end{align*}
These represent mesoscopic barrier and endpoint events, respectively.
We will always choose $t_0=t_0(n)$ to be at most poly-logarithmic in $n$.

For the upper bound on the maximum of $\CUEF,$ on the microscopic scale, which we consider as $\CUEF(\zeta_i)$ for $n_0 < i \leq n$, we will not put a barrier constraint.  In the case of branching random walk, such a setup incurs a $\log\log n$ error term in the maximum, and this $\log\log n$ term will appear here too.  

We must also control \corO{the}
behavior of the field on even smaller scales to ensure that they do not contribute anything to the maximum.  
Hence, let $\rho : \D \to \R$ be a real function, which we use for the purposes of estimating submicroscopic fluctuations --- variations between points in the field which are at hyperbolic distance $o(1)$. 
{
  (We define $\rho$ explicitly in \eqref{eq:defrho} and \eqref{eq-rho-defb}
  below.)
}
For any $t \in \R$ define the subset of fields $\EL$ by
\begin{align}
  \EL{\rho,t}[\omega] &= \left\{ \F(\omega\zeta_n) > n + t + \rho(\omega\zeta_n) \right\}. \nonumber \\
  \intertext{In terms of these events, with $n_1= 2^{\lceil \log n \rceil}$ and $\N_0 = \N \cup \left\{ 0 \right\},$ define}
  \EvU{\rho,t}[\omega] &=  \bigcup_{\substack{ h \in \mathbb{N}_0 \\ h < n_1} } \EL[\omega e^{2\pi i h(Nn_1)^{-1}}].
  \label{eq:EvU}
\end{align}
When $\omega=1,$ we drop it from the notation, so that $\EvU{\BF} = \EvU{\BF}[1]$ and likewise for 
\(
\BS,
\ES,
\) {and}
\(
\EL.
\)

The upper bound will then proceed by using a first moment estimate to show that $\CUEF \not\in \EvU[\omega]$ for $\omega \in \T$ ranging over a grid of cardinality $N$, 
  and $t=-\tfrac34 \log n+o(\log n)$.  We estimate $\Pr(\CUEF \in \EvU)$ by a combination of Proposition~\ref{prop:clt} to handle $\BS \cap \ES$ and carefully constructed exponential bias terms.  These exponential bias terms will alter the means of the field on the mesoscopic scale.  In one way, this is desired, as we wish to effectively remove the linear drift from the events $\BS \cap \ES.$  However, some of the biasing terms are included principally to detect microscopic variations. We must show that these have a negligible effect on the probability of $\BS \cap \ES,$ which is the content of the following lemma.
\begin{lemma}
  Let $m$ and $K$
  be as chosen at the start of Section~\ref{sec:main}.
  Let $\omega_1,\omega_2,\omega_3 \in \T$ have argument less than $2\pi N^{-1}.$
  Let $\Bias{\F}$ be given by
  \[
    \Bias{\F}
    =
     2\F(\omega_1 \zeta_{n-1})
    +2\F(\omega_2 \zeta_{n})
    -2\F(\omega_3 \zeta_{n}).
  \]
  For any $F \in \BSA(\left\{ \zeta_1, \zeta_2, \dots, \zeta_{n_0} \right\})$ with $\|F\|_\infty \leq 1,$
  {
    $$
    \left|\Exp \left[F(\CUEF + \WN) e^{\Bias{\CUEF}} \right]
    - \Exp \left[ F(\GF + \WN + \mu)\right] \Exp[ e^{\Bias{\CUEF}} ]\right|
    \ll
    (\log N)^{-K}\Exp[ e^{\Bias{\GF}}],
    $$
  }
  where $\mu(z) = 2\Exp[\GF(z)\GF(\zeta_{n-1})],$ 
  and the last estimate is uniform in the choice of such $F$.
  \label{lem:Lipschitz}
\end{lemma}
\begin{proof}
  It is elementary to check, \corO{see \eqref{eq-zetadef} and  
  Lemma \ref{lem:branch}, } that
  \[
    \dH(\omega_1\zeta_{n-1},\omega_2\zeta_n) \asymp 1.
  \]
  Hence, Proposition~\ref{prop:clt} applies, and we conclude from it and Lemma~\ref{lem:means} that 
  \[
    \Exp \left[F(\CUEF + \WN) e^{\Bias{\CUEF}} \right]
    = \Exp \left[ F(\GF + \WN + \mu')\right] \Exp[ e^{\Bias{\CUEF}} ]
    + \xi'  \Exp[ e^{\Bias{\GF}} ]
  \]
  where $|\xi'| \ll (\log N)^{-K},$ and where by Lemma~\ref{lem:means}
  \[
    \mu'(z)
    =-\log|1-z\overline{\omega_1} \zeta_{n-1}|
    -\log|1-z\overline{\omega_2} \zeta_{n}|
    +\log|1-z\overline{\omega_3} \zeta_{n}|.
  \]
  We will use that the map
    \[
        \left( \F(\zeta_1),\F(\zeta_2),\dots,\F(\zeta_{n_0}) \right) \mapsto \Exp\left[ F(\F+\WN)\right]
    \]
    is Lipschitz to make the comparison.  Specifically, the function is Lipschitz in each coordinate with constant $O(1),$ 
    \corO{since the addition of the 
    Gaussian random variables $\WN$ has the effect of convolving $F$ 
  with the Gaussian distribution}. 

    The difference $\mu' - \mu$ at any $z$ can be written as
    \[
    \mu'(z)
    -\mu(z)
    =\log \left|
    \frac{
    1-z\zeta_{n-1}
    }{1-z\overline{\omega_1} \zeta_{n-1}}
    \right|
    +\log \left|
    \frac{1-z\overline{\omega_3} \zeta_{n}}
    {1-z\overline{\omega_2} \zeta_{n}}
    \right|.
    \]
    On any fixed compact subset of $\C$ disjoint from $-1,$ 
    we have $\left|\log|1+x|\right| \ll |x|.$  
    Hence, we have that uniformly in $1 \leq i \leq n_0,$
    \[
      |\mu'(\zeta_i)-\mu(\zeta_i)|
      \ll \left|
      \frac{1 - \overline{\omega_1}}
      {1-\zeta_i}
      \right|
      +\left|
      \frac{\overline{\omega_2} - \overline{\omega_3}}
      {1-\zeta_i}
      \right|.
    \]
    As $1 - \zeta_i \asymp e^{-i},$ we conclude
    \begin{align*}
      \bigl|
      \Exp \left[ F(\GF + \WN + \mu')\right] 
      -\Exp \left[ F(\GF + \WN + \mu)\right] 
      \bigr|
      \ll
      N^{-1} \sum_{i=1}^{n_0} e^i \ll (\log N)^{-m}.
    \end{align*}
    As we chose $m > K,$ the proof is complete.
\end{proof}



Using Lemma \ref{lem:Lipschitz}
and Proposition~\ref{prop:clt}, we have effectively shown independence between microscopic fluctuations and the mesoscopic behavior of the process.  For the mesoscopic portion of the process, we will estimate the probability of the Gaussian field being in $\BS \cap \ES$ by applying 
\corO{the results in the appendix}
to compare with a Gaussian random walk.
For the Gaussian process, the following bound suffices.
\begin{lemma}
  Let $j \in \mathbb{N}.$
  Let $F = \one[{\BS \cap \ES}],$ and set $\mu : \D \to \R$ to be $\mu(z)=2\Exp[ \GF(z)\GF(\zeta_j)].$
  Suppose $\BF : \N \to \R$ satisfies 
  {$|\BF(i) -i |\ll \log n_0,$}
  uniformly in $1 \leq i \leq n_0.$  Then uniformly in $t_0$ satisfying 
  \(
  \nu(n_0) - (\log n_0)^2 \leq n_0 + t_0 \leq \nu(n_0),
  \)
  and uniformly in $n_0 \leq j \leq n,$
  \[
    \Exp \left[ F(\GF + \WN + \mu)\right]
    \ll \frac{(\log n_0)^3}{n_0^{3/2}}.
  \]
  \label{lem:ballotub}
\end{lemma}
\begin{proof}
	{
	Set $G(i)=\sqrt{2}(\GF(\zeta_i)+
	\corO{\WN(\zeta_i)).}$
	Set $h_i=\nu(i)-
	\corO{\mu(\zeta_i)}$,
	and note that 
	since by Lemma 
	\corO{\ref{lem:branch}} we have
	that $|\mu(\zeta_i)-i|\leq C$, it holds that
	$|h_i|\ll \log n_0$. Set also $t=t_0-
	\corO{\mu(\zeta_{n_0})}$ and note
	that $-C-(\log n_0)^2\leq t\leq C$.
	From the definitions and using the notation of 
	Appendix \ref{app-barrier1}, we have that 
	\begin{equation}
		\label{eq-Sat1}
	  \Exp \left[ F(\GF + \WN + \mu)\right]\leq p_{B,G}(n_0,t,h).
  \end{equation}
	  The conclusion follows from Corollary
	  \ref{cor-UBbarrier}.}
\end{proof}

\subsection{Upper bound for the maximum of $\CUEF$}
\label{sec:ub}

We will now fix the barrier function we use for the upper bound:
\[
  \BF(i) = i + 2\log n.
\]
The proof of the upper bound for $\EvU$ uses a dyadic chaining argument.
\corO{Recall \eqref{eq-defn}.}
Underlying the proof is the following lemma. 
\begin{lemma}
  Let $F = \one[{\BS[\BF] \cap \ES}],$ and let
  $G = \one[{\EL[\omega_1] \cap \EL[\omega_2]^c}].$
  For all $\omega_1,\omega_2 \in \T$ with argument at most $2\pi N^{-1}$ in absolute value, uniformly in 
  \(
  t_0 \in [\nu(n_0) - n_0 - (\log n_0)^2,  \nu(n_0) - n_0],
  \)
  $t > -\log n,$ and $\rho(\omega_1\zeta_n) - \rho(\omega_2\zeta_n) > (\log N)^{4-K},$
  \[
    \Exp\left[ 
    F(\CUEF+\WN)G(\CUEF)
    \right]
    \ll \frac{(\log n)^3e^{-n-2t-2\rho(\omega_1\zeta_n)}|\omega_1-\omega_2|^2N^2}{n^{3/2}(\rho(\omega_1\zeta_n) - \rho(\omega_2\zeta_n))^{2}}.
  \]
  \label{lem:dyadic}
\end{lemma}
\begin{proof}
  When $\CUEF \in \EL[\omega_1],$ the value of $\CUEF(\omega_1\zeta_{n})$ is at least $n+t+\rho(\omega_1\zeta_n).$  Using \eqref{eq:radialslide}, it follows there is a constant 
	\corO{C} sufficiently large that
  \begin{align}
    &\CUEF(\omega_1\zeta_{n}) \leq \CUEF(\omega_1\zeta_{n-1}) + C. \nonumber\\
    \intertext{Hence, we have that almost surely}
    \label{eq:dy1}
    &G(\CUEF) \ll 
		\corO{e^{2\CUEF(\omega_1\zeta_{n-1}) - 2(n + t+\rho(\omega_1\zeta_n))}}.\\
    \intertext{We also have that when $\CUEF \in \EL[\omega_1] \cap \EL[\omega_2]^c,$}
    &\CUEF(\omega_1\zeta_n) - \CUEF(\omega_2\zeta_n) > \rho(\omega_1\zeta_n) - \rho(\omega_2\zeta_n).\nonumber\\
    \intertext{Let $\phi(x) = \cosh(2x) - 1
      \corO{\geq x^2},$ which is increasing in \corO{$x>0.$} Letting} 
    &J(\CUEF) = e^{2\CUEF(\omega_1 \zeta_{n-1})}\phi(\CUEF(\omega_1\zeta_n) - \CUEF(\omega_2\zeta_n)), \nonumber\\
    \intertext{and combining these observations, we have}
    \label{eq:dy2}
    &\Exp[F(\CUEF+\WN)G(\CUEF)] \ll 
    \Exp \left[ F(\CUEF+\WN)\frac{e^{- 2(n + t+\rho(\omega_1\zeta_n))}
J(\CUEF)}{\phi(\rho(\omega_1\zeta_n) - \rho(\omega_2\zeta_n))} \right].
\end{align}

The expression $J(\CUEF)$ is a linear combination of three exponential bias terms:
\begin{equation*}
  \begin{aligned}
    J(\CUEF)
    &= 
    \frac{1}{2}e^{\Bias{\CUEF}[1]}
    +\frac{1}{2}e^{\Bias{\CUEF}[2]}
    -e^{\Bias{\CUEF}[3]}, \text{ where }\\
    \Bias{\CUEF}[1] &= 
    2\CUEF(\omega_1\zeta_{n-1})
    +2\CUEF(\omega_1\zeta_{n})
    -2\CUEF(\omega_2\zeta_{n}), \\
    \Bias{\CUEF}[2] &= 
    2\CUEF(\omega_1\zeta_{n-1})
    -2\CUEF(\omega_1\zeta_{n})
    +2\CUEF(\omega_2\zeta_{n}), \\
    \Bias{\CUEF}[3] &= 
    2\CUEF(\omega_1\zeta_{n-1}), \\
  \end{aligned}
\end{equation*}
all of 
which satisfy the hypotheses of Lemma~\ref{lem:Lipschitz}.  Therefore, we have that  
\begin{equation}
  \label{eq:dy4}
  \begin{aligned}
    \Exp[F(\CUEF +\WN)J(\CUEF)]
    &= \Exp[F(\GF + 
		\corO{\WN}
		+ \mu)]\Exp[J(\CUEF)]+\xi,
  \end{aligned}
\end{equation}
with
\[
  |\xi| \ll (\log N)^{-K}\Exp[ e^{2\CUEF(\omega_1 \zeta_{n-1})}(2+\phi(\CUEF(\omega_1\zeta_n) - \CUEF(\omega_2\zeta_n)))].
\]

Estimating $\Exp[J(\CUEF)]$ can be done using monotonicity (Proposition~\ref{prop:domination}).  We show 
{
below that
}
\begin{equation}
    \Exp[J(\CUEF)] \ll
    \Exp[ e^{2 \GF(\omega_1\zeta_{n-1})}]
    \cdot |\omega_1 - \omega_2|^2 \cdot N^2.
  \label{eq:J}
\end{equation}
This will yield the lemma, for 
combining \eqref{eq:J}, Lemma \ref{lem:ballotub},
\eqref{eq:dy4} and \eqref{eq:dy2}, we have that for all $t > -\log n,$ and $\rho$ so that $|\rho(\omega_1\zeta_n) - \rho(\omega_2\zeta_n)|^2 > (\log N)^{-K + 3/2},$ 
\[
  \Exp\left[ 
    F(\CUEF+\WN)G(\CUEF)
  \right]
  \ll \frac{(\log n)^3e^{-n-2t-2\rho(\omega_1\zeta_n)}|\omega_1-\omega_2|^2 N^2}{n^{3/2}|\rho(\omega_1\zeta_n) - \rho(\omega_2\zeta_n)|^2}.
\]

It remains to prove \eqref{eq:J}.  
Let $W= \CUEF(\omega_1\zeta_n) - \CUEF(\omega_2\zeta_n).$ Using Proposition~\ref{prop:domination}, we have 
\corO{that} 
for any $\lambda \in \R,$
\[
    \Exp\left[e^{2\CUEF(\omega_1 \zeta_{n-1})}e^{\lambda W}\right]
    \leq
    \Exp\left[e^{2\GF(\omega_1 \zeta_{n-1})}e^{\lambda(\GF(\omega_1\zeta_n) - \GF(\omega_2\zeta_n))}\right].
\]
Set  $\mu = 2\Exp\left[\GF(\omega_1 \zeta_{n-1})(\GF(\omega_1\zeta_n) - \GF(\omega_2\zeta_n))\right]$ and set 
    $\sigma^2 = \Exp\left(\GF(\omega_1\zeta_n) - \GF(\omega_2\zeta_n)\right)^2.$
Applying the usual combination of Markov's inequality and optimizing over $\lambda$, for any $x \geq 0,$
\[
    \Exp\left[e^{2\CUEF(\omega_1 \zeta_{n-1})}
        \one[ |W-\mu| \geq x]
    \right]
    \leq
    2\Exp\left[e^{2\GF(\omega_1 \zeta_{n-1})}\right]
    e^{-\tfrac{x^2}{2\sigma^2}}.
\]

Using $\phi(x) \ll x^2e^{2x}$ and H\"older's inequality, we can estimate
\[
   \Exp[J(\CUEF)] \ll
   \bigl(\Exp[ e^{2 \CUEF(\omega_1\zeta_{n-1})} W^4]\bigr)^{1/2}
   \bigl(\Exp[ e^{2 \CUEF(\omega_1\zeta_{n-1}) + 4W} ]\bigr)^{1/2}.
\]
Hence provided that $|\mu|, \sigma \ll 1,$ we have that
\[
    \Exp[J(\CUEF)] \ll
    \Exp[ e^{2 \GF(\omega_1\zeta_{n-1})}]
    (\mu^2 + \sigma^2),
\]
by expressing the moment as an integral over level sets of $W$ and applying monotonicity to the second expectation.

There only remains to estimate $\mu$ and $\sigma.$  For this purpose we use repeatedly that $\dH(\omega_1\zeta_n,\omega_2\zeta_n) \ll N |\omega_1-\omega_2|,$ which is elementary to verify.  For $\sigma,$ we have by \eqref{eq:covcosh}
\[
    \sigma^2
    = 
    \log\left( \cosh\left( \tfrac{\dH(\omega_1\zeta_n,\omega_2\zeta_n)}{2} \right) \right) \ll N^2 |\omega_1-\omega_2|^2.
\]
For $\mu,$ we have
\[    \mu=
    \log\left( 
    \tfrac
    { \cosh( {\dH(0,\omega_1\zeta_{n})}{2}^{-1} )\cosh( {\dH(\omega_1\zeta_{n-1},\omega_2\zeta_n)}2^{-1})}
    {\cosh( {\dH(\omega_1\zeta_{n-1},\omega_1\zeta_n)} 2^{-1})
    \cosh( {\dH(0,\omega_2\zeta_{n})}{2}^{-1} )}
    \right) 
    =
    \log\left( 
    \tfrac
    {\cosh( {\dH(\omega_1\zeta_{n-1},\omega_2\zeta_n)} 2^{-1})}
    {\cosh( {\dH(\omega_1\zeta_{n-1},\omega_1\zeta_n)} 2^{-1})}
    \right).
    \]
Now, using that 
\[
    |
    \dH(\omega_1\zeta_{n-1},\omega_2\zeta_n)
    -
    \dH(\omega_1\zeta_{n-1},\omega_1\zeta_n)
    |
    \leq \dH(\omega_2\zeta_n, \omega_1\zeta_n) \ll  N |\omega_1-\omega_2|,
\]
we have $|\mu| \ll N|\omega_1-\omega_2|,$ which completes the proof 
  of 
\eqref{eq:J} 
and hence of the lemma.
\end{proof}

We now proceed to give the proof of the upper bound using this machinery.  In light of Lemma~\ref{lem:dyadic}, we will now define $\rho.$  In fact, we only need to specify $\rho$ for the points 
\[
  \left\{ e^{2\pi i h(Nn_1)^{-1}}\zeta_n ~:~h \in \mathbb{N}_0, h \leq n_1 \right\}.
\]
Define $S_{0} = \{1\}$ and define for $j \in \mathbb{N},$
\[
  S_j = \left\{ e^{2\pi ih2^{-j}N^{-1}}~:~h \in (2\mathbb{N}_0+1), h < 2^j \right\}.
\]
For every point $\omega \zeta_n \in \D$ with $\omega \in S_j,$ define 
\begin{equation}
  \label{eq:defrho}
  \rho(\omega \zeta_n) = 1 + (1.1)^{-1} + \dots + (1.1)^{-j}
\end{equation}
(any constant \corO{in $(1,\sqrt{2})$}
would work in place of $1.1$).  Note that this makes $\rho$ uniformly bounded by $11.$

\begin{lemma}
  Uniformly in $t > -\log n,$
  \[
    \Exp[ \one[{\BS}](\CUEF+\WN) \one[{\EvU}](\CUEF) ] \ll
    \frac{e^{-n-2t}(\log n)^{
      5
  }}{n^{3/2}}.
  \]
  \label{lem:1pUB}
\end{lemma}
\begin{proof}
  There are two components to the argument.  First, for integer $t_0$ with
  \(
  -(\log n_0)^2 \leq t_0 \leq 0
  \),
  we use the dyadic chaining argument to estimate 
  \[
    \Exp[  \one[{\BS \cap \ES}](\CUEF+\WN)\one[{\EvU}](\CUEF)].
  \]
  This we then sum over all such integral $t_0.$  Second, we consider the set of fields
  \[
    \ESU = \left\{ \F(\zeta_{n_0}) < n_0 - \log(n_0)^2 + 1 \right\}.
  \]
  Then, we separately estimate
  \[
    \Exp[  \one[{\BS \cap \ESU}](\CUEF+\WN)\one[{\EvU}](\CUEF)],
  \]
  on which the last increment $\CUEF(\zeta_n) - \CUEF(\zeta_{n_0})$ must be abnormally large.

  For the first part of the argument, fix $t_0$ as described.  For any $\omega \in S_j$ for some $j > 0,$ let $g(\omega)$ be the element of $S_{j-1}$ with maximal argument, which does not exceed the argument of $\omega.$  Then for $\omega \in S_j,$ we have $|\omega - g(\omega)| \ll N^{-1}2^{-j}$ and $\rho(\omega) - \rho(g(\omega)) = (1.1)^{-j},$
  Let $F = \one[{\BS[\BF] \cap \ES}],$ and let
  $G_{\omega_1,\omega_2} = \one[{\EL[\omega_1] \cap \EL[\omega_2]^c}].$
  Then, we can estimate
  \begin{align*}
    \Exp&\left[ F(\CUEF + \WN)\one[{\EvU}](\CUEF)\right] \\
    &\leq
    \Exp\left[ F(\CUEF + \WN)\one[{\EL}](\CUEF)\right]
    +\sum_{j=1}^{\log_2 n_1}
    \sum_{\omega \in S_j}
    \Exp\left[ F(\CUEF + \WN)G_{\omega,g(\omega)}(\CUEF)\right]
    \\
    &\ll
    \Exp\left[ F(\CUEF + \WN)\one[{\EL}](\CUEF)\right]
    +\sum_{j=1}^{\log_2 n_1}
    \sum_{\omega \in S_j}
    \frac{(\log n)^3e^{-n-2t}4^{-j}}{n^{3/2}(1.1)^{-2j}} \\
    &\ll
    \Exp\left[ F(\CUEF + \WN)\one[{\EL}](\CUEF)\right]
    +
    \frac{(\log n)^3e^{-n-2t}}{n^{3/2}}.
  \end{align*}

  Estimating 
  $\Exp\left[ F(\CUEF + \WN)\one[{\EL}](\CUEF)\right]$
  follows the same outline as Lemma~\ref{lem:dyadic}, albeit simpler.  On the one hand, we have
  \begin{align}
    &\Exp\left[ F(\CUEF + \WN)\one[{\EL}](\CUEF)\right]
    \ll
    \Exp\left[ F(\CUEF + \WN)\one[{\EL}](\CUEF)
    e^{2\CUEF(\zeta_{n}) - 2(n + t)}
  \right]. \nonumber \\
  \intertext{\corO{On the other hand,}
	we have}
  &\Exp\left[ 
    F(\CUEF + \WN)\one[{\EL}](\CUEF)
  e^{2\CUEF(\zeta_{n})}
\right]
\ll
\Exp\left[ F(\CUEF+\WN) 
  e^{2\CUEF(\zeta_{n})}
\right].\nonumber \\
\intertext{Thus by \corO{Lemma~\ref{lem:Lipschitz} and}
Lemma~\ref{lem:ballotub}, we conclude}
&\Exp\left[ F(\CUEF + \WN)\one[{\EL}](\CUEF)\right]
\ll \frac{(\log n)^3e^{-n-2t}}{n^{3/2}}.
\label{eq:1pUB1} \\
\intertext{Hence, summing over all $t_0$ in $[-(\log n_0)^2,0]$, we get}
&\Exp\left[ 
\corO{\one[{\BS}](\CUEF+\WN)\one[{\EvU}](\CUEF) }
(1-\one[{\ESU}](\CUEF))\right]
\ll \frac{(\log n)^5e^{-n-2t}}{n^{3/2}}.
\label{eq:1pUB2}
\end{align}

To estimate 
$
\Exp\left[ 
\corO{\one[{\BS}](\CUEF+\WN)\one[{\EvU}](\CUEF) }\one[{\ESU}](\CUEF)\right],
$
we just note that on the event $\ESU \cap \EL[\omega],$ we have both 
\[
  \CUEF(\omega \zeta_n) - \CUEF(\zeta_{n_0}) - \WN(\zeta_{n_0}) \geq (\log n)^2 - O(\log n)
\]
and $\CUEF(\omega \zeta_n) > n-O(\log n).$  Hence
\begin{align*}
  \hspace{1cm}&\hspace{-1cm}
  \Exp\left[ 
	\corO{\one[{\BS}](\CUEF+\WN)\one[{\EvU}](\CUEF) }
	\one[{\ESU\cap \EL[\omega]}](\CUEF)\right] \\
  &\leq
  \Exp\left[ 
    e^{4\CUEF(\omega\zeta_n)-2\CUEF(\zeta_{n_0})-2\WN(\zeta_{n_0})-2n-2(\log n)^2+O(\log n)}
  \right].\\
  \intertext{Applying Proposition~\ref{prop:domination} to estimate the exponential moment, we get that}
  &\Exp\left[ 
	\corO{\one[{\BS}](\CUEF+\WN)\one[{\EvU}](\CUEF) }
		\one[{\ESU\cap \EL[\omega]}](\CUEF)
  \right]
  \leq
  e^{-n -2(\log n)^2 + O(\log n)}.
  \intertext{In particular, if we simply apply a union bound over all $\omega,$ we conclude}
  &\Exp\left[ 
		\corO{\one[{\BS}](\CUEF+\WN)\one[{\EvU}](\CUEF) }\one[{\ESU \cap \EvU}](\CUEF)
  \right]
  \leq e^{-n -2(\log n)^2 + O(\log n)},
\end{align*}
which together with \eqref{eq:1pUB2} completes the proof.
\end{proof}

Completing the upper bound for $\CUEF$ is now a simple matter
\begin{theorem}
  Let 
$\corO{\varpi(1)}$ be any sequence going to $\infty$ as $N\to\infty,$ then
  with probability going to $1$ as $N\to \infty,$
  \[
      \CUEF^*
    \leq \log N -\tfrac 34 \log\log N + \tfrac52\log\log\log N + \corO{\varpi(1)}.
  \]
  \label{thm:ub}
\end{theorem}
\begin{proof}

  We have so far only estimated properties of $\CUEF+\WN$ along rays from the origin.  Here, we will for the first time consider multiple rays, as well as control the influence of the additional noise $\WN.$  As it turns out, having so much noise is somewhat wasteful.  Thus, we introduce another white noise $\RWN$ defined by $\RWN(z) = \WN(|z|)$ for all $z\in\mathbb{D}.$  Lemmas~\ref{lem:1pUB} and~\ref{lem:dyadic} only depended on the marginal of $\WN$ along a single ray, and hence they hold with $\WN$ replaced by $\RWN.$

  Recall that $S_{0} = \{1\}$ and for $j \in \mathbb{N},$
  \[
    S_j = \left\{ e^{2\pi ih2^{-j}N^{-1}}~:~h \in (2\mathbb{N}_0+1), h < 2^j \right\}.
  \]
  Also recall the notation that for any $\omega \in S_j$ for some $j > 0,$ $g(\omega)$ is the element of $S_{j-1}$ with maximal argument, which does not exceed the argument of $\omega.$  
  Define the set of fields for $\omega' \in \T,$
  \[
    \EvUU[\omega']
    :=
    \EvU[\omega']
    \cup
    \bigcup_{j=n_1}^\infty 
    \bigcup_{\omega \in S_j}
    \EL[\omega \omega'].
  \]
  Since for $\omega \in S_j,$ we have $|\omega - g(\omega)| \ll N^{-1}2^{-j}$ and $\rho(\omega) - \rho(g(\omega)) = (1.1)^{-j},$ by Proposition~\ref{prop:domination} there is an absolute constant $c>0$ so that
  \begin{align*}
    \Pr\left[ 
      \CUEF \in \EL[\omega] \cap \EL[g(\omega)]^c
    \right]
    &\leq
    \Pr\left[ 
      \CUEF(\omega \zeta_n) - \CUEF(g(\omega)\zeta_n) > 
      \rho(\omega \zeta_n) - \rho(g(\omega)\zeta_n)
    \right] \\
    &\ll
    e^{-c(2/1.1)^{2j}}.
  \end{align*}
  Thus,
  \begin{align*}
    \Pr&\left[ \CUEF+\RWN \in \BS, \CUEF \in \EvUU \right] \\
    &\leq
    \Pr\left[ \CUEF+\RWN \in \BS, \CUEF \in \EvU \right]
    +\sum_{j=\log_2(n_1)}^{\infty}
    \sum_{\omega \in S_j}
    \Pr\left[ \CUEF \in \EL[\omega] \cap \EL[g(\omega)]^c \right] \\
    &\ll
    \Pr\left[ \CUEF+\RWN \in \BS,\CUEF \in \EvU \right]
    +\sum_{j=\log_2(n_1)}^{\infty}
    2^j
    e^{-c(2/1.1)^{2j}} \\
    &\ll
    \Pr\left[ \CUEF+\RWN \in \BS,\CUEF \in \EvU \right]
    +O(e^{-\omega(n)}).
  \end{align*}

  Applying Lemma~\ref{lem:1pUB}, we have by a union bound that if $t = -\tfrac 34 \log n + \tfrac52\log\log n + \corO{\varpi(1)},$
  \[
    \Pr\left[ 
      \CUEF+\RWN \in \cap_{h=0}^{N-1} \BS[\BF][e^{2\pi i h N^{-1}}]
     , 
      \CUEF \in \cup_{h=0}^{N-1} \EvUU[\nu][e^{2\pi i h N^{-1}}]
    \right] = o(1).
  \]
  Hence by density of the set 
  \[
    \cup_{h=0}^N \cup_{j=1}^\infty e^{2\pi i h N^{-1}}S_j \subset \T
  \]
  and almost sure continuity of $\CUEF,$ either 
  $\CUEF+\RWN \not\in \cap_{h=0}^{N-1} \BS[\BF][e^{2\pi i h N^{-1}}]$
  or 
  \[
    \max_{\omega \in \T} \CUEF(\omega \zeta_n) \leq n+t
  \]
  with high probability.  Since $1-|\zeta_n| = \Theta(N^{-1}),$ if we are in the latter case, we are done by Lemma~\ref{lem:deterministic}.

  It remains to show that $\CUEF+\RWN \in \cap_{h=0}^{N-1} \BS[\BF][e^{2\pi i h N^{-1}}]$ with high probability.  We first remove the influence of $\RWN.$  The maximum of $\RWN$ over the circles $|z| = \zeta_k,$ $k=1,2,\dots,n$ is less than $\log n$ with high probability.  
    (In fact, it is of order $\sqrt{\log n}$.)
  Hence, if we define
  \[
    \BF'(i) = i + \log n ,
  \]
  it suffices to show that $\CUEF \in \cap_{h=0}^{N-1} \BS[\BF'][e^{2\pi i h N^{-1}}]$ with high probability.

  We will presently extend the definition of $\rho.$  
  Fix $k \in \mathbb{N}$ with $1 \leq k \leq n_0.$  
  Set, for $\omega\in\T,$
  \[
    \DL[\omega]
    =\left\{ 
    \F(\omega \zeta_k) > \nu'(i) -11 + \rho(\omega \zeta_k)
    \right\}.
  \]
  Also define
  \[
    S_k^j = \left\{ e^{2\pi ih2^{-j}e^{-k}}~:~h \in (2\mathbb{N}_0+1), h < 2^j \right\}.
  \]
  For every point $\omega \zeta_k \in \D$ with $\omega \in S_k^j,$ 
  define 
    \begin{equation}
    \label{eq-rho-defb}
    \rho(\omega \zeta_k) = 1 + (1.1)^{-1} + \dots + (1.1)^{-j}.
  \end{equation}
  We reuse the notation that for any $\omega \in S_k^j$ for some $j > 0,$ $g(\omega)$ is the element of $S_k^{j-1}$ with maximal argument that does not exceed the argument of $\omega.$ 
  Let $p = e^k n_1.$  Then
  \[
    \cup_{j=n_1}^\infty \cup_{h=0}^p e^{2\pi i h p^{-1}} S_k^j 
  \]
  is dense in $\T.$
  By the same dyadic decomposition and the almost sure continuity of $\CUEF$, we have that
 \begin{align*}
   \Pr&\left[ \CUEF \in 
     \cup_{\omega \in \T} \DL[\omega]
   \right] \\
    &\leq
    \Pr\left[ \CUEF \in 
     \cup_{h=0}^{p} \DL[e^{2\pi i h p^{-1}}]
    \right] 
    +\sum_{j=\log_2(n_1)}^{\infty}
    \sum_{\omega \in S_k^j}
    \Pr\left[ \CUEF \in \DL[\omega] \cap \DL[g(\omega)]^c \right] \\
    &\ll
    \Pr\left[ \CUEF \in 
     \cup_{h=0}^{p} \DL[e^{2\pi i h p^{-1}}]
    \right] 
    + e^{-\omega(n)}.
  \end{align*}
  For all $\omega \in \T,$ we have by Proposition~\ref{prop:domination} that $\Pr[\CUEF(\omega \zeta_k) > t] \ll e^{-t^2/k}$ for all $t \geq 0.$  Hence  
  \[
   \Pr\left[ \CUEF \in 
     \cup_{\omega \in \T} \DL[\omega]
   \right]
   \ll
   e^{k}2^{\log n}e^{-2k-\log n} + e^{-\omega(n)}.
  \]
  Summing over $1 \leq k \leq n,$ we conclude that
  \[
    \Pr\left[ 
      \,\exists\,\omega \in \T,\,1 \leq k \leq n~:~
    \CUEF(\omega \zeta_k) > \BF'(k)
    \right]=o(1).
  \]


\end{proof}

\subsection{Lower bound preliminaries}
\label{sec:1p2}
As in the proof of the upper bound, we let $m$ be as in Proposition~\ref{prop:clt}, with $s=2$ and $K=100.$  
We again let $n_0 = \lfloor n - m\log n \rfloor.$  We also recall the collections of fields describing good mesoscopic behavior: for $\omega \in \T$  and $t_0 \in \R,$ define 
\corO{$\BS=\BS(\omega)$ and $\ES=\ES(\omega)$ as the set of fields satisfying}
\begin{align*}
  \BS[\BF][\omega] &= \left\{ \F(\omega\zeta_i) <\BF(i),~\forall~1\leq i \leq n_0
  \right\} \text{ and } \\
  \ES{t_0}[\omega] &= \left\{ n_0 + t_0 - 1 < \F(\omega\zeta_{n_0}) < n_0 + t_0 \right\}.
\end{align*}

Unlike the upper bound, we must include some barrier information at microscopic scales, although we can not afford to use as dense a set of constraints as in $\BS.$  Hence, we let $\{\eta_n\}$ be a slowly growing sequence
to be specified later. ($\eta_n$ will always be chosen 
so that $\eta_n=o(\log n)$.)
When no confusion occurs, we write $\eta$ instead of $\eta_n$. Let 
\begin{equation}
  \label{eq-defb}
  \corO{b_0=n_0, \quad 
  b_j = n - (\eta -j)\lfloor(m\log n)/ \eta \rfloor. 
}
\end{equation}

For any angle $\theta \in \R,$ let $\ROT[\theta]$ be the map that rotates the hyperbolic disk around $\zeta_{n_0}$ by the angle $\theta.$  This map can be expressed as the composition
\[
  \ROT[\theta] = T_{\zeta_{n_0}}^{-1}
  \circ \{z \mapsto e^{i\theta}z\} \circ T_{\zeta_{n_0}},
\]
in which form it is clearly a hyperbolic isometry.  Hence, the field
\begin{equation}
  \label{eq:invariance}
  \hat \GF(z) := \GF(T_{\zeta_{n_0}}(z)) - \GF(\zeta_{n_0}),
\end{equation}
has the same distribution as $\GF$ (which can be seen by checking covariances and that $\hat \GF(0) = 0$ almost surely).

We use this invariance to our advantage in describing the microscopic behavior of $\CUEF.$  
For $\omega \in \T,\theta\in\R,$ and $p \in \mathbb{N},$ define 
\corO{$\BL=\BL(\theta,\omega)$ and $\ELL=\ELL(\theta,\omega)$ as the set of fields satisfying}
\begin{align*}
  \BL{t_0}[\theta,\omega] &= \{ |\F(\omega\ROT[\theta](\zeta_{b_i})) - b_{i-1} - t_0| \leq \eta^{-1}(\log n),~\forall~1\leq i \leq \eta-1 \} \text{ and} \\
  \ELL{p,t_0}[\theta,\omega] &= \{ |\F(\omega\ROT[\theta](\zeta_{n-p})) - b_{\eta-1}-p - t_0|  \leq  \eta^{-1}(\log n)\}.
\end{align*}
We will choose $p$ independent of $n$.
 $\BL{t_0}[\theta,\omega]$ and $\ELL{p,t_0}[\theta,\omega]$
 represent the microscopic barrier and endpoint events along one possible continuation of the ray $\left\{ \zeta_{i},1 \leq i \leq n_0 \right\}.$  Provided that $|\theta| \ll 1,$ the entire collection 
\[
  \left\{ \zeta_{i},1 \leq i \leq n_0 \right\} \cup \left\{ \ROT(\zeta_{b_i}), 1 \leq i \leq \eta \right\}
\]
can be checked to be within uniformly bounded distance of the ray that connects $0$ to the endpoint $\ROT(\zeta_n)$ (see Figure~\ref{fig:Q}).  More useful for our purposes will be the following estimate
\begin{lemma}
  There is an absolute constant $\Xi \in (0,\pi)$ so that for all $|\theta| < \Xi$ and all $n_0 \in \mathbb{N}$, the set $\left\{ \ROT[\theta](\zeta_j), j \geq n_0 \right\}$ is contained in the wedge $\left\{ z \in \D : |\arg z| \leq \tfrac 12 e^{-n_0} \right\}.$
  \label{lem:lemonslice}
\end{lemma}
\begin{proof}
Let $\theta \in [0,\pi]$ and $j \in \N$ with $j \geq n_0.$  Consider the triangle formed by the points $\left\{ 0, \zeta_{n_0}, \ROT(\zeta_j) \right\}$,
 and let $\psi_j$ be the angle at $0.$  
  By the hyperbolic law of sines
  \[
    \frac{\sin(\psi_j)}{\sinh(j-n_0)}
    =\frac{\sin(\pi-\theta)}{\sinh(j)}
  \]
  \corO{Thus} $\psi_j$ is monotone increasing in $j,$ hence taking $j \to \infty,$ we have
  \[
    \sin(\psi_j) \leq e^{-n_0}\sin(\pi-\theta)=e^{-n_0}\sin(\theta)
  \]
  Hence, by adjusting $\Xi$ appropriately, we can ensure $\psi_j < \tfrac 12 e^{-n_0}.$   A symmetric argument holds for $\theta \in [-\pi, 0].$
\end{proof}

Moreover, the covariance structure along the quasi-geodesics
\(
  \left\{ \zeta_{i},1 \leq i \leq n_0 \right\} \cup \left\{ \ROT(\zeta_{b_i}), 1 \leq i \leq \eta \right\}
\)
still can be well-approximated by branching random walk:
\begin{lemma}
  For $|\theta_1|, |\theta_2| < \Xi,$
  uniformly in $h,j \in \mathbb{N}$ with $h,j \geq n_0$
  \[
    \Exp \GF(\ROT[\theta_1](\zeta_h))\GF(\ROT[\theta_2](\zeta_j))
    =\tfrac{1}{2}\min\{-\log|\sin\tfrac{\theta_1-\theta_2}{2}|,h-n_0,j-n_0\} + \tfrac12n_0 + 
    \xi,
  \]
  with
  \[
      \xi = \frac{1}{2}\left( \log|\cos(\tfrac{\theta_1}{2})\cos(\tfrac{\theta_2}{2})|\right) + O(e^{-j+n_0}+e^{-h+n_0}).
  \]
In particular, $|\xi| \ll 1$ uniformly.
  Also, uniformly in $h < n_0,$  $j \geq n_0$ and $|\theta|<\Xi$,
  \[
    \Exp \GF(\zeta_h)\GF(\ROT[\theta](\zeta_j))
    =\tfrac{h-\log 2}{2} + \log|\cos(\tfrac{\theta}{2})| + O(e^{-j+n_0}+e^{-n_0+h}+e^{-h}).
  \]
  \label{lem:branch2}
\end{lemma}
\begin{proof}
  For $j \geq n_0,$
  \(
  \ROT[\theta](\zeta_j) = T_{\zeta_{n_0}}( e^{i\theta} \zeta_{j-n_0}).
  \)
  In terms of $\hat \GF,$ this means that for $j \geq n_0,$ 
  \[
    \GF(\ROT[\theta](\zeta_j)) - \GF(\zeta_{n_0})
    =\hat\GF(e^{i\theta} \zeta_{j-n_0}).
  \]
  By \eqref{eq:covcosh}, for $j \geq n_0,$
  \[
    \Exp\left[ \GF(\zeta_{n_0}) \hat\GF(e^{i\theta} \zeta_{j-n_0})\right]
    =\frac{1}{2}
    \log \left(\frac 
    { \cosh( {\dH(0,\ROT[\theta](\zeta_j))}{2}^{-1} )}
    {\cosh( {\dH(\zeta_{n_0},\ROT[\theta](\zeta_j))}{2}^{-1})\cosh( {\dH(0,\zeta_{n_0})}{2}^{-1} )}
    \right).
  \]
  Applying the hyperbolic law of cosines to the triangle with vertices $\left\{ 0, \zeta_{n_0}, \ROT[\theta](\zeta_j) \right\}$ at the angle at $\zeta_{n_0},$ we get that
  \[
    \cosh(\dH(0,\ROT(\zeta_j))) = 
    \cosh(n_0)\cosh(j-n_0)
    -
    \sinh(n_0)\sinh(j-n_0)\cos(\pi-\theta).
  \]
  Hence,
  \[
    \dH(0,\ROT(\zeta_j)) = j + \log\left( \tfrac{1 + \cos(\theta)}{2} \right) + O(e^{-j+n_0} + e^{-n_0}),
  \]
  uniformly in $|\theta| < \Xi$ and $n_0,j$ with $n_0 \leq j.$
  From here, it follows that
  \begin{equation}
    \label{eq:gfgf}
    \Exp\left[ \GF(\zeta_{n_0}) \hat\GF(e^{i\theta} \zeta_{j-n_0})\right] = 
    \tfrac14\log(\tfrac{1+\cos(\theta)}{2}) 
    +\tfrac12\log 2
    +O(e^{-j+n_0}+e^{-n_0}).
  \end{equation}
  To prove the first displayed line of the lemma, we add and subtract $\GF(\zeta_{n_0}),$ so that
  \begin{align*}
    \Exp\left[  \GF(\ROT[\theta_1](\zeta_h))\GF(\ROT[\theta_2](\zeta_j))\right]
    =
    &\Exp\left[  \hat\GF(e^{i\theta_1}(\zeta_{h-n_0}))\hat\GF(e^{i\theta_2}(\zeta_{j-n_0}))\right]
    + \Exp\left[  \GF(\zeta_{n_0})^2 \right]\\
    +&\Exp\left[  \GF(\zeta_{n_0})\hat\GF(e^{i\theta_2}(\zeta_{j-n_0}))\right]
    +\Exp\left[  \hat\GF(e^{i\theta_1}(\zeta_{h-n_0}))\GF(\zeta_{n_0})\right].
  \end{align*}
  To the first two summands on the right hand side, we apply Lemma~\ref{lem:branch} and \eqref{eq:varcosh} respectively.  To the second two, we apply \eqref{eq:gfgf}, which gives the first conclusion of the lemma.
  
  For the second conclusion of the lemma, we start by writing 
  \begin{align*}
    \Exp\left[\GF(\zeta_h)\GF(\ROT[\theta](\zeta_j))\right]
    =
    &\Exp\left[  \hat\GF(e^{i\pi}(\zeta_{n_0-h}))\hat\GF(e^{i\theta}(\zeta_{j-n_0}))\right]
    + \Exp\left[  \GF(\zeta_h)\GF(\zeta_{n_0}) \right]\\
    +&\Exp\left[  \GF(\zeta_{n_0})\hat\GF(e^{i\theta}(\zeta_{j-n_0}))\right].
  \end{align*}
  To the expectation with only $\hat\GF,$ we use distributional invariance of the field to replace $\hat\GF$ with $\GF,$ and then we additionally reflect the field to get
  \begin{align*}
    \Exp\left[\GF(\zeta_h)\GF(\ROT[\theta](\zeta_j))\right]
    =
    &\Exp\left[\GF(\zeta_{n_0-h})\GF(e^{i(\pi-\theta)}\zeta_{j-n_0})\right] 
    + \Exp\left[  \GF(\zeta_h)\GF(\zeta_{n_0}) \right]\\
    +&\Exp\left[  \GF(\zeta_{n_0})\hat\GF(e^{i\theta}(\zeta_{j-n_0}))\right]
  \end{align*}
  Applying Lemma~\ref{lem:branch} to the first expectation on the right hand side, \eqref{eq:covcosh} to the second, and \eqref{eq:gfgf} to the third, we have
  \begin{align*}
    \Exp\left[\GF(\zeta_h)\GF(\ROT[\theta](\zeta_j))\right]
    =
    &\tfrac14\log(\tfrac{1-\cos(\pi-\theta)}{2}) + \tfrac{h}{2} - \log 2
    +O(e^{-h} + e^{-j+n_0} + e^{-n_0+h})
    \\
    +&\tfrac14\log(\tfrac{1+\cos(\theta)}{2}) + \tfrac12\log 2 + O(e^{-n_0} + e^{-j+n_0}),
  \end{align*}
  which after simplifying gives the statement in the lemma.
\end{proof}

We now define 
\begin{equation}
  \Ev{\BF,p,t_0}[\omega] =  \bigcup_{\substack{ h \in \mathbb{Z} \\ |h| < \Xi e^{n-n_0}}}\BS[\BF][\omega] \cap \ES[\omega] \cap \BL[he^{-n+n_0},\omega] \cap \ELL{p,t_0}[he^{-n+n_0},\omega]
  , 
  \label{eq:Ev}
\end{equation}
the event to which we will apply the second moment method.  
As in the upper bound, when $\omega=1,$ we drop it from the notation, so that $\Ev = \Ev[1].$ For $\BL[\theta,\omega]$ and $\ELL[\theta,\omega],$ we suppress the second argument.

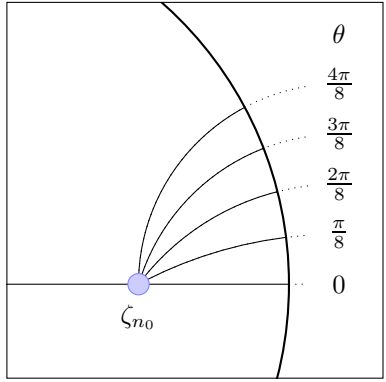
\begin{figure}
    \begin{tikzpicture}[scale=5]
        \clip[draw] (0.25, -0.25) rectangle (1.25, 0.75);
        \draw[thick] (0, 0) circle (1);
        \draw (0,0) -- (1,0);

        \begin{scope}
            \clip (0.25, 0) rectangle (1.05,0.75);
            \draw[dotted] (1.1333,0) circle (0.5333);
            \draw[dotted] (1.1333,-0.156) circle (0.555);
            \draw[dotted] (1.1333,-0.40) circle (0.6666);
            \draw[dotted] (1.1333,-1.0) circle (1.1333);
            \draw[dotted] (1,0) -- (1.1333,0);
        \end{scope}
        \begin{scope}
            \clip (0.25, 0) rectangle (1.25,0.75);
            \clip (0, 0) circle (1);
            \draw (1.1333,0) circle (0.5333);
            \draw (1.1333,-0.156) circle (0.555);
            \draw (1.1333,-0.40) circle (0.6666);
            \draw (1.1333,-1.0) circle (1.1333);
        \end{scope}

        \draw (1.1333,0.6666) node {$\theta$};
        \draw (1.1333,0.5333) node {$\tfrac{4\pi}{8}$};
        \draw (1.1333,0.4000) node {$\tfrac{3\pi}{8}$};
        \draw (1.1333,0.2666) node {$\tfrac{2\pi}{8}$};
        \draw (1.1333,0.1333) node {$\tfrac{\pi}{8}$};
        \draw (1.1333,0.0) node {$0$};

        \draw (0.6,0) node[shape=circle,draw=blue!50,fill=blue!20,inner sep=1mm,label=-90:$\zeta_{n_0}$]{};
    \end{tikzpicture}
    \caption{Depiction of the quasi-geodesics $[0,\zeta_{n_0}] \cup \ROT(\R_+)$ used in the event $\Ev,$ for various values of $\theta.$}
    \label{fig:Q}
\end{figure}

We will show \corO{that}
 some $\Ev[\omega]$ occurs with high probability with $\omega$ ranging over a set of equally spaced $\omega$ of cardinality $e^{n_0}.$  This we do by a second moment estimate.  To gain enough independence to push this through, we must effectively work on a sub-problem that discards the correlations that arise in $\CUEF(z)$ for $\dH(0,z) \ll 1$ (in fact, due to technical losses, we will need to allow this distance to grow, but still as
$o(\log n)$).  Hence for a parameter $r \in \N$ we define new fields $\CUEFr, \GF_r, \WN_r$ by the formula
\begin{equation}
    \label{eq:Fr}
    \F_r(z) = 
    \begin{cases}
      0 & \text{if } \dH(z,0) < r, \\
      \F(z) - \F(\zeta_r e^{2\pi i h [e^{n_0}]^{-1}}) & \text{if } \dH(z,0) \geq r,~h=[ \tfrac{\arg(z)[e^{n_0}]}{2\pi} + \tfrac 12],
    \end{cases}   
\end{equation}
where $[\cdot]$ denotes integer part.
With this convention, the fields $\CUEFr$ and $\GF_r$ are invariant in distribution under rotations of the disk by angles $2\pi h [e^{n_0}]^{-1},$ for $h \in \mathbb{Z}.$

Fix $0 < \delta < \frac 14.$ 
We will choose
  \begin{equation}
    \label{eq-def-r}
    r = \left[3(\log n_0)^{1-\delta}\right],\quad t_0=-\frac34 \log n,
  \end{equation}
  Also, throughout the proof of the lower bound for $\CUEF$, we can now work on any choice of barrier we so please, as we will not need to rule out the event that $\CUEF$ exceeds it. So we choose one which is more nearly optimal than in the upper bound, specifically
\begin{equation}
	\label{eq-nulb}
  \BF(i) = \frac{i-r}{n_0}(n_0 - \tfrac34 \log n_0) +  (\log n_0)^{1 - \delta} 
\end{equation}
for all $r \leq i \leq n_0.$  

Finally, because Proposition~\ref{prop:clt} allows us to compare $\CUEF+\WN$ to $\GF + \WN$ and not $\CUEF$ to $\GF,$ we will actually prove a lower bound for the field $\CUEF + \WN$ first.  As it will turn out, having shown a complete lower bound for the maximum of this field on an appropriate set of cardinality $\Theta(N)$, removing the effect of $\WN$ is possible knowing only a few points $\CUEF$ are large.  In effect, we use an upper bound on $\CUEF$ to show that the lower bound for $\CUEF+\WN$ transfers to $\CUEF.$ 

Our first task will be to estimate $\Pr\left( \CUEFr + \WN_r \in \Ev \right).$  
By design, a nearly optimal upper bound for this probability follows from the trivial estimate
\[
  \Pr\left( \CUEFr + \WN_r \in \Ev \right)
  \leq 
  \Pr\left( \CUEFr + \WN_r \in \BS \cap \ES\right)
\]
and Lemma~\ref{lem:ballotub} and Proposition \ref{prop:clt}.  
For the lower bound, we would like to use a conditional second moment method argument.  That is, we define a set $\Theta := e^{-n+n_0}\mathbb{Z} \cap (-\Xi,\Xi)$ and define
\[
  \tilde{Z} =  \sum_{\theta \in \Theta}\one[ \CUEFr + \WN_r \in \BS[\BF] \cap \ES \cap \BL(\theta) \cap \ELL(\theta)]. 
\]
\corO{Define the event}
\corO{\begin{equation}
\label{eq-rev1}
\Meso=\{\CUEFr + \WN_r \in \BS \cap \ES\}.
\end{equation}}
Then, we would estimate
\[
  \Pr\left[ \CUEFr + \WN_r \in \Ev | \Meso \right]
  =
  \Pr\left[ \tilde{Z} \geq 1 | \Meso \right]
  \geq \frac{
    (\Exp[ \tilde{Z} | \Meso ])^2
  }{
    \Exp[ \tilde{Z}^2 | \Meso ]
  }.
\]

One of the subtleties of this strategy is that to get this probability going to $1,$ one basically needs that for most pairs $(\theta_1,\theta_2)$ from $\Theta$
\begin{align*}
  &\Pr[\CUEFr + \WN_r
    \in \BL(\theta_1)
    \cap \BL(\theta_2)
    \cap \ELL(\theta_1)
  \cap \ELL(\theta_2)    | \Meso] \\
  =&\Pr[\CUEFr + \WN_r
    \in \BL(\theta_1) 
  \cap \ELL(\theta_1)    | \Meso] \\
  \cdot&\Pr[\CUEFr + \WN_r
    \in \BL(\theta_2) 
  \cap \ELL(\theta_2) | \Meso](1+o(1)).
\end{align*}
In the case of branching random walk, this is achieved using actual independence of the two events, once a suitably small common ancestral tree is discarded.  In our case, no such independence is available.  Moreover, the need for this accuracy makes introducing and removing biasing terms, as we do in the proofs of Lemmas~\ref{lem:dyadic} and \ref{lem:1pUB}, very expensive as we incur a constant multiplicative loss at every such stage.

We solve this problem by changing the second moment formalism.  
First, for any finite bias term $\Bias,$ and any cylinder function $\phi$,
	set
	\begin{equation}
	  \label{eq-FB}
    \FMC{\Bias}[\phi(\CUEFr,\WN_r)] = \frac{\Exp \left[\one[{\BS\cap\ES}](\CUEFr+\WN_r) e^{\Bias{\CUEFr}}\phi(\CUEFr,\WN_r)\right]}
    {\Exp \left[\one[{\BS\cap\ES}](\CUEFr+\WN_r) e^{\Bias{\CUEFr}}\right]}.
  \end{equation}
We will drop the dependence of the notation on $\Bias$ when $\Bias{\F} = 2\F(\zeta_{n_0}),$ as this is how we will typically use it.
Next, we define for any 
	real
$\theta$
with $|\theta| \leq \Xi$ and any $p \in \mathbb{N},$
\begin{equation}
  \label{eq-defY}
  \MFI(\theta) = e^{2\CUEFr(\ROT[\theta](\zeta_{n-p})) - 2\CUEFr(\ROT[\theta](\zeta_{b_1}))} \one[ \CUEFr + \WN_r \in \BL(\theta) \cap \ELL(\theta)].
\end{equation}
In terms of this replacement for the indicator, we form the biased counting variable
\[
  Z=  \sum_{\theta \in \Theta} \MFI(\theta). 
\]

We again want to show that this variable is non-negative, as this implies one of the indicated events holds.  Applying Cauchy Schwarz,
\begin{equation}
  \label{eq:2mm}
  \FMC[ \one[Z > 0]]
  \geq \frac
  { \bigl(\sum_{\theta}\FMC\left[ \MFI(\theta) \right]\bigr)^2}
  {\sum_{\theta_1,\theta_2}\FMC\left[\MFI(\theta_1)\MFI(\theta_2)\right]},
\end{equation}
where in both summations, $\theta$ ranges over 
\(
\Theta.
\)

Estimating $\FMC$ requires that we are able to evaluate the probability of $\Meso$ under various exponential biases.  Using Proposition~\ref{prop:clt}, we can reduce this to the same question about $\GF,$ which is still a nontrivial estimate.  Using Gaussian comparison theorems, we show:
\begin{lemma}
     Let $\mu: \D \to \R$ be given by $\mu(z) = 2\Exp[\GF(z)\GF_r(\zeta_{n_0})].$  
    Let $F = \one[{\BS \cap \ES}].$  
    There is a sequence $\{\rho_n\}$
    going to $0$ as $n \to \infty$ so that the following holds.  For any $C>0$ there is 
    a $\kappa_0(C)$ sufficiently large so that for all $\mu'(z) : \D \to \R$ with $|\mu(z)-\mu'(z)| \leq C$ and for all $n \geq \kappa_0(C),$
  \[
    \biggl|\frac{
      \Exp\left[ F(\GF_r +\WN_r + \mu') \right]
    }{
      \Exp\left[ F(\GF_r +\WN_r + \mu) \right]
    }
    - 1 \biggr| \leq \rho_n.
  \]
  \label{lem:ratio}
\end{lemma}
\begin{proof}
    By Lemma \ref{lem:branch2},
    the mean $\mu(\zeta_i)=i+O(1)$ for all $1\leq i \leq n_0.$
    We want to apply Lemma \ref{lem-ratio}, however
    comparing with the definition of $\BF$, see \eqref{eq-nulb},
    we see that
    the mean $\mu(i)$ is slightly too large to allow one to flatten the slope of the barrier.
    We introduce another exponential bias to compensate for this.

    Set $d = (1-\tfrac 34\tfrac{\log n_0}{n_0}).$  Let $\xi : \D \to \R$ be given by $\xi(z) = \Exp\left[ \WN_r(z)\WN_r(\zeta_{n_0}) \right]\cdot(\tfrac 32\tfrac{\log n_0}{n_0}),$ 
    which vanishes except at two points where it is  
    $O((\log n_0)/n_0)$.
    Then by Lemma~\ref{lem:means},
    \[
        \Exp\left[ F(\GF_r +\WN_r + \mu) \right]
        =
        \frac{
            \Exp\left[ F(\GF_r +\WN_r + d\cdot\mu-\xi)e^{\tfrac 32\tfrac{\log n_0}{n_0}(\GF_r+\WN_r)(\zeta_{n_0})} \right]
        }
        {
            \Exp\left[e^{\tfrac 32\tfrac{\log n_0}{n_0}(\GF_r+\WN_r)(\zeta_{n_0})} \right]
        }.
    \]
    On the event $\GF_r +\WN_r + d\cdot \mu -\xi \in \ES,$      
    \[
        \tfrac 32\tfrac{\log n_0}{n_0}(\GF_r+\WN_r)(\zeta_{n_0})
        =
	O( (\log n_0)^2/n_0).
    \]
    On the other hand,
    \[
        \Exp\left[e^{\tfrac 32\tfrac{\log n_0}{n_0}(\GF_r+\WN_r)(\zeta_{n_0})} \right]
        =e^{O( (\log n_0)^2/n_0)}.
    \]
    Hence we conclude that
    \begin{equation}
        \label{eq:mutweak}
        \Exp\left[ F(\GF_r +\WN_r + \mu) \right]
        =
        \Exp\left[ \corO{
				F(\GF_r +\WN_r + d\cdot\mu-\xi)
				}
				\right](1+O( (\log n_0)^2/n_0)).
    \end{equation}
    The same statement holds with $\mu$ replaced by $\mu',$ and holds uniformly in such $\mu'.$ Hence, the lemma follows by applying Lemma \ref{lem-ratio} in conjunction with Proposition
    \ref{prop-comp-bar}, in an analogous way to as was done in \eqref{eq-Sat1}.
\end{proof}
\begin{remark}
    An examination of the proof of Lemma \ref{lem:ratio}
    shows that the estimate there
is uniform also in $t_0 \in [-C\log n_0, -\tfrac34 \log n_0]$ for any fixed $C>\tfrac 34.$
We will not need this fact.
\end{remark}

We will assume that $\rho_n \gg (\log n)^{-1}$ for the purpose of comparing this error to other ones. 
This sequence
$\rho_n$ controls how quickly we can let $\eta_n \to \infty.$  
\emph{We now fix the sequence $\eta_n \to \infty$ so that $\rho_n\eta_n^3 
\to 0,$ which in particular ensures that $\eta_n=o(\log n)$.}

Applying Lemma~\ref{lem:ratio} and Proposition~\ref{prop:clt} to some of the biases we need, we have:
\begin{corollary}
  Let $\Bias{\F}[1] = 2\F(\zeta_{n_0}),$ and let $\Bias$ be any one of the following.
  \begin{enumerate}
    \item For all $|\theta| \leq \Xi,$ 
      \[
	\Bias{\F}[2] = 
	2\F(\ROT[\theta](\zeta_{n-p}))
	-2\F(\ROT[\theta](\zeta_{b_1}))
	+2\F(\zeta_{n_0}).
      \]
    \item For all $|\theta_1|,|\theta_2| \leq \Xi,$
      \begin{align*}
	\Bias{\F}[3] &= 
	2\F(\ROT[\theta_1](\zeta_{n-p}))
	-2\F(\ROT[\theta_1](\zeta_{b_1})) \\
	&+2\F(\ROT[\theta_2](\zeta_{n-p}))
	-2\F(\ROT[\theta_2](\zeta_{b_1}))
	+2\F(\zeta_{n_0}).
      \end{align*}
  \end{enumerate}
  Then,
\begin{equation}
  \label{eq-star-feb1}
    \frac{\Exp[ \one[\Meso]e^{\Bias{\CUEFr}} ] }
    {\Exp[ \one[\Meso]e^{\Bias{\CUEFr}[1]} ] }
    =\frac{\Exp[ e^{\Bias{\CUEFr}} ]}
    {\Exp[e^{\Bias{\CUEFr}[1]} ] }
    (1+O(\rho_n)).
  \end{equation} 
  We also have
\begin{equation}
  \label{eq-2star-feb1}
    \Exp[ \one[\Meso]e^{\Bias{\CUEFr}[1]} ] 
    \asymp \frac{e^{n_0-r} (\log n_0)^{2-2\delta} }{n_0^{3/2}}.
  \end{equation} 
  \label{cor:meso}
\end{corollary}
\begin{proof}
    The estimate \eqref{eq-star-feb1}
  follows from Proposition~\ref{prop:clt} to bring the expectation to a Gaussian one, Lemma~\ref{lem:branch2} to evaluate the means, and Lemma~\ref{lem:ratio} to conclude the ratio of expectations.  
  The proof of \eqref{eq-2star-feb1} begins the same way.  Using 
  \corO{\eqref{eq-rev1}, } Proposition~\ref{prop:clt} and Corollary~\ref{cor:exact}, we can write
  \[
    \Exp[ \one[\Meso]e^{\Bias{\CUEFr}[1]} ] 
    =(\Exp\left[ F(\GF_r +\WN_r + \mu) \right]+O(\log N)^{-K})
    \Exp[e^{\Bias{\GF_r}[1]}],
  \]
  in the notation of Lemma~\ref{lem:ratio}.
  Applying \eqref{eq:mutweak}, we have that 
  \begin{align*}
    \Exp[ \one[\Meso]e^{\Bias{\CUEFr}[1]} ] 
    &=
    \Exp\left[ F(\GF_r +\WN_r + d\cdot\mu-\xi \corE{)}\right]
    \Exp[e^{\Bias{\GF_r}[1]}](1+O( (\log n_0)^2/n_0)) \\
    &+
    O((\log N)^{-K})
    \cdot
    \Exp[e^{\Bias{\GF_r}[1]}].
\end{align*}
    Using Proposition~\ref{prop-comp-bar} and the ballot theorem 
      (see Theorem
    \ref{theo-app-ballot}),
    \[
        \Exp\left[ F(\GF_r +\WN_r + d\cdot\mu-\xi)\right]
        \asymp \frac{ (\log n_0)^{2-2\delta}}{n_0^{3/2}}.
    \]
    Using Lemma~\ref{lem:branch}, 
    \(
    \Exp[e^{\Bias{\GF_r}[1]}] \asymp e^{n_0 -r},
    \)
    the proof now follows.
\end{proof}
\begin{remark}
    An examination of the proof of Corollary \ref{cor:meso}
    shows that the estimate again
is uniform in $t_0 \in [-C\log n_0, -\tfrac34 \log n_0]$ for any fixed $C>\tfrac 34.$
We still will not need this fact.
\end{remark}
    \subsection{Field moment calculus}
\label{sec:fmc}
    Because we are constrained to exponential moments with specific parameters, we are not able to show by a direct appeal to Markov's inequality that certain quantities are concentrated.  To circumvent this, we use the spatial correlations inherent in $\CUEF$ and $\GF$ to effectively compute some moments of $\CUEF,$ restricted to specific events and under exponential biases.

    In what follows, fix $|\theta| \leq \Xi.$  Let $\gamma(t)$ be the hyperbolic arclength parameterized geodesic from $\zeta_{n_0}$ to $\ROT[\theta](\zeta_n).$
    Towards estimating some moments of $\CUEF$, we define a family of biasing functions that are locally perturbed modifications of $\Bias[2].$  For 
		\corO{ $q>0$, }
		let $W_k(q)$ be the collection of cylinder functions from $\D \to \R:$
    \[
      \F \mapsto \Bias{\F}[2] + 
      \sum_{z \in \mathbf{z}} \F(z)
      -\sum_{y \in \mathbf{y}} \F(y),
    \]
    where $\mathbf{z}$ and $\mathbf{y}$ have the following properties:
    \begin{enumerate}[(a)]
      \item
	$\mathbf{z}$ and $\mathbf{y}$ are finite
	subsets of the trace of $\gamma$ with $|\mathbf{z}|=|\mathbf{y}|=k.$
      \item
	There is a bijection $\phi : \mathbf{z} \to \mathbf{y}$ so that
	\[
	  \max_{z \in \mathbf{z}} \dH(z,\phi(z)) \leq 1.
	\]
      \item
	Letting $\mathbf{z}' = \mathbf{z} \cup \left\{ \zeta_{n_0}, \ROT[\theta](\zeta_{n-p}) \right\}.$ 
	\[
	  \max_{z \in \mathbf{z}'} 
	  e^{-N\exp(-\dH(0,z))}
	  \prod_{
	    \substack{x \in \mathbf{z}' \\ 
	  x \neq z}} 
	  \coth(\dH(x,z)/2) \leq q.
	\]
    \end{enumerate}

    Estimating moments of $\CUEF$ will be ultimately reduced to the estimation of exponentials of elements of $W_k$ for some fixed $k$.  These exponential moments of $\CUEF$ then need to be compared to those of $\GF$ with high precision.  However, we will need to consider points in the field where a direct comparison to Gaussian is impossible; as a specific example
    \[
      \Exp\left[ e^{\Bias{\CUEFr}[2]} \right]
      \neq
      \Exp\left[ e^{\Bias{\GF_r}[2]} \right](1+o(1)).
    \]
    Certain ratios of expectations can be accurately compared, however. Specifically:
    \begin{lemma}
      Fix $k \in \mathbb{Z}$ and $q \in \R.$ Let $\Bias \in W_k(q)$ be arbitrary, with $\mathbf{z}$ and $\mathbf{y}$ so that
      \[
	\Bias{\F} = 
	\Bias{\F}[2] + 
	\sum_{z \in \mathbf{z}} \F(z)
	-\sum_{y \in \mathbf{y}} \F(y).
      \]
      Let $w=\ROT[\theta](\zeta_{n-p}),$ the point that appears in $\Bias[2].$
      Define
      \[
	\epsilon
	=
	\frac{
	  \prod_{y \in \mathbf{y}} \tanh(\dH(w,y)/2)^2
	}{
	  \prod_{z \in \mathbf{z}} \tanh(\dH(w,z)/2)^2
	}
	-1.
      \]
      There are constants $\delta_0 > 0$ and $p_0$ independent of $n$ so that the following holds.
      If
      \begin{equation*}
	\Delta = \max_{z \in \mathbf{z} \cup \{\zeta_{n_0}\}} 
	e^{-N\exp(-\dH(0,z))}
	\prod_{
	  \substack{x \in \mathbf{z}\cup \{w,\zeta_{n_0}\} \\ 
	x \neq z}} 
	\coth(\dH(x,z)/2) \leq \delta_0,
      \end{equation*}
      and if $p_0 \leq p \leq \log\log n,$
      then
      \[
	\frac{\Exp\left[ e^{\Bias{\CUEFr}} \right]}{
	  \Exp\left[ e^{\Bias{\CUEFr}[2]} \right]
	}
	=
	\frac{\Exp\left[ e^{\Bias{\GF_r}} \right]}{
	  \Exp\left[ e^{\Bias{\GF_r}[2]} \right]
	}
	+O\left( \epsilon + \Delta \right),
      \]
      with the implied constants uniform in $W_k(q).$
      \label{lem:fmc0}
    \end{lemma}
    \begin{proof}
      Define
      \[
	a = 
	|w|^{2N}
	\frac{ \tanh(\dH(\ROT[\theta](\zeta_{b_1}),w))^2
	\tanh(\dH({\zeta_r},w))^2
      }{
	\tanh(\dH(w,\zeta_{n_0}))^2
      }.
    \]
    By Corollary~\ref{cor:exact}
    \[
      \frac{
	\Exp\left[ 
	  e^{\Bias{\CUEFr}[2]}
	\right]
      }
      {
	\Exp\left[ 
	  e^{\Bias{\GF_r}[2]}
	\right]
      }
      =1 - a + O\bigl(\Delta\bigr), 
    \]
    provided $p_0$ is chosen sufficiently large.
    We can also assure, by possibly increasing $p_0,$ that $a < \tfrac 14.$
    By choosing $\delta_0$ sufficiently small, the $O(\Delta)$ error can also be made less than $\tfrac 14.$

    The first correction term to the other exponential bias is similar:
    \[
      \frac
      {
	\Exp\left[e^{\Bias{\CUEFr}} \right]
      }
      {
	\Exp\left[e^{\Bias{\GF_r}} \right]
      }
      =1 - a
      \frac{
	\prod_{y \in \mathbf{y}} \tanh(\dH(w,y)/2)^2
      }{
	\prod_{z \in \mathbf{z}} \tanh(\dH(w,z)/2)^2
      }
      +O(\Delta).
    \]
    Hence we have that
    \[
      \frac{\Exp\left[ e^{\Bias{\CUEFr}} \right]}{
	\Exp\left[ e^{\Bias{\CUEFr}[2]} \right]
      }
      =
      \frac{\Exp\left[ e^{\Bias{\GF_r}} \right]}{
	\Exp\left[ e^{\Bias{\GF_r}[2]} \right]
      }
      \frac{1-a(1+\epsilon)+O(\Delta)}{1-a+O(\Delta)}
      .
    \]
    To complete the proof, we observe that the variance of $\Bias{\GF_r}$ is no more than that of $\Bias{\GF_r}[2]$ plus an absolute constant, so that the ratio of Gaussian expectations is bounded above by an absolute constant:
    \begin{align*}
      \Var\left(
      \Bias{\GF_r}
      \right)
      -\Var\left(
      \Bias{\GF_r}[2]
      \right)
      =&\sum_{z \in \mathbf{z}} 2\Cov(\Bias{\GF_r}[2], \GF(z)-\GF(\phi(z))) \\
      +&\sum_{z \in \mathbf{z}} 4\Var(\GF(z)-\GF(\phi(z))). 
    \end{align*}
    Both the covariance and variance summands can be estimated by an absolute constant using \eqref{eq:correlationbound}.  Thus
    \begin{equation}
        -k \ll
      \Var\left(
      \Bias{\GF_r}
      \right)
      -\Var\left(
      \Bias{\GF_r}[2]
      \right)
      \ll k.
      \label{eq:varbound}
    \end{equation}
  \end{proof}

  We also observe that the biasing terms in $W_k$ influence the probability of $\Meso$ in a negligible way.
  \begin{lemma}
    Fix $k \in \mathbb{N}$ and $q \in \R.$ 
    Let $F = \one[{\BS \cap \ES}].$
    Uniformly in $\Bias \in W_k(q),$
    \[
      \frac{\Exp\left[ F(\GF_r+\WN_r)e^{\Bias{\GF_r}} \right]}
      {\Exp\left[ F(\GF_r+\WN_r)e^{\Bias{\GF_r}[2]} \right]}
      =
      \frac{\Exp\left[ e^{\Bias{\GF_r}} \right]}
      {\Exp\left[ e^{\Bias{\GF_r}[2]} \right]}
      +O(\rho_n).
    \]
    \label{lem:fmc1}
  \end{lemma}
  \begin{proof}
    Let $\mu :\D \to \R$ and $\mu_2 : \D \to \R$ be the means of $\GF_r$ under the biases $e^{\Bias{\GF_r}}$ and $e^{\Bias{\GF_r}[2]}$ respectively.  Then we have by Lemma~\ref{lem:means} and \eqref{eq:covcosh} that
    \[
      \mu(x) - \mu_2(x) = \sum_{z \in \mathbf{z}}
      \log\left( \frac
      {\cosh(\dH(0,z))\cosh(\dH(x,\phi(z)))}
      {\cosh(\dH(0,\phi(z)))\cosh(\dH(x,z))}
      \right) 
    \]
    In particular, this can be uniformly bounded in terms of the distances of $z$ to $\phi(z),$ and hence can be controlled by an absolute constant.
    Hence by Lemma~\ref{lem:ratio},
    \[
      \frac{\Exp\left[ F(\GF_r+\WN_r)e^{\Bias{\GF_r}} \right]}
      {\Exp\left[ F(\GF_r+\WN_r)e^{\Bias{\GF_r}[2]} \right]}
      \frac
      {\Exp\left[ e^{\Bias{\GF_r}[2]} \right]}
      {\Exp\left[ e^{\Bias{\GF_r}} \right]}
      =
      1+O(\rho_n).
    \]
    Using \eqref{eq:varbound}, the claim follows. 
  \end{proof}

  \corO{Recall the definition of $b_1$, see \eqref{eq-defb}.}
  \begin{lemma}
    Fix $k \in \mathbb{N},$ $q \in \R,$ and let $\Bias \in W_k(q)$ be arbitrary.  Write $\mathbf{z}$ and $\mathbf{y}$ so that
    \[
      \Bias{\F} = 
      \Bias{\F}[2] + 
      \sum_{z \in \mathbf{z}} \F(z)
      -\sum_{y \in \mathbf{y}} \F(y).
    \]
    There is a $p_0$ sufficiently large and independent of $n$ so that
    for all $p_0 \leq p \leq \log\log n$
    and for all $0 \leq t \leq n-p-n_0,$
    \[
      \FMC{\Bias}[ \CUEF(\gamma(t)) - \CUEF(\gamma(0))] = 
      (t - b_1 + n_0)_+
      +\xi(t),
    \]
    \corO{see \eqref{eq-FB},}
    where $|\xi(t)| \ll (\log n) \rho_n^{1/2},$ uniformly in $W_k(q).$
    \label{lem:fmcmean}
  \end{lemma}
  \begin{proof}
	Let $\{t_0, t_1,t_2,t_3,\dots, t_{\ell}\}$ be a subset of $[0,t],$ written in increasing order, with $t_{\ell} = t$ and $t_0 = 0.$  We will make an explicit $n$-dependent choice for this set in a moment.
	First, we rewrite the moment we wish to calculate as
    \[
      \FMC{\Bias}[ \CUEF(\gamma(t)) - \CUEF(\gamma(0)) ]
      =
      \sum_{i=1}^{\ell}
      \FMC{\Bias}[ \CUEF(\gamma(t_i)) - \CUEF(\gamma(t_{i-1})) ].
    \]
    We then estimate the increment above and below by the inequality 
    \(
    1-e^{-2x} \leq 2x \leq e^{2x} - 1,
    \)
    valid for all real $x.$

    For the upper estimate, we therefore have
    \[
      \FMC{\Bias}[ \CUEF(\gamma(t)) - \CUEF(\gamma(0)) ]
      \leq
      \sum_{i=1}^{\ell}
      \FMC{\Bias}[ 2^{-1}(e^{2\CUEF(\gamma(t_i)) - 2\CUEF(\gamma(t_{i-1}))} - 1) ].
    \]
    We will actually show that
    \begin{equation}
        T:=
        \sum_{i=1}^{\ell}
      \FMC{\Bias}[ 2^{-1}(e^{2\CUEF(\gamma(t_i)) - 2\CUEF(\gamma(t_{i-1}))} - 1) ] = (t - b_1 + n_0)_+
      +O( (\log n) \rho_n^{1/2} ),
      \label{eq:fmcmean2}
    \end{equation}
    (noting the equality rather than the inequality) which we will use at a later point.

    Let $F = \one[{\BS \cap \ES}].$
    We can expand one of these increments as
    \begin{align*}
      \FMC{\Bias}[ e^{2\CUEF(\gamma(t_i)) - 2\CUEF(\gamma(t_{i-1}))} - 1 ]
      =&\frac{\Exp\left[ F(\CUEFr+\WN_r)e^{\Bias{\CUEFr} + 2\CUEF(\gamma(t_i)) - 2\CUEF(\gamma(t_{i-1}))} \right]
    }{ 
      \Exp\left[ F(\CUEFr+\WN_r)e^{\Bias{\CUEFr}} \right]}
      -1.
    \end{align*}
    Let $\mu : \D \to \R$ be the mean of $\GF_r$ under the bias $e^{\Bias{\GF_r}},$ and let $\mu^{(i)}$ be the mean under the bias $e^{\Bias{\GF_r}+2\GF(\gamma(t_i)) - 2\GF(\gamma(t_{i-1}))}.$
    Suppose that $t_i$ are chosen so that 
    \[
      \F \mapsto \Bias{\F_r}[2]+2\F(\gamma(t_i)) - 2\F(\gamma(t_{i-1})) \in W_{k+1}(q')
    \]
    for some $q'.$
    Then, we can apply Proposition~\ref{prop:clt} to write
    \begin{align*}
      &\FMC{\Bias}[ e^{2\CUEF(\gamma(t_i)) - 2\CUEF(\gamma(t_{i-1}))} - 1 ] \\
      &=\frac{\Exp\left[ F(\GF_r+\WN_r+\mu^{(i)})\right]\Exp\left[e^{\Bias{\CUEFr} + 2\CUEF(\gamma(t_i)) - 2\CUEF(\gamma(t_{i-1}))} \right]
    }{ 
      \Exp\left[ F(\GF_r+\WN_r+\mu)\right]\Exp\left[e^{\Bias{\CUEFr}} \right]}
      -1 + O(n_0^{-K}).
    \end{align*}
    Applying Corollary~\ref{cor:exact} and \eqref{eq:varbound}, we have that for $p_0$ sufficiently large
    \[
      \frac{\Exp\left[e^{\Bias{\CUEFr} + 2\CUEF(\gamma(t_i)) - 2\CUEF(\gamma(t_{i-1}))} \right]}
      {\Exp\left[e^{\Bias{\CUEFr}} \right]}
      \ll
      \frac{\Exp\left[e^{\Bias{\GF_r} + 2\GF(\gamma(t_i)) - 2\GF(\gamma(t_{i-1}))} \right]}
      {\Exp\left[e^{\Bias{\GF_r}} \right]}
      \ll 1.
    \]
    We therefore conclude that 
    \begin{align*}
      T
      =
      &O\biggl(1+ \sum_{i=1}^{\ell}
      \biggl|
      \frac
      {\Exp\bigl[ F(\GF_r+\WN_r+\mu^{(i)})\bigr]}
      {\Exp\left[ F(\GF_r+\WN_r+\mu)\right]}
      -1\biggr|
      \biggr)
      \\
      +&\sum_{i=1}^{\ell}
      \frac12
      \biggl[
	\frac{\Exp\left[e^{\Bias{\CUEFr} + 2\CUEF(\gamma(t_i)) - 2\CUEF(\gamma(t_{i-1}))} \right]}
	{\Exp\left[e^{\Bias{\CUEFr}} \right]}
      -1\biggr],
    \end{align*}
    provided ${\ell} \leq n_0^{K}.$
    Applying Lemma~\ref{lem:fmc1} we have that
    \[
      \sum_{i=1}^{\ell}
      \biggl|
      \frac
      {\Exp\bigl[ F(\GF_r+\WN_r+\mu^{(i)})\bigr]}
      {\Exp\left[ F(\GF_r+\WN_r+\mu)\right]}
      -1\biggr|
      \ll {\ell}\rho_n.
    \]

    We evaluate the exponential moments by comparison with $\GF.$  In the notation of Lemma~\ref{lem:fmc0}, write
    \[
      \frac
      {
	\Exp\left[e^{\Bias{\CUEFr} + 2\CUEF(\gamma(t_i)) - 2\CUEF(\gamma(t_{i-1}))} \right]
      }
      {
	\Exp\left[ 
	  e^{\Bias{\CUEFr}}
	\right]
      }
      =
      \frac
      {
	\Exp\left[e^{\Bias{\GF_r}+ 2\GF(\gamma(t_i)) - 2\GF(\gamma(t_{i-1}))} \right]
      }
      {
	\Exp\left[ 
	  e^{\Bias{\GF_r}}
	\right]
      }+O(\epsilon_i + \Delta_i),
    \]
    where $\epsilon_i$ is given by
    \[
      \epsilon_i = 
      \frac{
	\tanh\left((n-p-n_0 - t_{i-1})/2\right)^2
      } 
      {
	\tanh\left((n-p-n_0 - t_{i})/2\right)^2
      } 
      -1
      =O(e^{-n+n_0+p+t_i} |t_i - t_{i-1}|)
      ,
    \]
    and where in the notation of the definition of $W_k(q),$
    \begin{equation}
      \Delta_i =\max_{z \in \mathbf{z} \cup \{\gamma(t_i)\}} 
      e^{-N\exp(-\dH(0,z))}
      \prod_{
	\substack{x \in  \mathbf{z}' \cup \{\gamma(t_i)\} \\ 
      x \neq z}} 
      \coth(\dH(x,z)/2).
      \label{eq:fmccond2}
    \end{equation}

    The $\epsilon_i$ always sums over $i$ to $O(1)$ provided $\max_{1 \leq i \leq {\ell}}|t_i-t_{i-1}| \ll 1.$  The $\Delta_i$ term is negligible provided we choose $\left\{ t_i \right\}$ appropriately.  Define a sequence $\left\{ t_i' \right\}_1^{\ell}$ by setting $t_0' = 0$ and $t_i' - t_{i-1}'=\rho_n^{1/2}$ so long as $t_{i-1}' < n-n_0+\log \rho_n.$  To define the remainder of the sequence, let $u = [ (\rho_n^{-1/2})],$ and let the spacings be $u^{-1}, (u-1)^{-1}, \dots, 3^{-1},2^{-1},1.$  Then for $p_0$ and $n_0$ sufficiently large, this definition assures that the sequence reaches $t$ in finitely many steps, i.e.\ there is a finite ${\ell}$ (in fact ${\ell}=O((\log n_0 )\rho_n^{-1/2})$) so that $t_{\ell}' \leq t < t_{{\ell}+1}'.$ Then, we set 
		\corO{ $t_i = t_i' (t/t_{\ell}'),$} so that the spacings are at least those of stated and the endpoint is at $t$).

    We must control $\Delta_i$ along this sequence, which will be impossible for points $t_i$ where $\gamma(t_i)$ is close to an element of $\mathbf{z}'.$  So, for the moment, assume that $\gamma(t_i)$ is at least distance $1$ from any point of $\mathbf{z}.$  We will then show that exceptional $\gamma(t_i)$ can not influence the sum much.  Note that for such $t_i,$ we have that
    \[
      \F \mapsto \Bias{\F_r}[2]+2\F(\gamma(t_i)) - 2\F(\gamma(t_{i-1})) \in W_{k+1}(q')
    \]
    for some $q',$ which can be controlled solely in terms of $k$ and $q.$

    For the initial steps in the sequence (where the points are evenly spaced), we have that $\Delta_i \ll e^{-\Omega((\rho_n)^{-1/2})}.$ For $i$ in the final stretch, we have that $\Delta_{{\ell}-i}$ decays exponentially $i.$ Hence, with this choice of $t_i,$ we conclude that 
    \begin{align*}
      T
      =
      O( (\log n )\rho_n^{1/2})+
      \sum_{i=1}^{\ell}
      \frac12
      \biggl[
	\frac{\Exp\left[e^{\Bias{\GF_r} + 2\GF(\gamma(t_i)) - 2\GF(\gamma(t_{i-1}))} \right]}
	{\Exp\left[e^{\Bias{\GF_r}} \right]}
      -1\biggr].
    \end{align*}
    By the change of mean formula (Lemma~\ref{lem:means}), we thus have
    \begin{align*}
        T
      =
      O( (\log n )\rho_n^{1/2})+
      \sum_{i=1}^{\ell}
      \frac12
      \biggl[
	e^{2\mu(\gamma(t_i)) - 2\mu(\gamma(t_{i-1})) + 2\sigma_i^2
      }
    -1\biggr],
  \end{align*}
  where
  $\sigma_i^2 = \Var
  \left( 
  \GF(\gamma(t_i)) - \GF(\gamma(t_{i-1}))
  \right).$
  Now $\mu(\gamma(t_i)) - \mu(\gamma(t_{i-1})) = O(|t_{i}-t_{i-1}|)$ and $\sigma_i = O(|t_{i}-t_{i-1}|).$  Hence, 
  \begin{align*}
      T
      =
    O( (\log n )\rho_n^{1/2})+
    \mu(\gamma(t))
    -\mu(\gamma(0))
    +
    O\biggl(
    \sum_{i=1}^{\ell}
    |t_{i}-t_{i-1}|^2\biggr)
  \end{align*}
  The contribution to the sum of the increments that are $O(\rho_n^{-1/2})$ is $O( (\log n) \rho_n^{1/2}).$ For the other terms, we have $|t_{{\ell}-i}-t_{{\ell}-i-1}|=O(i^{-2}).$  Hence their contribution is at most $O(1),$ and we conclude
  \begin{align*}
      T
      =
    O( (\log n )\rho_n^{1/2})+
    \mu(\gamma(t)) - \mu(\gamma(0)).
  \end{align*}
  That $\mu(\gamma(t)) - \mu(\gamma(0)) = (t - b_1 + n_0)_+ + O(1)$ follows from Lemma~\ref{lem:branch2} and Lemma~\ref{lem:means} (c.f.~Lemma~\ref{lem:fmc1}).

  We now turn to controlling the contribution of exceptional $t_i,$ that is where the distance of $t_i$ to some point in $\mathbf{z}'$ is less than $1$.  First we note that by construction, there are at most $O(\rho_n^{-1/2})$ many of these, with the implicit constant in the $O(\cdot)$ depending only on $k.$  Suppose $t_i$ is any such increment.  By virtue of Proposition~\ref{prop:domination}, Corollary~\ref{cor:meso}, and Corollary~\ref{cor:exact},
  \[
    \FMC{\Bias}[ e^{\lambda(\CUEF(\gamma(t_i)) - \CUEF(\gamma(t_{i-1})))}]
    \ll n_0^{3/2} \Exp\left[ 
      e^{\lambda(\GF(\gamma(t_i)) - \GF(\gamma(t_{i-1}))+\mu(\gamma(t_i)
			\corO{)}- \mu(\gamma(t_{i-1})))}
    \right]
  \]
  for all $\lambda \in \R.$  The variance of the increment $\GF(\gamma(t_i)) - \GF(\gamma(t_{i-1}))$ is at most $1.$  Hence, this translates into Gaussian decay of the increment, once the increment is larger than a sufficiently large multiple of $\sqrt{\log n}.$
Hence for some $C>0$ sufficiently large,
\begin{eqnarray*}
   && \!\!\!\!\FMC{\Bias}\biggl[ |\CUEF(\gamma(t_i)) - \CUEF(\gamma(t_{i-1}))|\one[
    |\CUEF(\gamma(t_i)) - \CUEF(\gamma(t_{i-1}))|> C(\log n)^{1/2}]
  \biggr]\\
  &&\quad \quad
  \leq e^{-\Omega(\log n)}.
\end{eqnarray*}
Summing all of these exceptional increments therefore only contributes at most $O( (\log n)^{1/2}\rho_n^{-1/2}),$ which is negligible in comparison to the error claimed in the lemma.
\end{proof}

Using this machinery, we now estimate the second moment of such an increment.
\begin{lemma}
  There is a $p_0$ sufficiently large and independent of $n$ so that
  for all $p_0 \leq p \leq \log\log n$
  and for all $0 \leq t \leq n-p-n_0,$
  \[
    \FMC{\Bias[2]}[ \left(\CUEF(\gamma(t)) - \CUEF(\gamma(0))\right)^2 ] \leq 
    (t - b_1 + n_0)_+^2
    +\xi(t)
  \]
  where $|\xi(t)| \ll (\log n)^2 \rho_n.$
  \label{lem:fmc2nd}
\end{lemma}
\begin{proof}
  Let $\{t_0, t_1,t_2,t_3,\dots, t_{\ell}\}$ be the same sequence of points chosen in the proof of Lemma~\ref{lem:fmcmean}.  We rewrite the moment we wish to calculate as
  \begin{eqnarray*} 
  &&  
  \!\!\!\!\!\!\!
  \!\!\!\!\!\!\!
  \FMC{\Bias[2]}[ \CUEF(\gamma(t)) - \CUEF(\gamma(0)) ]\\
  &&\quad
    =
    \sum_{i=1}^{\ell}
    \FMC{\Bias[2]}[ \left(\CUEF(\gamma(t_i)) - \CUEF(\gamma(t_{i-1}))\right)
    \left(\CUEF(\gamma(t)) - \CUEF(\gamma(0))\right)].
    \end{eqnarray*} 
  
  We will approximate the increment $\CUEF(\gamma(t_i)) - \CUEF(\gamma(t_{i-1}))$ by 
  \[
    2^{-1}(e^{2\CUEF(\gamma(t_i)) - 2\CUEF(\gamma(t_{i-1}))} -1).
  \]
  We first show that after the approximation has been made, we get the desired result.
  Having done the approximation, we will be in a position to apply Lemma~\ref{lem:fmcmean}.  Again we may have an exceptional $t_i,$ this time occurring only at the terminal endpoint.  This can be controlled in the same manner as was done in Lemma~\ref{lem:fmcmean}. For the unexceptional points, we have that for some $q \in \R,$ the map
  \[
    \Bias: \F \mapsto \Bias{\F_r}[2]+2\F(\gamma(t_i)) - 2\F(\gamma(t_{i-1})) \in W_{1}(q).
  \]
  Hence, we can apply Lemma~\ref{lem:fmcmean} to $\Bias$ (and to $\Bias[2]$) to get
  \[
    \FMC{\Bias}[ \CUEF(\gamma(t)) - \CUEF(\gamma(0))] = 
    (t - b_1 + n_0)_+ + O((\log n) \rho_n^{1/2}).
  \]

  To compare this to $\FMC{\Bias[2]},$ we write
  \begin{multline*}
    \frac
    {\FMC{\Bias[2]}[ e^{\Bias{\CUEFr} - \Bias{\CUEFr}[2]}(\CUEF(\gamma(t)) - \CUEF(\gamma(0)))]}
    {\FMC{\Bias}[ \CUEF(\gamma(t)) - \CUEF(\gamma(0))]} \\
    =\frac
    {\Exp\left[ F(\CUEFr+\WN_r)e^{\Bias{\CUEFr}[2] + 2\CUEF(\gamma(t_i)) - 2\CUEF(\gamma(t_{i-1}))} \right]}
    {\Exp\left[ F(\CUEFr+\WN_r)e^{\Bias{\CUEFr}[2]} \right]}.
  \end{multline*}
  Using the same reductions as used in the proof of Lemma~\ref{lem:fmcmean}, we get that this ratio is
  \begin{eqnarray*}
    &&
    \!\!\!\!\!\!\!
    \!\!\!\!\!\!\!
    \!\!\!\!\!\!\!
    \!\!\!\!\!\!\!
    \!\!\!\!\!\!\!
    \!\!\!\!\!\!\!
    \frac
    {\Exp\left[ F(\CUEFr+\WN_r)e^{\Bias{\CUEFr}[2] + 2\CUEF(\gamma(t_i)) - 2\CUEF(\gamma(t_{i-1}))} \right]}
    {\Exp\left[ F(\CUEFr+\WN_r)e^{\Bias{\CUEFr}[2]} \right]}\\
    &&\quad =1+ 2(\mu(\gamma(t_i)) - \mu(\gamma(t_{i-1}))) + \xi_i,
  \end{eqnarray*}
  where $\mu$ is the mean of $\GF_r$ under the bias $e^{\Bias{\GF_r}[2]},$ and the error term $\xi_i$ satisfies 
  \[
    \sum_{i=1}^\ell |\xi_i| \ll (\log n )\rho_n^{1/2}.
  \]
  Hence
  \begin{multline}
    \sum_{i=1}^{\ell}
    \FMC{\Bias[2]}[ 
      2^{-1}(e^{2\CUEF(\gamma(t_i)) - 2\CUEF(\gamma(t_{i-1}))} -1)
    \left(\CUEF(\gamma(t)) - \CUEF(\gamma(0))\right) ] \\
    = (\mu(\gamma(t)) - \mu(\gamma(0)))(t - b_1 + n_0)_+ 
    +O((\log n )^2\rho_n).
    \label{eq:fmc21}
  \end{multline}
  Using that $\mu(\gamma(t)) - \mu(\gamma(0)) = (t - b_1 + n_0)_+ + O(1),$ we arrive at the desired conclusion, modulo having established the validity of the approximation.

  Turning to establishing the approximation,
  define 
  \[
    R_i 
    =-2^{-1}(e^{2\CUEF(\gamma(t_i)) - 2\CUEF(\gamma(t_{i-1}))} -1)
    +\left( 
    \CUEF(\gamma(t_i)) - \CUEF(\gamma(t_{i-1}))
    \right).
  \]
  Since $1+2x \leq e^{2x}$ for all $x\in \R,$ we have that
  \(R_i \leq 0 \)
  almost surely.

  In terms of $R_i,$ we wish to show that
  \[
    \xi := \sum_{i=1}^\ell \FMC{\Bias[2]}[ R_i \left(\CUEF(\gamma(t)) - \CUEF(\gamma(0))\right)]
    \ll (\log n)^2\rho_n.
  \]
  This we expand into a double sum.  For each $i,$ we estimate
  \begin{align*}
    \FMC{\Bias[2]}[ R_i \left(\CUEF(\gamma(t)) - \CUEF(\gamma(0))\right)] 
    &=\sum_{j=1}^\ell
    \FMC{\Bias[2]}[ R_i \left(\CUEF(\gamma(t_j)) - \CUEF(\gamma(t_{j-1}))\right)] \\
    &
   \!\!\!\!\!\!\!\!
   \!\!\!\!\!\!\!\!
   \!\!\!\!\!\!\!\!
   \leq\sum_{j=1}^\ell
    \FMC{\Bias[2]}[ R_i 2^{-1}(1-e^{-2\CUEF(\gamma(t_j)) + 2\CUEF(\gamma(t_{j-1}))})].
  \end{align*}
  Hence, we can estimate $\xi$ by
  \begin{align*}
      \xi
      \leq &
      \sum_{i=1}^\ell
      \sum_{j=1}^\ell
      \FMC{\Bias[2]}[
      \left( 
      \CUEF(\gamma(t_i)) - \CUEF(\gamma(t_{i-1}))
      \right)
      2^{-1}(1-e^{-2\CUEF(\gamma(t_j)) + 2\CUEF(\gamma(t_{j-1}))})]\\
      -&
      \sum_{i=1}^\ell
      \sum_{j=1}^\ell
      \FMC{\Bias[2]}[ 4^{-1}
      (e^{2\CUEF(\gamma(t_i)) - 2\CUEF(\gamma(t_{i-1}))}-1)
      (e^{2\CUEF(\gamma(t_j)) - 2\CUEF(\gamma(t_{j-1}))}-1)].
  \end{align*}
  The first line on the right hand side can be estimated exactly as in \eqref{eq:fmc21}.  The second line, on holding $i$ fixed and summing over $j$ can be estimated using \eqref{eq:fmcmean2} and the same comparison between $\FMC{\Bias}$ and $\FMC{\Bias[2]}$ done in the first part of this proof.  Combining both pieces, we get that
\begin{eqnarray*}
      \xi
     & \leq &(t - b_1 + n_0)_+^2 +O((\log n )^2\rho_n)
      \\
      &&-\sum_{i=1}^\ell
      [\mu(\gamma(t_i)) - \mu(\gamma(t_{i-1}))](t - b_1 + n_0)_+
      +O((\log n )^2\rho_n) 
      \ll(\log n )^2\rho_n.
\end{eqnarray*}
\end{proof}

\subsection{Estimating $\FMC[ \one[Z > 0]]$}
\label{sec:2p1}

Using these field moments, we can estimate the conditional probability of $\BL(\theta).$
\corO{Recall \eqref{eq:invariance}.}
\begin{lemma}
  For all $\epsilon > 0,$ there is a $p_0(\epsilon) > 0$ sufficiently large that for all $\log\log n > p \geq p_0(\epsilon)$ all $n$ sufficiently large, 
  and all $|\theta|<\Xi,$
  \[
    e^{n-p-b_1-2\log 2}(1-\epsilon)\leq \FMC[\MFI(\theta)] \leq e^{n-p-b_1-2\log2}(1+\epsilon).
  \]
  \label{lem:1pLB}
\end{lemma}
\begin{proof}
    Unpacking the meaning of $\FMC[\MFI(\theta)]$
  \[
    \FMC[\MFI(\theta)]
    =
    \frac{
      \Exp\left[ \one[\Meso] \one[{\CUEFr + \WN_r \in \BL[\theta]\cap\ELL[\theta]}]e^{\Bias{\CUEFr}[2]}\right]}
      {
	\Exp\left[ \one[\Meso] e^{\Bias{\CUEFr}[1]}\right]
      }
    \]
    where $\Bias[1]$ and $\Bias[2]$ are as in Corollary~\ref{cor:meso}.
    Applying Corollary~\ref{cor:meso} and the trivial supremum bound on the indicator, we get that
    \[
      \FMC[\MFI(\theta)]
      \leq \frac{\Exp\left[e^{\Bias{\CUEFr}[2]}\right]}
      {\Exp\left[ e^{\Bias{\CUEFr}[1]}\right]}(1+o(1)).
    \]
    Hence by Corollary~\ref{cor:exact}, there is a $p_0(\epsilon)$ sufficiently large so that for all $p_0(\epsilon)\leq p < \log\log n,$ we have
    \[
      \FMC[\MFI(\theta)]
      \leq \frac{\Exp\left[e^{\Bias{\GF_r}[2]}\right]}
      {\Exp\left[ e^{\Bias{\GF_r}[1]}\right]}(1+\epsilon + o(1)).
    \]
    It remains to evaluate the variances of these biasing terms.  

    The variance of $\Bias{\GF_r}[1]$ is simpler.  Since the distance between $r$ and $n_0$ is going to infinity, we have that
    \[
      \tfrac{1}{2}\Var\left( 
      \Bias{\GF_r}[1]
      \right)
      = n_0 - r - 2\log 2 + o(1).
    \]
    For $\Bias{\GF_r}[2],$ we use Lemma~\ref{lem:branch2}.  We start by observing that $\Bias{\GF_r}[2] = \Bias{\GF}[2] - 2\GF(\zeta_r).$  Since all pairwise separations go to infinity, the variance can be computed by comparison with simple random walk up to a $o(1)$ additive correction.  To see this, first split the variance $\Bias{\GF}[2] - 2\GF(\zeta_r)$ into the two increments which would be independent in the random walk case; the cross correlations and their additive corrections cancel: 
    \begin{align*}
      \tfrac{1}{2}\Var\left( 
      \Bias{\GF}[2]- 2\GF(\zeta_r)
      \right)
      = 
      &\tfrac{1}{2}\Var\left( 
      2\GF(\ROT[\theta](\zeta_{n-p}))- 2\GF(\ROT[\theta](\zeta_{b_1}))
      \right) \\
      +&\tfrac{1}{2}\Var\left( 
      2\GF(\zeta_{n_0})- 2\GF(\zeta_r)
      \right) +o(1) \\
      =&(n - p - b_1) + (n_0 -r) - 4\log 2 + o(1).
    \end{align*}

    The lower bound relies on calculating some field moments.  \corO{Fix w.l.o.g. in the rest of
		the proof $\theta=0$ and omit it from the notation.}
		We would like to control the $\FMC$ probability of $\CUEFr+\WN_r$ not being in \corO{$\BL \cap \ELL$} under an additional bias.  \corO{Recall first that the event 
		  $\Meso$, see \eqref{eq-rev1}, implies that $|(\CUEFr+\WN_r)(\zeta_{n_0})
		-n_0-t_0|\leq 1$. 
		On  $\Meso$ we therefore obtain that for some constant $C=C(p)$, 
		$$(\ELL)^c\subset 
		\{(\CUEFr+\WN_r)(\eta_{b_\eta})-(\CUEFr+\WN_r)(\eta_{n_0})-b_{\eta-1}+n_0| \corE{> C \eta^{-1} \log n}
		\}.$$ Arguing similarly on the events composing $(\BL)^c$, applying the union bound and then Chebychev's bound, we obtain}
	\begin{align*}
      &\FMC\left[ \one[(\BL)^c \cup \corO{(\ELL)^c}] 
			(\CUEFr+\WN_r) 
      e^{\Bias{\CUEFr}[2]-\Bias{\CUEFr}[1]}
    \right]\\
    &\ll
    \sum_{j=1}^\eta
    \FMC\left[
      e^{\Bias{\CUEFr}[2]-\Bias{\CUEFr}[1]}
      \frac{\eta^2 ((\CUEFr+\WN_r)(\zeta_{b_j}) - (\CUEFr+\WN_r)(\zeta_{n_0})-b_{j-1}+n_0)^2}{(\log n)^2}
    \right]
  \end{align*}
  By construction of the fields $\CUEFr$ and by Lemma~\ref{lem:lemonslice}, $\CUEFr(\zeta_{b_j}) - \CUEFr(\zeta_{n_0}) = \CUEF(\zeta_{b_j}) - \CUEF(\zeta_{n_0}).$  Hence, we can apply Lemma~\ref{lem:fmc2nd} to compute this second moment, but first we integrate the $\WN_r$ terms out of the second moment.  For $1 \leq j < \eta,$ these terms are independent of all other $\WN_r$ terms that appear in the same expectation, and so we can use that $\Exp( a+Z)^2 = \Exp a^2 + \Exp Z^2$ for a centered variable $Z$ independent of $a.$  In fact, this additional contribution to the second moment is much smaller in order than the second moment of the differences of the $\CUEF.$  
  In all, we get:
  \begin{align*}
    &\FMC\left[ \one[(\BL)^c \cup \corO{(\ELL)^c} 
		](\CUEFr+\WN_r) 
    e^{\Bias{\CUEFr}[2]-\Bias{\CUEFr}[1]}
  \right]\\
  &\ll
  \sum_{j=1}^\eta
  \FMC\left[ 
    e^{\Bias{\CUEFr}[2]-\Bias{\CUEFr}[1]}
    \frac{\eta^2 (\CUEFr(\zeta_{b_j}) - \CUEFr(\zeta_{n_0})-b_{j-1}+n_0)^2 + O(\eta^2)}{(\log n)^2}
  \right] \\
  &\ll
  \sum_{j=1}^\eta
  \FMC\left[ 
    e^{\Bias{\CUEFr}[2]-\Bias{\CUEFr}[1]}
    \frac{\eta^2 (\log n)^2\rho_n}{(\log n)^2}
  \right] 
  \ll
  \eta^3 \rho
  \cdot
  \FMC\left[  
    e^{\Bias{\CUEFr}[2]-\Bias{\CUEFr}[1]}
  \right].
\end{align*}
Thus, as $\eta^3 \rho \to 0,$ we conclude that
\[
  \FMC[\MFI(\theta)]
  \geq\FMC\left[  
    e^{\Bias{\CUEFr}[2]-\Bias{\CUEFr}[1]}
  \right]
  (1-o(1)).
\]
Hence, again by Corollary~\ref{cor:exact}, the lower bound follows.
\end{proof}

We now turn to estimating $\FMC[\MFI(\theta_1)\MFI(\theta_2)]$ for various values of $(\theta_1,\theta_2).$  There will be two regimes of $|\theta_1-\theta_2|$ in which we make different estimates.  We introduce the midpoint $\HM=\HM(\theta_1,\theta_2)=n_0- [\log |\sin \tfrac{\theta_1-\theta_2}{2}|],$ with $[\cdot]$ denoting integer part.  This is roughly the height \corO{at which 
 the $\theta_1$ and $\theta_2$ rays branch.}  In the first regime, where $\HM < b_1,$ the rays have branched early enough that there is essentially no correlation between $\MFI(\theta_1)$ and $\MFI(\theta_2).$  Otherwise, we must appropriately take advantage of the barrier information in $\MFI(\theta)$ to assure the correlation is not too high.

The estimate for small $\HM$ is no more complicated than the estimates in Lemma~\ref{lem:1pLB}.
\begin{lemma}
  For all $\epsilon > 0$ there is a $p_0(\epsilon) >0$ independent of $n$ so that for all $\log\log n > p > p_0$ and all $n$ sufficiently large the following holds.
  Suppose $|\theta_1|,|\theta_2| < \Xi$ are such that $\HM(\theta_1,\theta_2) \leq n_0 + (1-\epsilon)(b_1 - n_0).$  Then
  \[
    \FMC[\MFI(\theta_1)\MFI(\theta_2)]
    \leq
    \FMC[\MFI(\theta_1)]\FMC[\MFI(\theta_2)](1+\epsilon).
  \]
  \label{lem:finefield2pUB1}
\end{lemma}
\begin{proof}
  We estimate the left hand side by the trivial bound
  \[
    \FMC[\MFI(\theta_1)\MFI(\theta_2)]
    \leq
    \frac{
      \Exp\left[ \one[\Meso] e^{\Bias{\CUEFr}[3]}\right]
    }
    {
      \Exp\left[ \one[\Meso] e^{\Bias{\CUEFr}[1]}\right]
    }
		\corO{=
		\frac{
      \Exp\left[  e^{\Bias{\CUEFr}[3]}\right]
    }
    {
      \Exp\left[  e^{\Bias{\CUEFr}[1]}\right]
    }(1+O(\rho_n))
		}
		,
  \]
  where $\Bias[3]$ is as in Corollary~\ref{cor:meso}, 
	\corO{and the corollary was used in the second equality}. 
	Applying Corollary~\ref{cor:exact}, we can estimate this by the same with $\CUEFr$ replaced by $\GF_r.$  The variance of $\Bias{\GF_r}[1]$ was already computed in Lemma~\ref{lem:1pLB}.  The variance of $\Bias{\GF_r}[3]$ remains to be calculated.  As with the calculation of $\Bias{\GF_r}[2],$ 
  \(
  \Bias{\GF_r}[3]
  =
  \Bias{\GF}[3] - 2\GF(\zeta_r).
  \)
  The important point here is that since $\HM(\theta_1,\theta_2) \leq
  n_0 + (1-\epsilon)(b_1 - n_0),$ the correlation between the segment $\GF(\ROT[\theta_1](\zeta_{n-p}))- \GF(\ROT[\theta_1](\zeta_{b_1}))$ and the segment $\GF(\ROT[\theta_2](\zeta_{n-p}))- \GF(\ROT[\theta_2](\zeta_{b_1}))$ decays like $O(e^{-\Omega(b_1-n_0)})$ (by Lemma~\ref{lem:branch2}).  The other correlations between these segments and the earlier $\GF(\zeta_{n_0})-\GF(\zeta_r)$ segment also decay due to being well separated, and so
  \[
    \tfrac{1}{2}\Var\left( 
    \Bias{\GF}[3] - 2\GF(\zeta_r)
    \right)
    =2(n - p - b_1) + (n_0 -r) - 6\log 2 + o(1).
  \]
  Therefore, we have shown the estimate
  \[
    \FMC[\MFI(\theta_1)\MFI(\theta_2)]
    \leq e^{2(n - p - b_1) - 4\log 2}(1+\epsilon+o(1)).
  \]
  Applying Lemma~\ref{lem:1pLB}, the claim now follows.
\end{proof}

This lemma covers all but a vanishing fraction of pairs $(\theta_1,\theta_2)$ we need to consider.  However,  we must also assure that the remaining terms are not too correlated.  This is the content of the following lemma.
\begin{lemma}
  There is a $p_0$ sufficiently large and independent of $n$ so that
  for all $ \log\log n > p \geq p_0$ and all $n$ sufficiently large the following holds.
  Suppose $|\theta_1|,|\theta_2| < \Xi$ are such that $\HM(\theta_1,\theta_2) < n-p.$  Then
  \[
    \FMC[\MFI(\theta_1)\MFI(\theta_2)]
    \ll
    \FMC[\MFI(\theta_1)]\FMC[\MFI(\theta_2)]
    e^{\HM-b_1 + \tfrac{m+4}{2\eta}\log n}.
  \]
  \label{lem:finefield2pUB2}
\end{lemma}
\begin{proof}
    Unlike when $\HM$ was small, in the setting of the previous lemma, $\Bias[3],$ which is the biasing term that appears in the left-hand side, will have much too large a variance.  This is because, by analogy with branching random walk, $\Bias[3]$ overweights the segment before the $\theta_1$ and $\theta_2$ rays split.  Hence, we would like to re-bias the exponential weight on the left-hand side.  Ideally we would replace $\ROT[\theta_1](\zeta_{b_1})$ with $\ROT[\theta_1](\zeta_{\HM})$ in the biasing term.  However, we have no control on the value of $\ROT[\theta_1](\zeta_{\HM}),$ and so we instead choose an approximation over which we do.  To this end, let $b_* \in \left\{ b_1,b_2,\dots,b_\eta \right\}$ be the closest element to $\HM,$ so that \(|\HM - b_*| \leq (2\eta)^{-1}{m\log n}.\)

  Let $\Bias{\F}[\theta] =  2\F(\ROT[\theta_1](\zeta_{b_1})) - 2\F(\ROT[\theta_1](\zeta_{b_*})).$  When $\F \in \BL[\theta_1] \cap \ELL[\theta_1],$ we have that $\F$ is constrained at both of these points.  Specifically, letting
  \[
      q =  (b_1 - b_*) - 2\eta^{-1}\log(n),
  \]
  we have that $\Bias{\F}[\theta] \geq 2q$ for $\F \in \ELL[\theta_1].$  Hence,
  \[
    \FMC[\MFI(\theta_1)\MFI(\theta_2)]
    \leq
    \FMC[\MFI(\theta_1)\MFI(\theta_2)e^{\Bias{\CUEFr+\WN_r}[\theta] - 2q}].
  \]
  We can now drop the indicators of the microscopic field events and estimate
  \[
    \FMC[\MFI(\theta_1)\MFI(\theta_2)]
    \leq 
    \frac{
      \Exp\left[ \one[\Meso] e^{\Bias{\CUEFr}[3]+\Bias{\CUEFr+\WN_r}[\theta] - 2q}\right]
    }
    {
      \Exp\left[ \one[\Meso] e^{\Bias{\CUEFr}[1]}\right]
    }.
  \]
  \corE{By the same argument as in Corollary~\ref{cor:meso}, and by the independence of $\WN_r$ from $\CUEFr,$ 
  \[
    \FMC[\MFI(\theta_1)\MFI(\theta_2)]
    \ll
    \frac{
      \Exp\left[ e^{\Bias{\CUEFr}[3]+\Bias{\CUEFr+\WN_r}[\theta] - 2q}\right]
    }
    {
      \Exp\left[ e^{\Bias{\CUEFr}[1]}\right]
    }.
  \]
  We no longer need to make a precise upper estimate.  Hence, for any $\log\log n \geq p \geq 0,$ we can replace $\CUEFr$ by $\GF_r$ up to a multiplicative loss by an absolute constant:
  \[
    \FMC[\MFI(\theta_1)\MFI(\theta_2)]
    \ll
    \frac{
      \Exp\left[ e^{\Bias{\GF_r}[3]+\Bias{\GF_r+\WN_r}[\theta] - 2q}\right]
    }
    {
      \Exp\left[ e^{\Bias{\GF_r}[1]}\right]
    }.
  \]
  The dependence on $\WN_{r}$ in the bias can be integrated out, as it is independent.  This produces only a loss of a multiplicative factor, on account of the uniform boundedness of the variances of $\WN_r.$ Hence
  \[
    \FMC[\MFI(\theta_1)\MFI(\theta_2)]
    \ll
    \frac{
      \Exp\left[e^{\Bias{\GF_r}[3]+\Bias{\GF_r}[\theta] - 2q}\right]
    }
    {
      \Exp\left[e^{\Bias{\GF_r}[1]}\right]
    },
  \]
  and we turn to estimating the variance of the bias in the numerator.}

  We use Lemma~\ref{lem:branch2} to compare the variance of this bias term to branching random walk.  Here we only need the value up to an additive $O(1)$ constant, so we do not need to worry about distances between segments being macroscopic.  If $\HM \geq b_*,$ then $\Bias{\GF_r}[3]+\Bias{\GF_r}[\theta]$ counts the segment between $\HM$ and $b_*$ twice, and hence we have
  \[
    \tfrac{1}{2}\Var\left( 
    \Bias{\GF_r}[3]+\Bias{\GF_r}[\theta]
    \right)
    \leq 2(n - p - \HM) + 4(\HM-b_*) + (b_*-b_1) + (n_0 - r) + O(1).
  \]
  On the other hand, if $\HM < b_*,$ then $\Bias{\GF_r}[3]+\Bias{\GF_r}[\theta]$ corresponds to a sum of 3 independent random walk segments for a total variance of
  \[
    \tfrac{1}{2}\Var\left( 
    \Bias{\GF_r}[3]+\Bias{\GF_r}[\theta]
    \right)
    \leq (n - p - b_1) + (n-p-b_*) + (n_0 - r) + O(1).
  \]
  Combining these with $-2q,$ we have that
  \[
    \FMC[\MFI(\theta_1)\MFI(\theta_2)]
    \ll e^{2(n-p-b_1)+ (\HM - b_1) + |\HM-b_*| + 2\eta^{-1}\log n}.
  \]
  After applying Lemma~\ref{lem:1pLB} with $\epsilon = \tfrac 12$ to write this in terms of $\FMC[\MFI(\theta_i)],$ we have that
  \[
    \FMC[\MFI(\theta_1)\MFI(\theta_2)]
    \ll
    \FMC[\MFI(\theta_1)]\FMC[\MFI(\theta_2)]
    e^{\HM-b_1 + \tfrac{m+4}{2\eta}\log n}.
  \]
\end{proof}

We are now able to show the desired conditional lower bound.

\begin{proposition}
  For all $\epsilon > 0$ there is a $p_0(\epsilon) > 0$ so that for all $\log\log n > p > p_0$ and all $n$ sufficiently large
  \[
    \FMC[ \one[{\CUEFr + \WN_r \in \Ev}]] \geq 1-\epsilon.
  \]
  \label{prop:1pLB}
\end{proposition}
\begin{proof}
  Recall that we let 
  \[
    Z=  \sum_{\theta \in \Theta} \MFI(\theta), 
  \]
  where $\Theta =e^{-n+n_0}\mathbb{Z} \cap (-\Xi, \Xi)$
  Using \eqref{eq:2mm}, we bound 
  \begin{align*}
    \FMC[ \one[{\CUEFr + \WN_r \in \Ev}]]
    &=
    \FMC[ \one[Z>0]] 
    \geq 
    \frac
    { \bigl(\sum_{\theta}\FMC\left[ \MFI(\theta) \right]\bigr)^2}
    {\sum_{\theta_1, \theta_2}\FMC\left[\MFI(\theta_1)\MFI(\theta_2)\right]},
  \end{align*}
  with $\theta_1,\theta_2$ running over the set $\Theta.$  We now partition this sum according to the value of $\HM(\theta_1,\theta_2).$ We let $I_1$ be the sum
  \[
    I_1 = 
    \sum_{\substack{\theta_1, \theta_2 \\ \HM(\theta_1,\theta_2) \leq
      n_0 + (1-\epsilon)(b_1 - n_0)
    }}
    \FMC\left[\MFI(\theta_1)\MFI(\theta_2)\right],
  \]
  and we let $I_2$ be the sum over the remaining pairs $(\theta_1,\theta_2).$

  For a given $\theta_1$ and a given value of $\HM,$ there are at most $2\Xi e^{n-\HM}$ many integers $h_2$ so that $(\theta_1,e^{-n+n_0}h_2)$ attains this value of $\HM.$ 
  For $(\theta_1,\theta_2)$ such that $\HM \leq n_0 + (1-\epsilon)(b_1 - n_0),$ we have that by Lemma~\ref{lem:finefield2pUB1} that 
  \[
    \FMC\left[\MFI(\theta_1)\MFI(\theta_2)\right]\leq 
    (1+\epsilon)
    \FMC\left[\MFI(\theta_1)\right]\FMC\left[\MFI(\theta_2)\right]
  \]
  for all $p$ sufficiently large (independent of $n$) and all $n$ sufficiently large.
  Thus it follows that
  \[
    I_1 \leq (1+\epsilon)\bigl(\sum_{\theta}\FMC\left[ \MFI(\theta) \right]\bigr)^2.
  \]

  For larger $\HM,$ we apply Lemma~\ref{lem:finefield2pUB2} and sum.  Let
  \corO{$\ell_0=\lfloor
	n_0 + (1-\epsilon)(b_1 - n_0)\rfloor.$}
  \begin{align*}
    I_2=\sum_{\ell=\ell_0}^{n}
    \sum_{\substack{\theta_1,\theta_2 \in 
		\corO{\Theta}\\ \HM(\theta_1,\theta_2) = \ell}}
    \FMC\left[\MFI(\theta_1)\MFI(\theta_2)\right]
    \ll&\sum_{\ell=\ell_0}^{n}
    e^{n-n_0}e^{n-\ell}e^{\ell -b_1 +\tfrac{m+4}{2\eta}\log n}
    e^{2(n-p-b_1)} \\
    \ll&
    e^{n-n_0+n-b_1+\tfrac{m+4}{2\eta}\log n + O(\log\log n)}
    e^{2(n-p-b_1)},
  \end{align*}
  using that the number of terms in the sum is order $\log n.$ Finally, comparing this back to the sum of expectations squared
    \[I_2\leq
    e^{n_0-b_1+\tfrac{m+4}{2\eta}\log n + O(\log\log n)}
    \bigl(\sum_{\theta}\FMC\left[ \MFI(\theta) \right]\bigr)^2
  \]
  As $b_1 - n_0 = m\eta^{-1}\log n + O(1)$ and $m>100,$ this whole expression is 
	\corO{controlled by}
	$o\left(\bigl(\sum_{\theta}\FMC\left[ \MFI(\theta) \right]\bigr)^2\right).$
\end{proof}

\subsection{Lower bound proof}
\label{sec:2p}

We now proceed to the proof of the lower bound for the maximum of $\CUEF.$
This mirrors closely the proof that is given in Section~\ref{sec:2p1} for proving that $Z > 0$ conditioned on $\Meso.$  We again start by introducing biased indicators:
\begin{equation}
    \label{eq:FI}
  \FI(\omega) = e^{2\CUEFr(\zeta_{n_0})} \one[{\CUEFr + \WN_r \in \Ev(\omega)}].
\end{equation}
\corO{(These are not to be confused 
with the variables $W_k(q)$ of Section \ref{sec:fmc}.)}
We also introduce the ``counting'' variable
\[
  Z = \sum_{\omega} \FI(\omega),
\]
with the sum over the set $\mathbb{T}_{n_0} := \left\{ e^{2\pi i h \lfloor e^{n_0}\rfloor^{-1}}~:~h\in \mathbb{N}, h \leq e^{n_0} \right\},$ to which we will apply the second moment method.

By adjusting $p$ and by virtue of Proposition~\ref{prop:1pLB}, for any $\epsilon > 0$ we have that
\begin{equation}
  (1-\epsilon)\Exp\left[ \one[\Meso]  e^{2\CUEFr(\omega\zeta_{n_0})} \right]
  \leq
  \Exp\left[ \FI(\omega) \right] 
  \leq 
  \Exp\left[ \one[\Meso]  e^{2\CUEFr(\omega\zeta_{n_0})} \right],
  \label{eq:Wplower}
\end{equation}
for all $n$ sufficiently large.  Further, using Corollary~\ref{cor:meso}, the weighted probability of $\Meso$ can be calculated up to a $O(1)$ multiplicative error.

The main work that remains is to estimate the pair probability $\Exp\left[\FI(\omega_1)\FI(\omega_2)\right].$  To make this estimate, we can completely disregard the dependence on the microscopic field.  Let $\Meso(\omega)$ be the event $\CUEFr +\WN_r \in  \BS[\BF][\omega] \cap \ES[\omega.]$ Specifically, we begin by bounding
\begin{align*}
  \Exp\left[\FI(1)\FI(\omega)\right]
  &\leq
  \Exp\left[\one[\Meso(1)]\one[\Meso(\omega)]e^{2\CUEFr(\zeta_{n_0})+2\CUEFr(\omega\zeta_{n_0})}
\right].
\end{align*}
As in the proof of Lemma~\ref{lem:finefield2pUB2}, the bias 
\(
\Bias{\F} = 2\F(\zeta_{n_0}) + 2\F(\omega \zeta_{n_0})
\)
overshoots the desired means of $\CUEF(\zeta_i)$ for $i < -\log|\arg(\omega)|.$  We want to subtract a term to compensate for this distortion.  Hence, we define
\[
  \HM(\omega) = \min\{-\left[\log|\arg(\omega)|\right], n_0\}, \quad \HM_r(\omega)=\HM(\omega)-r,
\]
where $[\cdot]$ denotes integer part
and define
\[
  \Bias{\F}[\omega] = 2\F(\zeta_{n_0}) + 2\F(\omega \zeta_{n_0}) - 2\F(\zeta_{\HM(\omega)}).
\]
This point $\zeta_{\HM(\omega)}$ is nearly the midpoint on the hyperbolic geodesic connecting $\zeta_{n_0}$ to $\omega\zeta_{n_0}.$
Letting $\mu_\omega: \D \to \R$ be the mean of $\GF$ under the bias $e^{2\Bias{\GF}[\omega]}.$  At the points $\left\{ \zeta_i,\omega\zeta_i \right\}_1^n,$ we have by Lemmas~\ref{lem:branch} and \ref{lem:means}
\begin{align*}
  |\mu_\omega(\zeta_i) - i| \ll 1~\text{ and }
  |\mu_\omega(\omega\zeta_i) - i| \ll 1,
\end{align*}
uniformly in $1 \leq i \leq n_0.$  Using this mean, we can make a precise calculation in the Gaussian process of the mesoscopic barrier and endpoint events occurring.

In fact, we must introduce an extra condition on the Gaussian process 
at $\HM(\omega)$ to reflect a feature of branching random walk:
if two rays are conditioned to 
be large at their endpoints, it is typical for both rays 
to be within a logarithmic factor of the barrier at their branch point.

Hence, let
\[
  \ELM = \left\{ \F(\zeta_{\HM(\omega)}) > \BF(\HM(\omega)) - (\log n_0)^2 \right\}.
\]
\begin{lemma}
  Set $F = \one[{\BS[{\BF}](1) \cap \BS[{\BF}](\omega) \cap \ES(1) \cap \ES(\omega) \cap \ELM}].$
  There is an absolute constant $C>0$ so that
  uniformly in $\omega \in \T_{n_0}$ 
  \[
    \Exp \left[F(\GF_r + \WN_r + \mu_\omega)\right]
    \ll
    \frac{
    (\log n_0)^C}{
        (1+ (\HM_r(\omega))_+)^{3/2}
    (n_0-\HM(\omega)+1)^3}.
  \]
  \label{lem:ballot2}
\end{lemma}
\begin{proof}
    As in Lemma~\ref{lem:ratio}, we need to slightly rebalance the mean of this variable before applying Gaussian comparisons and the ballot theorem.  What follows is an exact analogue of what was done there.  Set $d = (1-\tfrac 34\tfrac{\log n_0}{n_0}).$  Let $\xi : \D \to \R$ be given by 
    \[
        \xi(z) = \Exp\left[ \WN_r(z)\Bias{\WN_r}[\omega](\zeta_{n_0}) \right]\cdot(\tfrac 34\tfrac{\log n_0}{n_0}),
    \]
    observing that $\xi(z) = O( (\log n_0)/n_0)$ uniformly in $z\in\D.$ 
      Then by Lemma~\ref{lem:means},
      \[
          \Exp\left[ F(\GF_r +\WN_r + \mu_\omega) \right]
          =
          \frac{
              \Exp\left[ F(\GF_r +\WN_r + d \cdot \mu_\omega-\xi)\exp({\tfrac 34\tfrac{\log n_0}{n_0}\Bias{\GF_r+\WN_r}[\omega]}) \right]
          }
          {
              \Exp\left[
              \exp({\tfrac 34\tfrac{\log n_0}{n_0}\Bias{\GF_r+\WN_r}[\omega]}) 
          \right]
          }.
      \]
      On the event $\GF_r +\WN_r + d\cdot \mu_\omega -\xi \in
      \ES(1) \cap \ES(\omega) \cap \ELM,$
      \[
          \tfrac 34\tfrac{\log n_0}{n_0}\Bias{\GF_r+\WN_r}[\omega]
          =
          O( (\log n_0)^2/n_0).
      \]
      Further,
      \[
\Exp\left[
              \exp({\tfrac 34\tfrac{\log n_0}{n_0}\Bias{\GF_r+\WN_r}[\omega]}) 
          \right]
          =e^{O( (\log n_0)^2/n_0)}.
      \]
      Hence we conclude that
      \begin{equation*}
          \Exp\left[ F(\GF_r +\WN_r + \mu_\omega) \right]
          \ll
          \Exp\left[ F(\GF_r +\WN_r + d\cdot\mu_\omega-\xi
					\corO{)}\right].
      \end{equation*}
  The proof now follows from Proposition~\ref{prop-overlap}, decomposing according to the value of $\GF_r(\HM_r(\omega))$ 
and the ballot theorem.
\end{proof}

This estimate transfers relatively painlessly to the expectation of $\FI(1)\FI(\omega).$

\begin{lemma}
  There is a $p_0 > 1$ and a $C>0$ so that for all $\log\log n > p > p_0,$ the following holds.
  For all $\omega \in \T_{n_0},$ 
  \[
    \Exp \left[ \FI(1)\FI(\omega) \right]
    \ll \frac{
      \Exp \left[ \FI(1)\right]^2
      n_0^3
      (\log n_0)^C
      e^{2\BF(\HM)-\HM_r}}
      {(1+ (\HM_r(\omega))_+)^{3/2}(n_0-\HM(\omega)+1)^3}.
    \]
    \label{lem:mesofield2pUB1}
  \end{lemma}
  \begin{proof}
    We begin by discarding the microscopic information.  That is, we estimate
    \[
      \Exp \left[ \FI(1)\FI(\omega) \right]
      \ll \Exp\left[ 
	\one[\Meso(1) \cap \Meso(\omega)]
	e^{2 \CUEFr(\zeta_{n_0})
	+2 \CUEFr(\omega\zeta_{n_0})}
      \right].
    \]
    We split this into two parts, according to $\ELM.$
    We first estimate the contribution to the expectation when $\CUEFr+\WN_r \in \ELM.$
    On the event $\CUEFr + \WN_r \in \BS(1) \cap \BS(\omega) \cap \ES(1) \cap \ES(\omega),$ we have by definition that $(\CUEFr+\WN_r)(\zeta_{\HM}) \leq \BF(\HM).$  Consequently,
      setting
        $F = \one[{\BS[{\BF}](1) \cap \BS[{\BF}](\omega) 
      \cap \ES(1) \cap \ES(\omega) \cap \ELM}]$
    we have the bound,
    \begin{align*}
      &\!\Exp \!\left[\!F(\CUEFr + \WN_r)e^{2 \CUEFr(\zeta_{n_0})+2 \CUEFr(\omega\zeta_{n_0})}\!\right]
     \! \!\ll\! 
      \Exp \! \left[\!F(\CUEFr \!+ \!\WN_r)
	e^{\Bias{\CUEFr}[\omega]-2\WN_r(\zeta_{\HM(\omega)})+2\BF(\HM)} 
      \!\right]\!. \\
      \intertext{After biasing by $e^{-2\WN_r(\zeta_{\HM(\omega)})},$ the variable $\WN_r(\zeta_{\HM(\omega)})$ has a negative mean, hence by monotonicity of $F,$}
      &\Exp\left[ F(\CUEFr+\WN_r) e^{\Bias{\CUEFr}[\omega]-2\WN_r(\zeta_{\HM(\omega)})} ~\middle|~\CUEFr \right] \\
      &\leq \Exp\left[ F(\CUEFr+\WN_r)e^{\Bias{\CUEFr}[\omega]}~\middle|~\CUEFr \right]
      \Exp\left[  e^{-2\WN_r(\zeta_{\HM(\omega)})}  \right].\\ 
      \intertext{Thus, we have that}
      &\Exp \left[F(\CUEFr + \WN_r)e^{2 \CUEFr(\zeta_{n_0})+2\CUEFr(\omega\zeta_{n_0})}
      \right]
      \ll \Exp \left[F(\CUEFr + \WN_r)e^{\Bias{\CUEFr}[\omega]+2\BF(\HM)}\right].
    \end{align*}

    {Applying Propositions~\ref{prop:clt}, and~\ref{prop:domination}, and using that the additive errors in Proposition~\ref{prop:clt} are negligible, we have that}
    \begin{align*}
      &\Exp \left[F(\CUEFr + \WN_r)e^{2 \CUEFr(\zeta_{n_0})+2 \CUEFr(\omega\zeta_{n_0})}\right]
      \ll 
      \Exp \left[F(\GF_r + \WN_r + \mu_\omega)\right]
      \Exp \left[
	e^{\Bias{\GF_r}[\omega]+2\BF(\HM)} 
      \right]. \\
      \intertext{Using Lemma~\ref{lem:branch}, we have that}
      &\frac12\Var(\Bias{\GF_r}[\omega]) - 2(n_0 - \HM(\omega)) - \HM_r(\omega) \ll 1.\\
      \intertext{So, we conclude, applying Lemma~\ref{lem:ballot2},}
      &\Exp \left[F(\CUEFr + \WN_r)e^{2 \CUEFr(\zeta_{n_0})+2 \CUEFr(\omega\zeta_{n_0})}\right]
      \ll \frac{(\log n_0)^C e^{2(n_0-\HM(\omega)) + \HM_r(\omega)+2\BF(\HM)}}
      {
        (1+ (\HM_r(\omega))_+)^{3/2}
      (n_0-\HM(\omega)+1)^3}. \\
      \intertext{Using Corollary~\ref{cor:meso} and Proposition~\ref{prop:1pLB} with $\epsilon=\tfrac12$, we write this again in terms of $\Exp \left[ \FI(1) \right]$ (adjusting $C$ as need be)}
      &{
          \Exp \left[F(\CUEFr + \WN_r)e^{2 \CUEFr(\zeta_{n_0})+2 \CUEFr(\omega\zeta_{n_0})}\right]
      \ll \frac{
	\Exp \left[ \FI(1)\right]^2
    n_0^3
    (\log n_0)^C e^{-\HM_r(\omega)+2\BF(\HM)}}{
        (1+ (\HM_r(\omega))_+)^{3/2}
    (n_0-\HM(\omega)+1)^3}.} \\
      \end{align*}

      We now turn to estimating the contribution to $\Exp\left[ \FI(1)\FI(\omega) \right]$ on the complementary event $\CUEFr+\WN_r \not\in \ELM.$  In this case, we have that $(\CUEFr+\WN_r)(\zeta_{\HM})\leq \BF(\HM) - (\log n_0)^2.$  Therefore,
      \begin{align*} 
	\Exp &\left[\one[\Meso(1) \cap \Meso(\omega)]
    \one[{(\ELM)}^c](\CUEFr+\WN_r)e^{2 \CUEFr(\zeta_{n_0})+2 \CUEFr(\omega\zeta_{n_0})}\right]\\
	&\ll
	\Exp \left[
	  e^{\Bias{\CUEFr}[\omega] - 2\WN_r\left( \zeta_{n_0} \right)+2\BF(\HM)-2(\log n_0)^2} 
	\right] 
	\ll
	\Exp \left[
	  e^{\Bias{\CUEFr}[\omega]+2\BF(\HM)-2(\log n_0)^2} 
	\right].
      \end{align*}
      This additional $(\log n_0)^2$ makes the entire expression smaller in order than the contribution of the $\CUEFr+\WN_r \in \ELM,$ and this completes the proof.
    \end{proof}

    We must also make a finer estimate when $\HM(\omega) \leq r.$  For these terms, we need that $\Exp\left[ \FI(1)\FI(\omega) \right]$ is the product of expectations up to a $(1+o(1))$ multiplicative error.  For the Gaussian process, we prove this using the exponential decay of correlations in $\GF$:
    \begin{lemma}
      Set $F = \one[{\BS[{\BF}](1) \cap \BS[{\BF}](\omega) \cap \ES(1) \cap \ES(\omega)}].$
      For all $\epsilon>0,$ and all $n_0$ sufficiently large, the following holds.
      Uniformly in $\omega \in \T_{n_0}$ with 
      $\HM(\omega) \leq r$,
      \begin{align*}
	&\Exp \left[F(\GF_r + \WN_r)e^{2\GF_r(\zeta_{n_0})+2\GF_r(\omega\zeta_{n_0})} \right]\\
	&\leq 
	\Exp \left[ 
	  \one[{\BS[{\BF}]\cap \ES}](\GF_r + \WN_r)e^{2\GF_r(\zeta_{n_0})}
	\right]^2(1+\epsilon+O(e^{-r+\HM(\omega)})).
      \end{align*}
      \label{lem:ballot3}
    \end{lemma}
    \begin{proof}
          Let $\mu(z) = 2\Exp\left[ \GF(z) (\GF_r(\zeta_{n_0}) + \GF(\omega\zeta_{n_0}))\right].$ Let $\mu_r$ be defined as in \eqref{eq:Fr}.  
	  By Lemma~\ref{lem:means} we have
          \[
              \Exp \left[F(\GF_r + \WN_r)e^{2\GF_r(\zeta_{n_0})+2\GF_r(\omega\zeta_{n_0})} \right]
              =\Exp \left[F(\GF_r + \WN_r+ \mu_r)\right]
          \Exp\left[e^{2\GF_r(\zeta_{n_0})+2\GF_r(\omega\zeta_{n_0})} \right].
          \]
          By Lemma~\ref{lem:branch}, we have that 
          \( |\Exp\left[ \GF(\zeta_i)\GF(\omega \zeta_j) \right]| \ll 1
          \)
          uniformly over $r \leq i,j \leq n_0.$
          Hence, in addition we have $\mu_r(\zeta_j) = j-r +O(1)$ for all $r \leq j \leq n_0$ and $\mu_r(\zeta_j \omega) = j-r+O(1)$ for all $r \leq j \leq n_0.$  Thus, we can apply Proposition~\ref{prop-separated} and Lemma~\ref{prop-comp-bar} to get that for any $\epsilon > 0$ and all $n$ sufficiently large
          \[
              \Exp \left[F(\GF_r + \WN_r+ \mu_r)\right]
              \leq 
              (1+\epsilon)
              \Exp\left[ 
                  \one[{\BS[{\BF}]\cap \ES}](\GF_r + \WN_r+\mu_r)
              \right]^2
              +O(e^{-\omega(\log n)}).
          \]
	  The additive error is much smaller in order than the probability, on account of Proposition~\ref{prop-comp-bar} 
	    and the lower bound
	    in the ballot theorem \ref{theo-app-ballot}.
	    
	    Using Proposition~\ref{prop-comp-bar}, Lemma~\ref{lem:means} and Corollary~\ref{cor:means},
          \begin{align*}
              &\Exp\left[ 
                  \one[{\BS[{\BF}]\cap \ES}](\GF_r + \WN_r+\mu_r)
              \right] \\
              \leq &(1+\epsilon)
              \frac{\Exp\left[ 
                  \one[{\BS[{\BF}]\cap \ES}](\GF_r + \WN_r)e^{2\GF_r(\zeta_{n_0})}
              \right]}{\Exp\left[ e^{2\GF_r(\zeta_{n_0})} \right]}
              +O(e^{-\omega(\log n)}).
          \end{align*}
          To complete the proof of the lemma, we only need to show that the exponential moments can be compared.  Hence, we compute the variance of the biasing term.  By expanding the variance as a sum of $\GF(z)$ over terms in $z,$ and applying Lemma~\ref{lem:branch}, we have
          \[
              2\Var\left( \GF_r(\zeta_{n_0}) + \GF_r(\omega\zeta_{n_0}) \right)
              =
              4\Var\left( \GF_r(\zeta_{n_0})\right)
              +O(e^{-r+\HM(\omega)}).
          \]
          The lemma now follows.
    \end{proof}

    \begin{lemma}
      For all $\epsilon > 0,$ there is a $p >0$ sufficiently large that 
      for all $n$ sufficiently large,
      and
      all $\omega \in \T_{n_0}$ so that 
      $ \HM(\omega) \leq r,$
      \[
	\Exp \left[ \FI(1)\FI(\omega) \right]
	\leq
	\Exp \left[ \FI(1)\right]
	\Exp\left[\FI(\omega) \right]
	(1+\epsilon +O(e^{-r+\HM(\omega)}))
	.
      \]
      \label{lem:mesofield2pUB2}
    \end{lemma}
    \begin{proof}
      Let $F = \one[{\BS[{\BF}](1) \cap \BS[{\BF}](\omega) \cap \ES(1) \cap \ES(\omega)}].$
      We again begin by discarding the microscopic information and estimating
      \[
	\Exp \left[ \FI(1)\FI(\omega) \right]
	\leq \Exp\left[ 
	  F(\CUEFr+\WN_r)
	  e^{2 \CUEFr(\zeta_{n_0})
	  +2 \CUEFr(\omega\zeta_{n_0})}
	\right].
      \]
      Applying Proposition~\ref{prop:clt} and Proposition~\ref{prop:domination}, we get
      \[
	\Exp \left[ \FI(1)\FI(\omega) \right]
	\leq \Exp\left[ 
	  F(\GF_r+\WN_r)
	  e^{2 \GF_r(\zeta_{n_0})
	  +2 \GF_r(\omega\zeta_{n_0})}
	\right](1+O(\log N)^{-K}).
      \]
      The proof now follows on applying Lemma~\ref{lem:ballot3}.
    \end{proof}

    With these ingredients, we are now in position to prove that some $\Ev$ occurs with high probability.
    \begin{proposition}
      For all $\epsilon > 0$ there is a $p > 0$ so that for all $n$ sufficiently large,
      \[
	\Pr\left(
	\cup_\omega \left\{ \CUEFr +\WN_r \in \Ev[\omega] \right\}
	\right)
	> 1-\epsilon,
      \]
      where $\omega$ runs over $\T_{n_0}.$
      \label{prop:LB}
    \end{proposition}
    \begin{proof}
      Recall that we let 
      \(
      Z = \sum_{\omega} \FI(\omega),
      \)
      with the sum over $\omega \in \T_{n_0}$, and we wish to show that $Z > 0.$
      We bound 
      \begin{align*}
	\Pr[ \one[Z>0]]
	&\geq
	\frac
	{ \bigl(\sum_{\omega}\Exp\left[ \FI(\omega) \right]\bigr)^2}
	{\sum_{\omega_1, \omega_2}\Exp\left[\FI(\omega_1)\FI(\omega_2)\right]}.
      \end{align*}
      By rotation invariance, we can rewrite both top and bottom sums as
      \begin{align*}
	\Pr[ \one[Z>0]]
	&\geq
	\frac
    {|\T_{n_0}|\bigl(\Exp\left[ \FI(1) \right]\bigr)^2}
	{\sum_{\omega}\Exp\left[\FI(1)\FI(\omega)\right]}.
      \end{align*}
      We now partition this sum according to the value of $\HM(\omega).$ We let $I_1$ be the sum
      \[
	I_1 = 
	\sum_{\substack{\omega \\ \HM(\omega) \leq r + \log \epsilon
	}}
	\Exp\left[\FI(1)\FI(\omega)\right]
      \]
      and we let $I_2$ be the sum over the remaining $\omega.$

      Hence, applying Lemma~\ref{lem:mesofield2pUB2}, we conclude that 
      \[
          I_1 \leq |\T_{n_0}|(1+\epsilon)\Exp\left[\FI(1)\right]^2.
      \]

      For a given value of $\HM,$ there are at most $O(e^{n_0 - \HM})$ many integers $\omega \in \corO{\T_{n_0}}
			$ that attain this value of $\HM.$
      Let $\ell_0$ be the floor of $r + \log \epsilon.$
      \begin{align*}
	I_2=&\sum_{\ell=\ell_0}^{n_0}
	\sum_{\substack{\omega \in \corO{\T_{n_0}}
	\\ \HM(\omega) = \ell}}
	\Exp\left[\FI(1)\FI(\omega)\right]\\
	\ll&\sum_{\ell=\ell_0}^{n_0}
	e^{n_0 - \ell}
	\frac{\Exp \left[ \FI(1)\right]^2
	  (n_0)^3
	  (\log n_0)^C
	  e^{\ell-r-\tfrac{3\ell}{2n_0}\log n_0 + 2(\log n_0)^{1
	    -
	\delta}}
	}
	{(1+(\ell-r)_+)^{3/2}(n_0-\ell+1)^3}
	\\
	\ll&
	e^{n_0-r}
	\Exp \left[ \FI(1)\right]^2
	e^{2(\log n_0)^{1-\delta}}
	(\log n_0)^C
	\sum_{\ell=\ell_0}^{n_0}
	\frac{
	  n_0^3
	  e^{-\tfrac{3\ell}{2n_0}\log n_0 }
	}
	{
        (1+(\ell-r)_+)^{3/2}
        (n_0-\ell+1)^3}.
	\\
	\intertext{The sum is $O(-\log \epsilon)$: for $\ell \leq \tfrac{2n_0}{3},$ 
	the summand can be dominated by $O(
    (1+(\ell-r)_+)^{-3/2})$ which gives the $O(-\log\epsilon),$ for
	$ \tfrac{2n_0}{3} \leq n_0 \leq n_0 - n_0^{1/2},$ 
	the summand is $O(n_0^{-1}),$ 
	and for $\ell \geq n_0 - n_0^{1/2}$ 
	the summand is $O((n_0 - \ell+1)^{-3}).$ 
      Hence we get that}
	I_2\ll&
	-\log(\epsilon) e^{n_0-r}
	\Exp \left[ \FI(1)\right]^2
	(\log n_0)^C
	e^{2(\log n_0)^{1
	    -
	\delta}}.
      \end{align*}
      As $r = 3(\log n_0)^{1
	    -
    \delta},$ in particular, the sum is now negligible with respect to $|\T_{n_0}|\Exp \left[ \FI(1)\right]^2.$
    \end{proof}
    \begin{theorem}
        For any $\delta > 0,$
        with probability going to $1,$ 
        \[
            \max_{|z| = 1} \CUEF(z) 
            \geq \log N - (\tfrac 34 + \delta) \log\log N.
        \]
        \label{thm:LB}
    \end{theorem}
    \begin{proof}
        By Proposition~\ref{prop:LB}, for each $\epsilon>0$
        there exists a deterministic set 
        $\hat T_\epsilon\subset \mathbb{D}$ of cardinality $N$ at most 
        so that, for all $N$ large enough,
        \begin{equation*}
            \label{AC-00}
            \Pr(\max_{z\in \hat T_\epsilon} (\CUEFr(z)+ \WN_r(z))
            \geq m_N-\epsilon \log \log N)
            \geq 1-\epsilon,
        \end{equation*}
        where $m_N=\log N-\frac{3}{4}\log\log N$. In the sequel, we always assume that
        $\epsilon<1/8$.  
        By the same type of chaining argument used in the proof of Theorem~\ref{thm:ub}, we can show that with high probability 
        \[
            \max_{\omega \in \T}\CUEF(\omega \zeta_r) \leq r + 100\log n.
        \]
        Hence, it follows that for each $\epsilon >0$ and all $N$ large enough,
        \begin{equation}
            \label{AC-0}
            \Pr(\max_{z\in \hat T_\epsilon} (\CUEF(z)+ \WN_r(z))
            \geq m_N-\epsilon \log \log N)
            \geq 1-\epsilon.
        \end{equation}


        Let $j_{\max}=\max\{j\in \mathbb{N}: 2^j\leq2\sqrt{\log N}\}$
        and set $J=\mathbb{N}\cap [0,j_{\max}]$. For $j\in J$, set
        $x_j=2^j$ and 
        $$A_j=\{z\in \hat T_\epsilon: \CUEF(z)-m_N+2\epsilon \log\log N
            \in [-x_j,-x_{j+1}]\} .$$ Note that, with
            $G$ denoting a centered Gaussian random variable with variance 
            $\frac12 \log N$,
            \begin{eqnarray*}
                \E|A_j|&\leq
                &|\hat T_\epsilon|\cdot \sup_{z\in \D}\Pr(\CUEF(z)\geq 
                m_N-2\epsilon \log \log N-x_{j+1})\\
                &\leq&
                N\cdot \E(e^{2 G})\cdot e^{-2(m_N-2\epsilon\log \log N-2x_j)}
                \leq  e^{2\log \log N+4x_{j}}\,,
            \end{eqnarray*}
            where the second inequality follows from  Proposition
            \ref{prop:domination} and the last from a Gaussian computation.
            In particular,
            \begin{equation}
                \label{AC-1}
                \Pr(|A_j|\geq e^{4\log \log N+5x_{j}})
                \leq e^{-x_{j}-2\log \log N}.
            \end{equation}

            For any $z \in \C$ with $\dH(0,z) > r,$ (including $\hat T_\epsilon$), $\WN_r(z)$ is distributed like a standard normal with variance $2.$
            We now obtain that
            \begin{eqnarray}
                \label{AC-2}
                &&\Pr(\exists z\in A_j: \CUEF(z)+\WN_r(z)\geq m_N-\epsilon
                \log\log N)\nonumber \\
                &\leq & \Pr(\exists z\in A_j: \WN_r(z)\geq \epsilon 
                \log\log N+x_j)\nonumber \\
                &\leq & \Pr(|A_j|\geq e^{4\log \log N+5x_j})+
                e^{4\log\log N+5x_j}\Pr(\WN_r(1)\geq \epsilon \log\log N+x_j)\nonumber\\
                &\leq&
                e^{-x_j-2\log\log N}+e^{-\epsilon^2(\log\log N)^2/4-x_j^2/2+4
                \log\log N+5x_j}
                \leq C_N e^{-x_j}\,,
            \end{eqnarray}
            where $C_N\to_{N\to\infty} 0$ and we used
            \eqref{AC-1} in the next to last inequality.

            By a simple Gaussian computation, 
            for $N$ large enough one has
            \begin{eqnarray}
                \label{AC-3}
                &&	\Pr(\max_{z\in \hat T_\epsilon} 
                \WN_r(z)\geq \epsilon\log \log N+
                x_{j_{\max}+1})\\
                &\leq& 
                N\cdot \Pr(\WN_r(1)\geq \epsilon 
                \log\log N+ 2\sqrt{\log N})
                \leq e^{-\epsilon \log \log N \sqrt{\log N}}.
                \nonumber
            \end{eqnarray}
            Therefore, 
            \begin{eqnarray*}
                &&\Pr(\max_{z\in \hat T_\epsilon} (\CUEF(z)+\WN_r(z))\geq 
                m_N-\epsilon \log\log N)\\
                &\leq &
                \Pr(\max_{z\in \hat T_\epsilon} \CUEF(z)\geq m_N-2\epsilon\log 
                \log N)\\
                &&\quad
                +\sum_{j=0}^{j_{\max}}
                \Pr(\exists z\in A_j: \CUEF(z)+\WN_r(z)\geq m_N-\epsilon \log\log N)\\
                &&\quad
                +
                \Pr(\max_{z\in \hat T_\epsilon} \WN_r(z)\geq \epsilon\log \log N+
                x_{j_{\max}+1})\\
                &\leq &
                \Pr(\max_{z\in \hat T_\epsilon} \CUEF(z)\geq m_N-2\epsilon\log \log N)
                +C_N \sum_{j=0}^{j_{\max}}e^{-x_j}+
                e^{-\epsilon \log \log N \sqrt{\log N}},
            \end{eqnarray*}
            where \eqref{AC-2} and \eqref{AC-3} were used in the last display.
            It follows from this and  \corO{\eqref{AC-0}} that
            \begin{eqnarray*}
                && \liminf_{N\to\infty}\Pr(\max_{z\in \D}
                \CUEF(z)\geq m_N-2\epsilon\log \log N)\\
                &\geq&
                \liminf_{N\to\infty}
                \Pr(\max_{z\in \hat T_\epsilon} \CUEF(z)\geq m_N-2\epsilon\log \log N)\\
                &\geq& \liminf_{N\to\infty}
                \Pr(\max_{z\in \hat T_\epsilon} (\CUEF(z)+\WN_r(z))\geq 
                m_N-\epsilon \log\log N)\geq 1-\epsilon\,.
            \end{eqnarray*}
            This completes the proof of the theorem, on account of the maximum principle.
        \end{proof}

    \section{Baxter-type Toeplitz determinant identities}
    \label{sec:Baxter}
    Let $f \in \operatorname{L}^1(\T).$  Define the \emph{$N$-th order Toeplitz determinant with symbol $f$} by
        \[
            D_N(f)= \det\left( \hat{f}(j-k) \right)_{1 \leq j,k \leq N},
        \]
        with $\hat f( k)$ the $k$-th Fourier coefficient of $f,$ i.e.
        \[
            f(e^{i\theta}) = \sum_{k=-\infty}^\infty \hat f(k) e^{ik\theta}.
        \]
        Toeplitz determinants relate to expectations against the unitary group through the celebrated relationship \cite[p.23]{Szego1}
        \[
            \Exp\left[ \prod_{h=1}^N f(e^{i\theta_h}) \right]
            =
            D_N(f),
        \]
        where we recall that $\left\{ e^{i\theta_h} \right\}_{h=1}^N$ are the eigenvalues of an $N\times N$ Haar unitary matrix.

    A core lemma of \cite{Johansson}, used to great effect to estimate the Fourier transform of polynomial linear statistics is a Toeplitz determinant identity he attributes to \cite{Baxter}.  
    Related identities were used in Szeg\"o's proof of the Strong Szeg\"o theorem~\cite{Szego2}, and similar identities have appeared elsewhere, \corE{for example in \cite{CFS}, \cite{CFKRS}, \cite{Day}, and \cite{Bottcher}.}
    In any case, \corE{Baxter's formula} shows that
    the so-called Cauchy identity holds already at finite $N$ for the Toeplitz determinant of certain rational symbols.

    Throughout this section, we assume $U_\ell$ and $V_m$ are the polynomials
    \begin{equation}
      \begin{aligned}
	U_\ell(z) &= \prod_{j=1}^\ell (1-a_j z) \\
	V_m(z) &= \prod_{j=1}^m (1-b_j z), 
      \end{aligned}
      \label{eq:Uell}
    \end{equation}
    for some $|a_j| < 1$ and $|b_j|<1.$ 

    \begin{proposition}[Baxter]
      Let $U_\ell$ and $V_m$ be as in \eqref{eq:Uell}. Then if $N \geq \ell$ or $N \geq m,$
      \[
	D_N\left( \frac{1}{U_\ell(e^{-i\theta})V_m(e^{i\theta})} \right) = \prod_{i=1}^m \prod_{j=1}^\ell \frac{1}{1-a_j b_i}.
      \]
    \end{proposition}

    This can be expressed using a contour integral representation as
    \[
      \prod_{i=1}^m \prod_{j=1}^\ell \frac{1}{1-a_j b_i}
      =\exp\left( 
      \frac{1}{2\pi i} \int_{S^1} \log\left(\frac{1}{V_m(z)}\right) \frac{d}{dz} \log( U_\ell(z^{-1}) )\,dz
      \right).
    \]
    This can be checked by an elementary residue calculation.
    This formula is attractive as the contour integral is one representation of the $\Hhalf$ inner-product for the logarithms of $V_m$ and $U_\ell$.
    It is also possible to take $m \to \infty$ in the formula, holding $\ell \leq N,$ so that $V_m$ can be replaced by a zero-free analytic function $v$ in an open neighborhood of the closed unit disk with $v(0) = 0.$

    We give a modification of this formula that additionally allows for a low degree term like $U_\ell$ or $V_m$ to be placed in the numerator of the symbol.  We begin with the following explicit formula, which is purely algebraic in nature.
    \begin{lemma}
	Let $k \in \mathbb{N}$ have $k < \ell.$
      Let $U_\ell$ and $V_m$ be as in \eqref{eq:Uell} and in addition assume all $\left\{ a_j \right\}_{j=1}^\ell$ are pairwise distinct.  Suppose $p(z)$ is a polynomial of degree strictly less than $\ell+m$ which does not vanish on $\left\{ a_j \right\}_{j=1}^\ell.$  Then
      \begin{align*}
	D_{\ell - k}\left( 
	\frac{
	  p(e^{i\theta})e^{-ik\theta}}
	  {U_\ell(e^{-i\theta})V_m(e^{i\theta})}
	  \right)
	  =
	  \corO{\prod_{j=1}^\ell} p(a_j)
	  &\prod_{i=1}^m \prod_{j=1}^\ell \frac{1}{1-a_j b_i} 
	  \sum_{\substack{S \subseteq [\ell] \\ |S|=k}}
	  \prod_{i \in S}
	  \left[ \frac{\prod_{j=1}^m (1-b_ja_i)}
	    {p(a_i) \prod_{j \not\in S} (a_j - a_i)}
	  \right].
	\end{align*}
	\label{lem:BaxterExact}
      \end{lemma}
      Compare with \cite[Lemma 7.4]{Baxter}.  In short, the proof here is just partial fractions, Cauchy-Binet, and the observation that the resulting matrices are essentially products of \corO{Vandermonde} determinants.
      \begin{proof}
	\corE{Observe the formula in the lemma is continuous in $\left\{ b_j \right\}_{j=1}^m \subset \D^m.$  Hence, by taking limits, it suffices to prove the formula under the additional assumption that all $\left\{ b_j \right\}_{j=1}^m$ are distinct.}
	Let $f(z)$ be the symbol
	\[
	  f(z) = \frac{p(z)z^{-k}}{U_\ell(z^{-1})V_m(z)}.
	\]
	Let $q(z) = z^\ell U_\ell(z^{-1}),$ so that
	\(
	f(z)z^{k-\ell} = \frac{p(z)}{q(z)V_m(z)}
	\)
	is a ratio of polynomials whose numerator has lower degree than its denominator.  
	To this rational function, all of whose poles are simple, we therefore can apply partial fractions to get
	\[
	  f(z)z^{k-\ell}
	  = 
	  \sum_{j=1}^\ell \frac{\alpha_j}{z-a_j} 
	  +\sum_{j=1}^m \frac{\beta_j}{1-zb_j},
	\]
	for some coefficients $\left\{ \alpha_j \right\}_1^\ell$ and $\left\{ \beta_j \right\}_1^m.$  Further, we have the formula
	\[
	  \alpha_j = \frac{p(a_j)}{\prod_{i \neq j}(a_j - a_i)
	  \prod_{i}(1-a_jb_i)
	},
      \]
      for any $j = 1,\dots,\ell.$

      Given the partial fractions expansion, we can give the Laurent expansion of $f$ in a neighborhood of the unit circle as
      \[
	f(z)
	= 
	\sum_{j=1}^\ell 
	\alpha_j
	\sum_{r=0}^\infty
	z^{\ell -k-r-1} a_j^r
	+O(z^{\ell-k}).
      \]
      In particular, for the purposes of computing the $\ell-k$ Toeplitz determinant, we only need $\left\{ \alpha_j \right\}_1^\ell.$  Letting $\hat f : \mathbb{Z} \to \mathbb{C}$ denote the Fourier coefficients of $f,$ we have that
      \begin{align*}
	&\left( \hat f(j-i) \right)_{1 \leq i,j \leq \ell-k}
	= \\
	&\begin{bmatrix}
	  \alpha_1 & \alpha_2 & \cdots & \alpha_\ell \\
	  \alpha_1 a_1 & \alpha_2 a_2 & \cdots & \alpha_\ell a_\ell \\
	  \vdots & \vdots & & \vdots \\
	  \alpha_1 a_1^{\ell-k-1} & \alpha_2 a_2^{\ell-k-1}& \cdots & \alpha_\ell a_\ell^{\ell-k-1}\\
	\end{bmatrix}
	\begin{bmatrix}
	  a_1^{\ell-k-1} & a_1^{\ell-k-2}  & \cdots & 1 \\
	  a_2^{\ell-k-1} & a_2^{\ell-k-2}& \cdots & 1 \\
	  \vdots & \vdots & & \vdots \\
	  a_\ell^{\ell-k-1} & a_\ell^{\ell-k-2}& \cdots &1\\
	\end{bmatrix}.
      \end{align*}

      Applying Cauchy-Binet to this expression, we have that
      \begin{align*}
	D_{\ell - k}(f)
	&=
	\sum_{\substack{S \subset [\ell] \\ |S|=k}}
	\prod_{j \not\in S} \alpha_j
	\prod_{\substack{i \neq j \\ i,j \not\in S}} (a_j - a_i).
      \end{align*}
      Now substitute in the value of $\alpha_i,$ to get
      \begin{align*}
	D_{\ell - k}(f)
	&=
	\sum_{\substack{S \subset [\ell] \\ |S|=k}}
	\prod_{j \not\in S}\frac{p(a_j)}{\prod_{i}(1-a_jb_i)}
	\prod_{\substack{i \neq j \\ j \not\in S \\ i \in S}} \frac{1}{(a_j - a_i)} \\
	&=
	\corO{\prod_{j=1}^\ell p(a_j)}
	\prod_{i=1}^m \prod_{j=1}^\ell \frac{1}{1-a_j b_i}
	\sum_{\substack{S \subseteq [\ell] \\ |S|=k}}
	\prod_{i \in S}
	\left[ \frac{\prod_{j=1}^m (1-b_ja_i)}
	{p(a_i) \prod_{j \not\in S} (a_j - a_i)}\right], 
      \end{align*}
      which was the claim.
    \end{proof}

    This formula can now be given a contour integral representation that allows the size of the determinant to be larger than $\ell-k$ and which allows us to remove the condition that all $\left\{a_i \right\}_1^\ell$ are distinct.
    \begin{proposition}
      Let $U_\ell$ and $V_m$ be as in \eqref{eq:Uell}, and let $k \in \mathbb{N}_0$ be fixed 
	with $k < \ell.$ 
      Then for any $N \geq \ell - k$ and any polynomial $p$ of degree at most $N$ which does not vanish on $\left\{ a_j \right\}_{j=1}^\ell$ and which does not vanish at $0,$
      \begin{align*}
	D_{N}\left( 
	\frac{
	  p(e^{i\theta})e^{-ik\theta}}
	  {U_\ell(e^{-i\theta})V_m(e^{i\theta})}
	  \right)
	  &=(-1)^{\binom{k}{2} + kN}p(0)^{N+k} \\
	  &\cdot
	  \exp\left( 
	  \frac{1}{2\pi i} \int_{\gamma} \log\left(\frac{p(z)}{p(0)V_m(z)}\right) \frac{d}{dz} \log( U_\ell(z^{-1}) )\,dz
	  \right) \\
	  & \cdot
	  \frac{1}{(2\pi i)^k} \idotsint_{\gamma} \frac{\Delta(z_1,\dots,z_k)^2}{k!}\prod_{i=1}^k \frac{V_m(z_i)}{p(z_i) z_i^{N+k} U_\ell(z_i^{-1})}dz_i,
	\end{align*}
	where $\Delta$ is the \corO{Vandermonde} determinant and
	where $\gamma$ is a positively oriented contour enclosing $0$, $\left\{ a_j \right\}_{j=1}^\ell$ but no zeroes of $p$ \corE{or $V_m.$}
	\label{prop:Baxter}
      \end{proposition}
      \begin{proof}
	The main task is to show that the desired expression holds for $N = \ell - k$ under the additional assumption that all $\left\{ a_i \right\}_1^\ell$ are distinct.  That is, we will show that
	\begin{align*}
	  D_{\ell-k}\left( 
	  \frac{
	    p(e^{i\theta})e^{-ik\theta}}
	    {U_\ell(e^{-i\theta})V_m(e^{i\theta})}
	    \right)
	    &=(-1)^{\binom{k}{2} + k(\ell - k)}p(0)^{\ell} \\
	    &\cdot
	    \exp\left( 
	    \frac{1}{2\pi i} \int_{\gamma} \log\left(\frac{p(z)}{p(0)V_m(z)}\right) \frac{d}{dz} \log( U_\ell(z^{-1}) )\,dz
	    \right) \\
	    & \cdot
	    \frac{1}{(2\pi i)^k} \idotsint_{\gamma} \frac{\Delta(z_1,\dots,z_k)^2}{k!}\prod_{i=1}^k \frac{V_m(z_i)}{p(z_i) z_i^{\ell} U_\ell(z_i^{-1})}dz_i.
	  \end{align*}
	  The right hand side is now continuous under smooth deformations of the roots $\left\{ a_i \right\}_1^\ell$ that do not cross the contour.  \corE{This fact establishes the formula for the case $N=\ell-k$ even without the restriction that $\left\{ a_i \right\}_1^\ell$ are all distinct.  In addition, it is possible to move any root to $0.$  Such a move allows one to establish the formula in the case $N = \ell -k + 1.$ Iterating this procedure, this establishes the formula in the case $N \geq \ell - k$ without the restriction that $\left\{ a_i \right\}_1^\ell$ are pairwise distinct.}

	  We show the formula when $N = \ell - k$ by expanding the right hand side using the residue theorem.  Expanding the first integral, we have that
	  \begin{align*}
	    \frac{1}{2\pi i} \int_{\gamma} \log\left(\frac{p(z)}{p(0)V_m(z)}\right) &\frac{d}{dz} \log( U_\ell(z^{-1}) )\,dz \\
	    &=
	    \frac{1}{2\pi i} \int_{\gamma} \log\left(\frac{p(z)}{p(0)V_m(z)}\right) \left[ -\frac{\ell}{z} + \sum_{j=1}^\ell \frac{1}{z-a_j}\right]\,dz \\
	    &=\sum_{j=1}^\ell \log\left(\frac{p(a_j)}{p(0)V_m(a_j)}\right).
	  \end{align*}
	  Hence,
	  \begin{align}
	    \label{eq:allcross}
	    p(0)^{\ell}
	    \exp\left( 
	    \frac{1}{2\pi i} \int_{\gamma} \log\left(\frac{p(z)}{p(0)V_m(z)}\right)
	    \frac{d}{dz} \log( U_\ell(z^{-1}) )\,dz\right) \hspace{-6cm}&\hspace{6cm} \\
	    &=\prod_{j=1}^\ell \frac{p(a_j)}{V_m(a_j)}
	    =\prod_{j=1}^\ell p(a_j)
	    \prod_{i=1}^m \prod_{j=1}^\ell \frac{1}{1-a_j b_i}.
	    \nonumber
	  \end{align}
		  Applying Lemma~\ref{lem:BaxterExact}, it remains to show that
	  \begin{align*}
	    &\frac{1}{(2\pi i)^k} \idotsint_{\gamma} \frac{\Delta(z_1,\dots,z_k)^2}{k!}\prod_{i=1}^k \frac{V_m(z_i)}{p(z_i) z_i^{\ell} U_\ell(z_i^{-1})}dz_i \\
	    &=(-1)^{\binom{k}{2} + k(\ell - k)}
	    \sum_{\substack{S \subseteq [\ell] \\ |S|=k}}
	    \prod_{i \in S}
	    \left[ \frac{\prod_{j=1}^m (1-b_ja_i)}
	      {p(a_i) \prod_{j \not\in S} (a_j - a_i)}
	    \right].
	  \end{align*}

	  For this, we use the following lemma.
	  \begin{lemma}
	    Let $\corE{D} \subset \mathbb{C}$ be an open, simply connected set, and suppose that $f : \corE{D} \to \mathbb{C}$ is holomorphic.
	    Let $Q_r = \prod_{j=1}^r(1-z \corE{c_j})$ for some finite set $\left\{ \corE{c}_j \right\}_1^r \subset \mathbb{C}.$
	    Let $\gamma \subset U$ be a positively oriented contour enclosing $\left\{ \corE{c}_j \right\}_1^r$ then
	    \begin{align*}
	      \frac{1}{(2\pi i)^k} \idotsint_{\gamma} &\frac{\Delta(z_1,\dots,z_k)^2}{k!}\prod_{i=1}^k \frac{f(z_i)}{z_i^{r} Q_r(z_i^{-1})}dz_i\\
	      &=  
	      (-1)^{\binom{k}{2}+k(r-k)}
	      \sum_{\substack{S \subseteq [r] \\ |S|=k}}
	      \prod_{i \in S}
	      \left[ \frac{\corO{f(\corE{c}_i)}}
	      {\prod_{j \not\in S} (\corE{c}_j - \corE{c}_i)}\right].
	    \end{align*}
	    \label{lem:vdm}
	  \end{lemma}
	  The proof of Proposition~\ref{prop:Baxter} now follows by applying Lemma~\ref{lem:vdm} with $f=V_m/p,$ $Q_r=U_\ell,$ \corE{and $c_j = a_j$.}   So we turn to the proof of the lemma.  This is just a direct calculation with the residue theorem.  We iterate the integral, doing the $z_k$ integral first.  The relevant integral is
	  \[
	    \frac{1}{2\pi i} \int_\gamma \frac{f(z_k)}{z_k^{r} Q_r(z_k^{-1})}\prod_{s < k}(z_s - z_k)^2\,dz_k
	    =\sum_{i_k=1}^r  \frac{f(\corE{c}_{i_k})}{\prod_{j \neq i_k} (\corE{c}_{i_k} - \corE{c}_j)} \prod_{s < k}(z_s - \corE{c}_{i_k})^2.
	  \]
	  The $z_{k-1}$ integral has the same form, except that the $(z_{k-1} - \corE{c}_{i_k})^2$ term from evaluating the $z_k$ term cancels one of poles of the integrand.  This procedure is easily iterated to give
	  \begin{align*}
	    \frac{1}{(2\pi i)^k} \idotsint_{\gamma} &\frac{\Delta(z_1,\dots,z_k)^2}{k!}\prod_{i=1}^k \frac{f(z_i)}{z_i^{r} Q_r(z_i^{-1})}dz_i\\
	    &=
	    \frac{1}{k!}
	    \sum_{\substack{i_k,\dots,i_1 \\ \text{distinct}}}
	    \prod_{s=1}^k \frac{f(\corE{c}_{i_s})}{\prod_{j \not\in \left\{ i_k,\dots,i_s \right\}} (\corE{c}_{i_s} - \corE{c}_j)}
	    \prod_{t < s} \left( \corE{c}_{i_t} - \corE{c}_{i_s} \right) \\
	    &=  
	    (-1)^{\binom{k}{2}}
	    \sum_{\substack{S \subseteq [r] \\ |S|=k}}
	    \prod_{i \in S}
	    \left[ \frac{f(\corE{c}_i)}
	    { \prod_{j \not\in S} (\corE{c}_i - \corE{c}_j)}\right] \\
	    &=  
	    (-1)^{\binom{k}{2}+k(r-k)}
	    \sum_{\substack{S \subseteq [r] \\ |S|=k}}
	    \prod_{i \in S}
	    \left[ \frac{f(\corE{c}_i)}
	    {\prod_{j \not\in S} (\corE{c}_j - \corE{c}_i)}\right].
	  \end{align*}
	\end{proof}

	We will now evaluate the effect of contracting the contour $\gamma$ to infinity.  This will give us an exact formula for the correction term with complexity only depending on the degree of $p.$   
	\begin{lemma}
	  Suppose $p(z)$ is given by
	  \[
	    p(z) = a \prod_1^{2k} (z-c_j),
	  \]
	  with all $c_j$ pairwise distinct, and suppose that $N+3k \geq \ell+ m+2.$
	  Let $\gamma$ be a contour enclosing all $\left\{ a_j \right\}_1^\ell$ and $0$ but enclosing no roots of $p.$  Then,
	  \begin{align*}
	    \frac{1}{(2\pi i)^k} \idotsint_{\gamma} &\frac{\Delta(z_1,\dots,z_k)^2}{k!}\prod_{i=1}^k \frac{V_m(z_i)}{p(z_i) z_i^{N+k} U_\ell(z_i^{-1})}dz_i \\
	    &=(-1)^{\binom{k}{2}}
	    \sum_{\substack{S \subseteq [2k] \\ |S|=k}}
	    \prod_{i \in S}
	    \left[ \frac{V_m(c_i)}
	    {a c_i^{N+k}U_\ell(c_i^{-1})\prod_{j \not\in S} (c_j - c_i)}\right].
	  \end{align*}
	  \label{lem:residues}
	\end{lemma}
    \begin{proof}
      For contours (or more generally cycles, which are oriented sums of closed curves), let
        \(
            I(\gamma_1,\dots,\gamma_k)
        \)
	be the displayed integral in the \corO{lemma}, with $z_i$ integrated along the $\gamma_i$ contour for each $1 \leq i \leq k.$
        Fix another contour $\gamma'$ not intersecting any roots of $p$ nor $0$ nor any of $\left\{ a_j \right\}_1^\ell.$
        For a subset $T \subseteq [k],$ let 
        \(
        I_T
        \)
        be
        \(
            I(\gamma_1,\dots,\gamma_k)
        \)
        where $\gamma_i = \gamma-\gamma'$ if $i \in T$ and $\gamma_i = \gamma'$ otherwise.
        Using the multilinearity of the integral, we can write
        \[
            I(\gamma,\dots,\gamma)
            =
            \sum_{\substack{T \subseteq [k]}}
            I_{T}.
        \]

        Now we send $\gamma'$ to infinity.  By the assumption on the degrees of the polynomials, for each $i$ the integrand is $O(|z_i|^{-2}).$  Hence, any term in the above expansion of the form $I(\dots, \gamma', \dots)$ will tend to $0.$  Thus we have that
        \[
            I(\gamma,\dots,\gamma)
            =
            I(\gamma-\gamma',\dots,\gamma-\gamma')
            +o(1).
        \]
	On the other hand, by the residue theorem, 
	this right hand side stabilizes once $\gamma'$ encloses all roots of 
	$p$, \corO{all $a_i$'s and $0$}.  
	Moreover, we can apply Lemma~\ref{lem:vdm} to evaluate it, with $r = 2k$,
	\[
	  f(z)=\frac{V_m(z)}{a z^{N+k} U_\ell(z^{-1})}~\text{and}~Q_{2k}(z) = \frac{z^{2k} p(z^{-1})}{a} = \prod_{j=1}^{2k}(1-zc_j),
	\]
	\corE{and with $D$ being the domain encircled by $\gamma'-\gamma$ which can be chosen simply connected. By construction, $f$ is holomorphic in $U$ since $\gamma'-\gamma$ does not encircle any of the $a_i$'s or $0.$}
	We get
	\[
	  \lim_{\gamma' \to \infty}
            I(\gamma-\gamma',\dots,\gamma-\gamma')
	    =
	    (-1)^{k+k^2+\binom{k}{2}}
	    \sum_{\substack{S \subseteq [2k] \\ |S|=k}}
	    \prod_{i \in S}
	    \left[ \frac{V_m(c_i)}
	    {a c_i^{N+k}U_\ell(c_i^{-1})\prod_{j \not\in S} (c_j - c_i)}\right].
	\]
      The extra sign change $(-1)^k$ arises as $\gamma-\gamma'$ winds negatively around the roots of $p.$
    \end{proof}


	\begin{example}
	  \label{ex:2p}
	  Consider evaluating the Toeplitz determinant with symbol $|1-z_1e^{i\theta}|^2/|1-z_2e^{i\theta}|^2.$
	  This makes $p(x) = (1-z_1x)(x - \bar z_1),$ $U_1 = (1-\bar z_2 x)$ and $V_1 = (1- z_2 x).$
	  Applying Proposition~\ref{prop:Baxter},
	  \begin{align*}
	    D_{N}\left( 
	    \frac{
	      p(e^{i\theta})e^{-i\theta}}
	      {U_\ell(e^{-i\theta})V_m(e^{i\theta})}
	      \right)
	      &=(-1)^{N}p(0)^{N+1} 
	       \frac{p(\bar z_2)}{(1-|z_2|^2)p(0)}\\
	      & \;\;\;\cdot
	      \left(
	      \frac{-(1-\bar z_1 z_2)}{(1-|z_1|^2) (\bar z_1)^{N}(\bar z_1 - \bar z_2)}
	      +\frac{(z_1 - z_2)(z_1)^{N}}{(1-|z_1|^2)(1-\bar z_2 z_1)}
	      \right) \\
	      &=
	      \frac{|1-\bar z_1 z_2|^2}{(1-|z_1|^2) (1-|z_2|^2)}
	      -\frac{|z_1 - z_2|^2|z_1|^{2N}}{(1-|z_1|^2) (1-|z_2|^2)},
	    \end{align*}
	    for all $N$ sufficiently large.
	    The first term is exactly $\Exp[ e^{2\GF(z_1) - 2\GF(z_2)}]$.  Note for $z_1$ at hyperbolic distance $\log N + O(1)$ from the origin, the second term is on the same order as the first and so is non-negligible. 
	  \end{example}

	  Finally, we reformulate the Toeplitz determinant identity in a way that is most applicable to how we use it.  Namely, suppose we have two finite sets $\mathbf{z}, \mathbf{y} \in \D$ of jointly pairwise distinct points.  We wish to evaluate the Toeplitz determinant with symbol
	  \[
	    \frac{\prod_{z \in \mathbf{z}} |1-ze^{i\theta}|^2 }{
	      f_1(e^{-i\theta})f_2(e^{i\theta})\prod_{y \in \mathbf{y}} |1-ye^{i\theta}|^2
	    }
	  \]
	  for high degree polynomials $f_1$ and $f_2$ without zeroes in $\D.$

	  \begin{proposition}
          Fix $k \in \mathbb{Z},$ and let $\mathbf{z}, \mathbf{y} \subset \D$ 
	  be jointly pairwise distinct points with 
	  $|\mathbf{z}| = |\mathbf{y}|\corO{= k}.$
	  Let $f_1$ and $f_2$ be polynomials of degree strictly less than $N/2$ with $f_1(0)=f_2(0) = 1$ and without zeros in $\bar \D$ so that $z \mapsto f_1(z^{-1})$ does not vanish on $\mathbf{z}.$ \corE{Let $\gamma$ wind positively once along the unit circle.} \corO{Then},
	    \begin{align*}
	      D_{N} \hspace{0.5cm}&\hspace{-0.5cm}
	      \left( 
	      \frac{
		\prod_{z \in \mathbf{z}} |1-ze^{i\theta}|^2
	      }
	      {
		f_1(e^{-i\theta})f_2(e^{i\theta})
		\prod_{y \in \mathbf{y}} |1-ye^{i\theta}|^2
	      }
	      \right)
	      =\Exp e^{\Bias{\GF}} \\
	      &
	      \cdot \exp\left( 
	      \frac{1}{2\pi i} \int_{\gamma} \log\left(\frac{1}{f_2(z)}\right) \frac{d}{dz} \log( f_1(z^{-1}) )\,dz
	      \right)
	      \cdot \frac{
		\prod_{z \in \mathbf{z}} f_1(z) f_2(\bar z)
	      }{
		\prod_{y \in \mathbf{y}} f_1(y) f_2(\bar y)
	      } \\
	      &\cdot
	      \sum_{\substack{S_1,S_2 \subset \mathbf{z}\\ |S_1|=|S_2|}}
          (-1)^{|S_1|}
          \!\cdot\!
	      \left[
		\prod_{S_1} \frac{ z^N f_1(z^{-1})}{f_2(z)}
	      \right]
          \!\cdot\!
	      \left[ 
		\prod_{S_2} \frac{ \bar z ^N f_2(\bar z^{-1})}{f_1(\bar z)}
	      \right]
          \!\cdot\!
	      c^{\mathbf{y}}(S_1,S_2)
          \!\cdot\!
	      c_{\mathbf{z}}(S_1,S_2),
	    \end{align*}
	    where
	    \begin{align*}
	      c^{\mathbf{y}}(S_1,S_2) &=
	      \prod_{
		\substack{
		  z \in S_1 \\
		  y \in \mathbf{y}
		}}
		T_y(z)
		\cdot
		\prod_{
		  \substack{
		    z \in S_2 \\
		    y \in \mathbf{y}
		  }}
		  T_{\bar y}(\bar z)
		  , \text { and } \\
		  |c_{\mathbf{z}}(S_1,S_2)| &=
		  \prod_{
		    \substack{
		      z \in S_1 \\
		      w \in S_1^c
		    }}
		    \frac{1}{|T_z(w)|}
		    \cdot
		    \prod_{
		      \substack{
			z \in S_2 \\
			w \in S_2^c
		      }}
		      \frac{1}{|T_{\bar z}(\bar w)|}
		      \cdot
		      \frac{1}{\Exp e^{ 
			2 \sum_{S_1 \cap S_2^c} \GF[z]
			-2 \sum_{S_2 \cap S_1^c} \GF[z]
		      }},
		    \end{align*}
		    \corO{and the sets $S_i^c$ denote the complements of $S_i$ in $\bf{z}$.}
		    The functions $T_y(z)$ are the disk automorphisms, as in \eqref{eq:diskautomorphism}.  An exact expression for $c_{\mathbf{z}}(S_1,S_2)$ is given in \eqref{eq:exc2}.  When $S_1=S_2,$ $c_{\mathbf{z}}(S_1,S_2) > 0.$
		    \label{prop:exact}
		  \end{proposition}
          \begin{remark}
              The requirement that $\mathbf{y}$ be a set, and not a multiset, is not important. Indeed, by continuity, it is possible to deform the points in $\mathbf{y}$ continuously.  By sending these points to $0,$ it is also possible to extend the statement to the case that $|\mathbf{y}| \leq |\mathbf{z}|,$ by adding some number of zeros to $\mathbf{y}.$
          \end{remark}
	  The dominant term in this expansion is given by $S_1=S_2=\emptyset.$  The coefficient $|c^{\mathbf{y}}(S_1,S_2)| \leq 1,$ as the functions $T_y$ map to the unit disk.  On the other hand, $c_{\mathbf{z}}(S_1,S_2)$ could in principle be very large if some $w,z \in \mathbf{z}$ are very close to each other. 
	  Note that the Toeplitz determinant is continuous as $w \to z,$ so this representation hides some cancellation.  For our purposes, however, it will be enough just to bound $c_{\mathbf{z}}(S_1,S_2)$ above, which we do using the easily verified identity
		  \begin{equation}
		    \label{eq:Tdist}
		    |T_w(y)| = \tanh(\dH(w,y)/2).
		  \end{equation}
		  Note also that by Jensen's inequality, the expectation that appears in the modulus of $c_{\mathbf{z}}(S_1,S_2)$ is larger than $1,$ and so can be disregarded for the purposes of an upper bound.  
          For the purposes of evaluating exponential moments, these estimates can be summarized as follows.
		  \begin{corollary}
		    Let $k \in \mathbb{Z}$ be fixed.
		    There is a $C=C(k)$ sufficiently large that the following holds.
            Let $\mathbf{z}, \mathbf{y} \in \D$ be jointly pairwise distinct points with $|\mathbf{y}| \leq |\mathbf{z}| = k.$ Define
		    \begin{equation*}
		      \Delta = \max_{z \in \mathbf{z}} e^{-N\exp(-\dH(0,z))}\prod_{
			\substack{x \in \mathbf{z} \\ 
		      x \neq z}} 
		      \coth(\dH(x,z)/2).
		    \end{equation*}
		    Then,
		    \(
		      {\Exp e^{\Bias{\CUEF}}}
		      \geq 
		      {\Exp e^{\Bias{\GF}}}
              (1 - C(\Delta^2 + \Delta^{2k})).
		    \)
		    Further, if we write $\mathbf{z} = \left\{ w \right\} \cup \mathbf{z}',$ and define
		    \begin{equation*}
		      \Delta' = \max_{z \in \mathbf{z}'} e^{-N\exp(-\dH(0,z))}\prod_{
			\substack{x \in \mathbf{z} \\ 
		      x \neq z}} 
		      \coth(\corO{\dH(x,z)}/2).
		    \end{equation*}
		    Then
		    \[
		      \frac
		      {\Exp e^{\Bias{\CUEF}}}
		      {\Exp e^{\Bias{\GF}}}
		      = 1 -
		      |w|^{2N}
		      \frac{
			\prod_{y \in \mathbf{y}} \tanh(\dH(w,y)/2)^2
		      }{
			\prod_{z \in \mathbf{z}} \tanh(\dH(w,z)/2)^2
		      }
		      + \xi,
		    \]
		    where $|\xi| \leq C (\Delta')(1+\Delta)^{2k-1}.$
		    \label{cor:exact}
		  \end{corollary}
		  \begin{proof}
		    In the notation of Proposition~\ref{prop:exact}, we have the finite expansion
		    \[
		      \frac{
			\Exp\left[ 
			  e^{\Bias{\CUEF}}
			\right]}
			{
			  \Exp\left[ 
			    e^{\Bias{\GF}}
			  \right]
			}
			=
			\sum_{\substack{S_1,S_2 \subset \mathbf{z}\\ |S_1|=|S_2|}}
            (-1)^{|S_1|}
			\left[
			  \prod_{S_1} z^N
			\right]
			\left[ 
			  \prod_{S_2} \bar z^N
			\right]
			c^{\mathbf{y}}(S_1,S_2)
			c_{\mathbf{z}}(S_1,S_2).
		      \]
		      The $S_1 = S_2 = \emptyset$ term is $1.$  
		      
		      \corO{Writing $|z|\leq e^{-\exp(-\dH(0,z))}$,
		    each summand} can be estimated by 
		    $\Delta^{|S_1|+|S_2|},$ from which the first estimate follows. The $S_1 = S_2 = \left\{ w \right\}$ term gives what is displayed in the corollary, \corE{recalling that $c_{\mathbf{z}}(S_1,S_1) > 0$ and $c^{\mathbf{y}}(S_1,S_1) > 0$ as is claimed in Proposition~\ref{prop:exact}}.  As $T_w(y)$ maps to the unit disc, we have $|c^{\mathbf{y}}(S_1,S_2)| \leq 1.$  By Jensen, we have 
		      $\Exp e^{  
			2 \sum_{S_1 \cap S_2^c} \GF[z]
			-2 \sum_{S_2 \cap S_1^c} \GF[z]
		      } \geq 1.$
		      Terms in the sum where either $S_1$ or $S_2$ are not equal to $\left\{ w \right\}$ can be estimated by $(\Delta')(\Delta)^{|S_1|+|S_2|-1}.$  The estimate now follows.
		    \end{proof}

		    \begin{proof}[Proof of Proposition \ref{prop:exact}]
		      The proof is nothing but the combination of Proposition~\ref{prop:Baxter} and Lemma~\ref{lem:residues} and some reorganization.  In particular we apply these results with 
		      \begin{alignat*}{2}
			p(\omega) &= 
			\prod_{z \in \mathbf{z}} (1-z\omega)(\omega - \bar{z}), 
			~~& &
			\\
			q_1(\omega) &=
			\prod_{y \in \mathbf{y}} (1-\bar y \omega),
			& U_\ell(\omega) &=
			f_1(\omega)q_1(\omega), \\
			q_2(\omega) &= \prod_{y \in \mathbf{y}} (1- y \omega),
			& V_m(\omega) &=
			f_2(\omega)q_2(\omega).
		      \end{alignat*}
		      By how we have chosen $f_1$ and $f_2,$ the hypotheses of both Proposition~\ref{prop:Baxter} and Lemma~\ref{lem:residues} are satisfied, and so we can write
		      \begin{align*}
			D_{N}\left( 
			\frac{
			  p(e^{i\theta})e^{-ik\theta}}
			  {U_\ell(e^{-i\theta})V_m(e^{i\theta})}
			  \right)
			  \hspace{-1cm}&\hspace{1cm}
			  =(-1)^{\binom{k}{2} + kN}
			  (-1)^{\binom{k}{2}}
			  p(0)^{N+k} 
			&\biggr\} I_1 \\
			&\cdot
			\exp\left( 
			\frac{1}{2\pi i} \int_{\gamma} \log\left(\frac{p(z)}{p(0)V_m(z)}\right) \frac{d}{dz} \log( U_\ell(z^{-1}) )\,dz
			\right) 
		      &\biggr\} I_2 \\
		      & \cdot
		      \sum_{\substack{S \subseteq [2k] \\ |S|=k}}
		      \prod_{i \in S}
		      \left[ \frac{V_m(c_i)}
		      {a c_i^{N+k}U_\ell(c_i^{-1})\prod_{j \not\in S} (c_j - c_i)}\right],
		    &\biggr\} I_3
		  \end{align*}
		  where $a = (-1)^k \prod_{z \in \mathbf{z}} {z}$ and where 
		  \[
		    \left\{ c_1,c_2,\dots,c_{2k} \right\}
		    =
		    \left\{ z^{-1} : z \in \mathbf{z} \right\}
		    \cup
		    \left\{ \bar z : z \in \mathbf{z} \right\}.
		  \]

		  In $I_1,$ we are left with
		  \[
		    I_1 = (-1)^k \left( \prod_{z \in \mathbf{z}} \bar z \right)^{N+k}.
		  \]
		  In $I_2,$ we want to separate the terms involving only $f_1$ and $f_2$ from the rest.  Here we use \eqref{eq:allcross}, and we write $f_1(\omega) = \prod_{a \in \mathbf{a}} (1-a\omega)$ for some finite set $\mathbf{a} \in \C.$  Then we can write
		  \begin{align*}
		    I_2 =&
		    \prod_{a \in\mathbf{a}}
		    \frac{1}{f_2(a)}
		    \cdot
		    \prod_{y \in\mathbf{y}}
		    \frac{p(\bar y)}{p(0)q_2(\bar y)f_2(\bar y)}
		    \cdot
		    \prod_{a \in\mathbf{a}}
		    \frac{p(a)}{p(0)q_2(a)}.
		  \end{align*}
		  These can be rewritten by expanding the definitions of $p$ and $q_2$ and reordering the products to give
		  \begin{align*}
		    I_2 =&
		    \exp\left( 
		    \frac{1}{2\pi i} \int_{\gamma} \log\left(\frac{1}{f_2(z)}\right) \frac{d}{dz} \log( f_1(z^{-1}) )\,dz
		    \right) 
		    \cdot
		    \frac{\prod_{z \in\mathbf{z}}
		    f_1(z)f_2(\bar{z})}
		    {\prod_{y \in\mathbf{y}}
		  f_1(y) f_2(\bar y)} \\
		  &\cdot
		  \prod_{\substack{
		    z \in\mathbf{z}, 
		    y \in\mathbf{y} 
		  }}
		  |1-z \bar y|^2
		  \cdot
		  \prod_{y,w, \in \mathbf{y}}
		  \frac{1}{1-y\bar w}
		  \cdot
		  \prod_{z \in\mathbf{z}}
		  \frac{q_1(\bar{z}^{-1})}{q_2(\bar z)}
		  \frac{f_1(\bar{z}^{-1})}{f_2(\bar z)}.
		\end{align*}
		Meanwhile we can expand $\Exp e^{\Bias{\GF}}$ using the covariance formula~\eqref{eq:gfcov} 
		\[
		  \Exp e^{\Bias{\GF}}
		  =\frac{
		    \prod_{\substack{
		      z \in\mathbf{z},
		      y \in\mathbf{y} 
		    }}
		    |1-z \bar y|^2
		  }
		  {
		    \prod_{y,w \in \mathbf{y}}
		    (1-y\bar w)
		    \prod_{z,w \in \mathbf{z}}
		    (1-z\bar w)
		  }.
		\]
		Hence, we have that 
		\begin{equation}
		  I_2 =
		  \exp\left( 
		  \frac{1}{2\pi i} \int_{\gamma} \log\left(\frac{1}{f_2(z)}\right) \frac{d}{dz} \log( f_1(z^{-1}) )\,dz
		  \right) 
		  \cdot
		  \frac{\prod_{z \in\mathbf{z}}
		  f_1(z)f_2(\bar{z})}
		  {\prod_{y \in\mathbf{y}}
		f_1(y) f_2(\bar y)} 
		\Exp e^{\Bias{\GF}}
		\tilde I_2,
		\label{eq:I2}
	      \end{equation}
	      where we define
	      \begin{equation*}
		\tilde I_2
		=
		\prod_{z,w \in \mathbf{z}}
		(1-z\bar w)
		\cdot
		\prod_{z \in\mathbf{z}}
		\frac{q_1(\bar{z}^{-1})}{q_2(\bar z)}
		\frac{f_1(\bar{z}^{-1})}{f_2(\bar z)}.
	      \end{equation*}

	      Hence, it remains to show that for an appropriate $c_{\mathbf{z}}(S_1,S_2)$ we have
	      \begin{equation}
		\corO{I_1\tilde I_2 I_3}
		=
		\sum_{\substack{S_1,S_2 \subset \mathbf{z}\\ |S_1|=|S_2|}}
		\corO{(-1)^{|S_1|}}
		\left[
		  \prod_{S_1} \frac{ z^N f_1(z^{-1})}{f_2(z)}
		\right]
		\left[ 
		  \prod_{S_2} \frac{ \bar z ^N f_2(\bar z^{-1})}{f_1(\bar z)}
		\right]
		\corO{c^{\mathbf{y}}(S_1,S_2)}
		c_{\mathbf{z}}(S_1,S_2).
		\label{eq:ex1}
	      \end{equation}
	      This amounts to expanding the definition of $I_3.$  Recall the sum in $I_3$ was over some subset $S$ of indices $[2k]$ corresponding to an arbitrary subset of $\mathbf{z}^{-1} \cup \bar{\mathbf{z}}$ of size $k.$  Hence, we may as well consider $S$ a subset of $\mathbf{z}^{-1} \cup \bar{\mathbf{z}}.$  Then, we define
	      \begin{align*}
		S_1 &= (S \cap \mathbf{z}^{-1})^{-1} \\
		S_2 &= \mathbf{z} \setminus \overline{(S \cap \bar{\mathbf{z}})}.
	      \end{align*}
	      \corO{Recalling that
	      $S_1^c$ and $S_2^c$ are} the complements of these sets within $\mathbf{z},$ we have
	      \begin{align*}
		I_3 &=
		\sum_{\substack{S_1,S_2 \subset \mathbf{z}\\ |S_1|=|S_2|}}
		\frac{
		  \prod_{z \in S_1}
		  \left[ \frac{z^{N+k}f_2(z^{-1})q_2(z^{-1})}
		  {a f_1(z)q_1(z)}\right]
		  \prod_{z \in S_2^c}
		  \left[ \frac{f_2(\bar{z})q_2(\bar{z})}
		  {a \bar z^{N+k}f_1(\bar z^{-1})q_1(\bar z^{-1})}\right]}
		  { 
		    \prod_{\substack{z \in S_1 \\ w \in S_1^c}} (w^{-1} - z^{-1})
		    \prod_{\substack{z \in S_1 \\ w \in S_2}} (\bar w - z^{-1})
		    \prod_{\substack{z \in S_2^c \\ w \in S_1^c}} (w^{-1} - \bar z)
		    \prod_{\substack{z \in S_2^c \\ w \in S_2}} (\bar w - \bar z)
		  }.
		\end{align*}
		Combining the $a = \prod_{z \in \mathbf{z}}(-z)$ terms with the denominator and distributing appropriately, we can remove the inverses, leaving
		\begin{align*}
		  I_3 &=
		  \sum_{\substack{S_1,S_2 \subset \mathbf{z}\\ |S_1|=|S_2|}}
		  \frac{
		    (-1)^{|S_1^c|}
		    \prod_{z \in S_1}
		    \left[ \frac{z^{N+k}f_2(z^{-1})q_2(z^{-1})}
		    {f_1(z)q_1(z)}\right]
		    \prod_{z \in S_2^c}
		    \left[ \frac{f_2(\bar{z})q_2(\bar{z})}
		    {\bar z^{N+k}f_1(\bar z^{-1})q_1(\bar z^{-1})}\right]}
		    { 
		      \prod_{\substack{z \in S_1 \\ w \in S_1^c}} (w - z)
		      \prod_{\substack{z \in S_1 \\ w \in S_2}} (1- z \bar w )
		      \prod_{\substack{z \in S_2^c \\ w \in S_1^c}} (1 - w \bar z)
		      \prod_{\substack{z \in S_2^c \\ w \in S_2}} (\bar w - \bar z)
		    }.
		  \end{align*}
		  Including the $\tilde I_2$ and $I_1$ terms, we have
		  \begin{align*}
		    I_1 \tilde I_2 I_3=
		    \sum_{\substack{S_1,S_2 \subset \mathbf{z}\\ |S_1|=|S_2|}}
		    (-1)^{|S_1|}
		    &\prod_{z \in S_1}
		    \left[ \frac{z^{N+k}f_2(z^{-1})q_2(z^{-1})}
		    {f_1(z)q_1(z)}\right] \\
		    &\cdot \prod_{z \in S_2}
		    \left[ \frac{\bar z^{N+k}f_1(\bar z^{-1})q_1(\bar z^{-1})}
		    {f_2(\bar{z})q_2(\bar{z})}\right]
		    c_{\mathbf{z}}(S_1,S_2),
		  \end{align*}
		  where we define $c_{\mathbf{z}}(S_1,S_2)$ to be 
		  \begin{equation}
		    c_{\mathbf{z}}(S_1,S_2)
		    =
		    \frac{ 
		      \prod_{z,w \in \mathbf{z}}
		      (1-z\bar w)
		    }
		    { 
		      \prod_{\substack{z \in S_1 \\ w \in S_1^c}} (w - z)
		      \prod_{\substack{z \in S_1 \\ w \in S_2}} (1- z \bar w )
		      \prod_{\substack{z \in S_2^c \\ w \in S_1^c}} (1 - w \bar z)
		      \prod_{\substack{z \in S_2^c \\ w \in S_2}} (\bar w - \bar z)
		    }
		    \label{eq:exc1}
		  \end{equation}
		  In light of \eqref{eq:ex1}, it just remains to observe that
		  \begin{align*}
		    c^{\mathbf{y}}(S_1,S_2)
		    &=
		    \prod_{\substack{z \in S_1 \\ y \in \mathbf{y}}}
		    \frac{z-y}{1-\bar y z}
		    \cdot
		    \prod_{\substack{z \in S_2 \\ y \in \mathbf{y}}}
		    \frac{\bar z- \bar y}{1-y \bar z} 
		    =
		    \prod_{z \in S_1}
		    \left[ \frac{z^{k}q_2(z^{-1})}
		    {q_1(z)}\right]
		    \prod_{z \in S_2}
		    \left[ \frac{\bar z^{k}q_1(\bar z^{-1})}
		    {q_2(\bar{z})}\right].
		  \end{align*}

		  We now give an alternate representation of $c_{\mathbf{z}}(S_1,S_2)$ using which the bound in the statement of the 
		  \corO{proposition} will follow.  We extract the M\"obius transformation terms first, and write
		  \begin{equation}
		    \tilde{c}(S_1,S_2)
		    =
		    \frac{ 
		      \prod_{z,w \in \mathbf{z}}
		      (1-z\bar w)
		    }
		    { 
		      \prod_{\substack{z \in S_1^c \\ w \in S_1}} (1-z\bar w)
		      \prod_{\substack{z \in S_1 \\ w \in S_2}} (1- z \bar w )
		      \prod_{\substack{z \in S_1^c \\ w \in S_2^c}} (1 - z \bar w)
		      \prod_{\substack{z \in S_2^c \\ w \in S_2}} (1 - z\bar w)
		    }
		    \label{eq:exc2},
		  \end{equation}
		  so that
		  \[
		    {c_{\mathbf{z}}}(S_1,S_2)
		    =
		    \tilde{c}(S_1,S_2)
		    \frac{1}{
		      \prod_{\substack{z \in S_1 \\ w \in S_1^c}} T_w(z)
		      \prod_{\substack{z \in S_2^c \\ w \in S_2}} T_{\bar w}(\bar z)
		    }.
		  \]

		  Cancelling and combining like terms gives
		  \begin{equation*}
		    \tilde{c}(S_1,S_2)
		    =
		    \frac{
		      \prod_{\substack{z \in S_2^c \\ w \in S_1 \cap S_2}} (1-\bar z w)
		    }{
		      \prod_{\substack{z \in S_2^c \\ w \in S_1 \cap S_2}} (1-z\bar w)
		    }
		    \frac{
		      \prod_{\substack{z \in S_1\setminus S_2 \\ w \in S_1 \setminus S_2}} (1-z\bar w)
		      \prod_{\substack{z \in S_2\setminus S_1 \\ w \in S_2 \setminus S_1}} (1-z\bar w)
		    }{
		      \prod_{\substack{z \in S_1\setminus S_2 \\ w \in S_2 \setminus S_1}} |1-z\bar w|^2
		    }.
		  \end{equation*}
		  Hence we conclude that
		  \begin{equation*}
		    |\tilde{c}(S_1,S_2)|
		    =\frac{1}{\Exp 
		      \exp\left( 
		      2 \sum_{S_1 \cap S_2^c} \GF[z]
		      -2 \sum_{S_2 \cap S_1^c} \GF[z]
		      \right)
		    },
		  \end{equation*}
		  which completes the proof.
		\end{proof}

		\section{Smoothed total variation comparison to normal} 
		\label{sec:tv}

        In this section, we prove Proposition~\ref{prop:clt}, using the Toeplitz determinant identities developed in Section~\ref{sec:Baxter}.  We will use these identities to estimate the Fourier transforms of finite dimensional marginals $\CUEF$ under exponential biases.  
	\corO{To fit with 
	Proposition~\ref{prop:clt},
      } we would like to take $f_2(\omega)$ \corO{in Proposition 
	\ref{prop:exact}}
	to be something like $e^{i \sum_z \xi_z\log(1-\omega z)}$ for some collection of $z \in \D.$  As we require $f_2$ to be a polynomial, our first step is to give a quantitative estimate on how well this can be approximated by a polynomial of degree $N/2.$

        Here we will use in a nontrivial way the requirement in Proposition~\ref{prop:clt} on the locations of the points $z.$  As there, let $R \in \mathbb{N}.$ 
        For any $\left\{ \omega_j \right\}_1^R \subset \T$ and any $\left\{ \xi_{h,j} \right\}_{\substack{ h=1,\dots,d \\\ j= 1,\dots,R}}$ real numbers, define $L : \D \to \C$ by
        \[
            L(z) = \sum_{\substack{ h=1,\dots,d \\\ j= 1,\dots,R}}
            \xi_{h,j} \frac{\log( 1 - \zeta_h \omega_j z)}{2}.
        \]

        Let $\Bias$ be an exponential bias as in the statement of Proposition~\ref{prop:clt}.
	  We will start by estimating the Fourier transform
        \[
            \varphi(\left\{ \xi_{h,j} \right\})
            \!:= \Exp\!\left[\! 
                e^{i \sum_{h,j} \xi_{h,j} \CUEF(
		\zeta_h
		\omega_j) + \Bias{\CUEF}}
           \! \right]
           \! =\!D_{N} 
           \! \left(\! 
            \frac{
                \prod_{z \in \mathbf{z}} |1-ze^{i\theta}|^2
            }
            {
                e^{-i \overline{L(e^{i\theta})}
                -i L(e^{i\theta})}
                \prod_{y \in \mathbf{y}} |1-ye^{i\theta}|^2
            }
            \!\right)\!.
        \]
	  We will then apply Fourier inversion to produce an estimate on the difference between the  density of $\left\{ (\WN+\CUEF)(\zeta_h\omega_j) \right\}_{\substack{ h=1,\dots,d \\\ j= 1,\dots,R}}$ and $\left\{ (\WN+\GF)(\zeta_h\omega_j) \right\}_{\substack{ h=1,\dots,d \\\ j= 1,\dots,R}}.$

        We now define $f_2$ as the degree $A = [N/2]-1$ truncation of the Taylor series expansion at $0$ of $e^{-iL(z)}.$  The remainder $e^{-iL(z)} - f_2(z)$ can be represented by the following contour integral formula
        \[
            g_2(x) :=
            {e^{-iL(x)}}
            -f_2(x)
            =
            \frac{1}{2\pi i}
            \int_{|z|=1}
            \frac{e^{-iL(z)}}
            {z-x} \left( \frac{x}{z} \right)^A {dz}.
        \]
        Also define $f_1(\bar z)$ as the degree $A$ truncation of the Taylor series of $e^{-i \overline{L(z)}}$ in $\bar z,$
	\corO{and set $g_1(x):=e^{i\overline{L(x)}}-f_1(\overline{x})$.}
        \begin{lemma}
            Fix $R \in \mathbb{N}.$ Uniformly in $d\in \mathbb{N}$ with $1 \leq d \leq \log N,$ in $\left\{ \omega_j \right\}_1^R \subset \T$ in $\left\{ \xi_{h,j} \right\}_{\substack{ h=1,\dots,d \\\ j= 1,\dots,R}} \subset \R$ and in $r \leq \zeta_d^{-1} - N^{-1},$
            \[
                \corO{\sup_{|x|=r}
                |g_j(x)|
                \ll
		\sinh(\corO{\frac{\pi}{2}}{\textstyle \sum_{h,j} |\xi_{h,j}|})
		|r \zeta_d|^{A}, \quad j=1,2.}
            \]
            \label{lem:approx}
        \end{lemma}
        \begin{proof}
	\corO{We only consider the case $j=2$, the argument for $j=1$ is similar.}
	  The choice of branch was not important to the definition of $e^{iL(z)}$ on the unit disc.  We will estimate $g_2$ by deforming the contour of integration outside of the unit disc, however, and so we set the definition here.  We take all the logarithms $\log(1-y)$ to be real for $y\in (0,1),$ so that there is a discontinuity along the ray $\corO{y\in}[1,\infty).$  Hence the imaginary part of the logarithm approaching the strip from above \
	    \corO{$[1,\infty)$} is $-\pi$ and is $\pi$ coming from below.
            
                This allows $e^{iL(z)}$ to be defined on the complex plane minus $R$ rays.  These rays start at $\{\zeta_d^{-1}\omega_j^{-1}\}_{j=1}^R$ and emanate out from the origin.  Since the imaginary parts of the logarithms are bounded, we also have that $e^{iL(z)}$ remains bounded on this domain.  Further, once $A > 0,$ we have that
            \[
                \frac{e^{-iL(z)}}
                {z-x} \left( \frac{x}{z} \right)^A
                =O(|z|^{-2}),
            \]
            with $x$ held fixed.  Therefore, we can deform the contour to run twice along every ray, once towards the origin and once away, giving us the representation
            \[
                g_2(x)
                =\frac{1}{2\pi i}
                \sum_{j=1}^R
                \int_{\zeta_d^{-1}}^\infty
                \frac{
                    e^{-iL_{+}(\bar{\omega}_j t)}
                    -e^{-iL_{-}(\bar{\omega}_j t)}
                }
                {\bar{\omega}_jt-x} \left( \frac{x}{\bar{\omega}_j t} \right)^A \bar{\omega}_j dt,
            \]
            where we define 
            $L_+(z) = \lim_{\theta \downarrow 0} L(ze^{i\theta})$ and $L_-(z) = \lim_{\theta \uparrow 0} L(ze^{i\theta}).$

            This integral we further decompose according to the value of $t.$  For $t \in (\zeta_{h}^{-1}, \zeta_{h-1}^{-1}),$ we have that 
            \[
                e^{-iL_{+}(\bar{\omega}_j t)}
                -e^{-iL_{-}(\bar{\omega}_j t)}
                =
                e^{i\beta_{h,j}(t)}
                \left( 
		\corO{   e^{-\frac{\pi}{2}\sum_{k=h}^d \xi_{k,j}}
		-e^{\frac{\pi}{2}\sum_{k=h}^d \xi_{k,j}}
	      }
                \right),
            \]
            for some real function $\beta_{h,j}(t).$  Hence, estimating $g_2(x)$ for $|x|=r,$ we get that
            \begin{align*}
                |g_2(x)|
                &\leq
                \frac{1}{\pi}
                \sum_{j=1}^R
                \sum_{h=1}^d
		|\sinh(\corO{\frac{\pi}{2}}{\textstyle \sum_{k=h}^d \xi_{k,j})}|
                \int_{\zeta_h^{-1}}^{\zeta_{h-1}^{-1}}
                \frac{r^A}{t-r}
                \frac{dt}{t^A} \\
                &\ll
		\sinh(\corO{\frac{\pi}{2}}
		{\textstyle \sum_{h,j} |\xi_{h,j}|})
                \int_{\zeta_d^{-1}}^\infty 
                \frac{r^A}{t-r}
                \frac{dt}{t^A}. \\
                \intertext{Since $d \leq \log N,$ $|\zeta_{d}|^{-1}-r \gg N^{-1}.$ Hence, bounding the $t-r$ in the integrand in supremum, we get that}
                |g_2(x)|
                &\ll
                \sinh(\corO{\frac{\pi}{2}}
		{\textstyle \sum_{h,j} |\xi_{h,j}|})
                |r \zeta_d|^{A}.
            \end{align*}
            \end{proof}

	    We will be interested in  $\sum_{h,j}|\xi_{h,j}|^2 \ll \log N \log\log N.$  Hence the $\ell^1$ norm of $\left\{ \xi_{h,j} \right\}$ is $\sum_{h,j}|\xi_{h,j}| \ll \sqrt{R}\log N \sqrt{\log\log N}.$
	    We would like that  $|\zeta_d|^A$ is sufficiently small to absorb this.  This occurs as soon as $d \leq \log N - m\log\log N$ for some fixed $m > 1,$ as we have that
            \[
                A \log |\zeta_d| \asymp -e^{\log N - d}.
            \]
            This estimate allows us to relatively easily approximate $\varphi(\left\{ \xi_{h,j} \right\})$ by approximating the symbol in the Toeplitz determinant.
            \begin{lemma}
                Fix $m >1$ a real number and fix $R \in \mathbb{N}.$
                Uniformly in $\left\{ \xi_{h,j} \right\} \subset \R$ having $\sum_{h,j} |\xi_{h,j}|^2 \ll (\log N)^m$ and uniformly in $d \leq \log N - m\log\log N,$
                \[
                    \varphi(\left\{ \xi_{h,j} \right\})
                    =
                    D_{N}\left( 
                    \frac{
                        \prod_{z \in \mathbf{z}} |1-ze^{i\theta}|^2
                    }
                    {
                        f_1(e^{-i\theta})f_2(e^{i\theta})
                        \prod_{y \in \mathbf{y}} |1-ye^{i\theta}|^2
                    }
                    \right)
                    +e^{-\Omega((\log N)^m)}\cdot\Exp\left[ e^{\Bias{\CUEF}} \right].
                \]
                \label{lem:SymbolApprox}
            \end{lemma}
            \begin{proof}
                We will from here on suppress the argument of $\varphi$ and write
                \[
                    \varphi
                    =
                    \Exp\left[ 
                        e^{\Bias{\CUEF}}\prod_{h=1}^N e^{i2\Re L(e^{i\theta_h})}
                    \right].
                \]
                We simply the replace the terms one-by-one in this product.  Each $e^{i2\Re L(e^{i\theta_h})}$ has modulus $1,$ and by Lemma~\ref{lem:approx}, each difference
                \[
                    e^{i2\Re L(e^{i\theta_h})}
                    -
                    \frac{1}{
                            f_1(e^{-i\theta_h})
                            f_2(e^{i\theta_h})
                        }
                    =O(e^{-\Omega((\log N)^m)})
                \]
                in supremum over all possible $\theta_h.$
                This decays so quickly it can absorb the extra polynomial factors in $N$ that arise as a result of replacing all $N$ terms in the product.
%
            \end{proof}

            We would like to apply Proposition~\ref{prop:exact} to estimate $\varphi,$ but we must first check that $f_1(0)=f_2(0)=1$ and that $f_1$ and $f_2$ are zero-free in $\bar{\D}.$  The fact that they are $1$ at the origin arises from their being truncations of $e^{-iL(z)},$ which is $1$ at the origin.  To check that they are $0$-free, we have that $|\Im(L)| \ll \sum_{h,j} |\xi_{h,j}|$ uniformly in $\D.$  Hence we can estimate for $z \in \bar{\D},$ 
            \[
                |f_2(z)| \geq | e^{-iL(z)}| - |g_2(z)|
                \geq
                e^{-O( \sum_{h,j} |\xi_{h,j}|)}
                -\sup_{\omega \in \T} |g_2(\omega)|,
            \]
            by the maximum principle. In particular, under the hypotheses of Lemma~\ref{lem:SymbolApprox}, we have 
	      $\sum_{h,j} |\xi_{h,j}| \ll R^{1/2} (\log N)^{ (1+m)/2}$
	    while $|g_2(\omega)| = e^{-\Omega( (\log N)^m)},$ so that $f_2$ is $0$-free in $\bar{\D}.$  Furthermore, we can also write, uniformly in $z\in \bar{\D},$
            \[
                f_2(z) = e^{-iL(z)}(1 + e^{-\Omega( (\log N)^m)}).
            \]

            \begin{lemma}
                Fix $m >1$ a real number and fix $R \in \mathbb{N}.$
                Uniformly in $\left\{ \xi_{h,j} \right\} \subset \R$ having $\sum_{h,j} |\xi_{h,j}|^2 \ll (\log N)^{m},$ uniformly in $d \leq \log N - m\log\log N$ and uniformly in $\mathbf{z} \subset \D$ of a given finite cardinality having pairwise hyperbolic separation 
		  $N^{-1},$
                \begin{align*}
                    \varphi
                    &=\Exp\left[ 
		      e^{i \sum_{h,j} \xi_{h,j} \GF(\corO{\zeta_h}\omega_j) + \Bias{\GF}}
                    \right]
                    \\
                    &\cdot
                    \sum_{\substack{S_1,S_2 \subset \mathbf{z}\\ |S_1|=|S_2|}}
                    (-1)^{|S_1|}
                    \!\cdot\!
                    \left[
                        \prod_{S_1} \frac{ z^N e^{-i\overline{L(\bar{z}^{-1})}}}{e^{-i{L(z)}}}
                    \right]
                    \!\cdot\!
                    \left[ 
                        \prod_{S_2} 
                        \frac{ \bar z^N e^{-iL(\bar z^{-1})}}
                        {e^{-i\overline{L({z})}}}
                    \right]
                    \!\cdot\!
                    c^{\mathbf{y}}(S_1,S_2)
                    \!\cdot\!
                    c_{\mathbf{z}}(S_1,S_2),
                    \\
                    &+e^{-\Omega((\log N)^m)}\cdot\Exp\left[ e^{\Bias{\CUEF}} \right].
                \end{align*}
                \label{lem:cf}
            \end{lemma}

            \begin{proof}
                We start by applying Lemma~\ref{lem:SymbolApprox} and Proposition~\ref{prop:exact}.  Together, these give
                \begin{equation}
                    \label{eq:cf0}
                    \begin{aligned}
                    \varphi
                    &=\Exp e^{\Bias{\GF}}
                    \cdot \exp\left( 
                    \frac{1}{2\pi i} \int_{\gamma} \log\left(\frac{1}{f_2(z)}\right) \frac{d}{dz} \log( f_1(z^{-1}) )\,dz
                    \right)
                    \cdot \frac{
                        \prod_{z \in \mathbf{z}} f_1(z) f_2(\bar z)
                    }{
                        \prod_{y \in \mathbf{y}} f_1(y) f_2(\bar y)
                    } \\
                    &\cdot
                    \sum_{\substack{S_1,S_2 \subset \mathbf{z}\\ |S_1|=|S_2|}}
                    (-1)^{|S_1|}
                    \!\cdot\!
                    \left[
                        \prod_{S_1} \frac{ z^N f_1(z^{-1})}{f_2(z)}
                    \right]
                    \!\cdot\!
                    \left[ 
                        \prod_{S_2} \frac{ \bar z ^N f_2(\bar z^{-1})}{f_1(\bar z)}
                    \right]
                    \!\cdot\!
                    c^{\mathbf{y}}(S_1,S_2)
                    \!\cdot\!
                    c_{\mathbf{z}}(S_1,S_2), \\
                    &+e^{-\Omega( (\log N)^m)}\Exp e^{\Bias{\GF}}.
                \end{aligned}
                \end{equation}
                We remark that
                \begin{align*}                  
                    &\Exp\left[ 
                        e^{i \sum_{h,j} \xi_{h,j} \GF(\zeta_h\omega_j) + \Bias{\GF}}
                    \right] \\
                    &=\Exp e^{\Bias{\GF}} 
                    \cdot \exp\left( 
                    \frac{1}{2\pi i} \int_{\mathbb{T}} L(z) \frac{d}{dz} \bigl(\overline{L(\bar{z}^{-1})}\bigr)\,dz
                    \right)
                    \cdot \frac{
                        \prod_{z \in \mathbf{z}} e^{-i2\Re{L(\bar{z})}}
                    }{
                        \prod_{y \in \mathbf{y}}  e^{-i2\Re{L(\bar{y})}}
                    }.
                \end{align*}
		This can be seen on the one hand by noting that the mean $\mu_{h,j}$ of $\GF(\zeta_h\omega_j)$ under the bias $\Bias{\GF}$ is given by 
		\begin{equation}
		  \label{eq:muhj}
		  \mu_{h,j}
		  =
                    \sum_{z \in \mathbf{z}} 
                    -\log|1-\zeta_h \omega_j \bar z|
                    +\sum_{z \in \mathbf{z}} 
                    \log|1-\zeta_h \omega_j \bar y|,
		  \end{equation}
                and on the other hand by the identity
                \[
                    \frac{1}{2\pi i} \int_{\mathbb{T}} L(z) \frac{d}{dz} \bigl(\overline{L(\bar{z}^{-1})}\bigr)\,dz
                    =-\sum_{k =1}^\infty k\hat{L}(k) \overline{\hat L(k)},
                \]
                where $\hat L(k)$ are the Fourier coefficients of $L$ (c.f.\,\eqref{eq:Fourier}).

                It then remains to replace all the $f_1$ and $f_2$ in \eqref{eq:cf0} by $e^{-i\overline{L({\bar{\cdot}})}}$ and $e^{-iL}$ respectively.  Uniformly in $|z| \leq 1,$ using Lemma~\ref{lem:approx} and that 
		$|\Im L(z)| \ll R^{1/2}(\log N)^{(1+m)/2},$
	      we have that 
                \begin{equation}
                    \label{eq:cf1}
                    f_2(z) = e^{-iL(z)}( 1 +  e^{-\Omega( (\log N)^m)}).
                \end{equation}
                The same holds for terms $f_1(z),$ mutatis mutandis.

                However we also need to compare two other types of terms: we need to estimate the derivative $f_1'(z)$ for $|z|=1$ and $f_{1,2}(z^{-1})$ for $z \in \D.$  To estimate the derivative of $f_1,$ observe that we have \(f_1= e^{-i\overline{L(\bar{\cdot})}} + g_1\).  Hence on differentiating, we have that the error in approximation between $f_1'(z)$ and $\frac{d}{dz}e^{-i\overline{L(\bar{z})}}$ is $g_1'(z).$  This we can represent as
                \[
                    g_1'(z) = \frac{1}{2\pi i}
                    \int_\gamma \frac{g_1(x)}{(x-z)^2}\,dx
                \]
                where $\gamma$ encloses $z.$  Hence, letting $\gamma$ be a circle of radius $N^{-1}$ around $z$ and applying Lemma~\ref{lem:approx},
                we get that
                \[
                    |g_1'(z)|
                    \ll
                    N
                    \cdot\sinh({\textstyle \sum_{h,j} |\xi_{h,j}|}) \cdot
                    |\zeta_d|^{A} \cdot (1+N^{-1})^A 
                    \ll e^{-\Omega( (\log N)^m)}.
                \]
                This allows us to write
                \begin{equation}
                    \label{eq:cf2}
                    \frac{d}{dz} \log( f_1(z^{-1}) )
                    =
                    \frac{d}{dz}\bigl(-i\overline{L(\bar{z}^{-1})} \bigr)
                    +e^{-\Omega( (\log N)^m)}.
                \end{equation}

                For terms $f_2(z^{-1})$ with $z \in \D,$ we decompose into cases.  For $z$ with $\dH(0,z) \geq d + C,$ we have that
                \[
                    A \log |z^{-1} \zeta_d| \asymp 
                    e^{\log N - d - C}
                    -e^{\log N - d} \asymp - e^{\log N - d}.
                \]
                Hence, we still have that uniformly in $z$ with $\dH(0,z) \geq d + C,$
                \begin{equation}
                    \label{eq:cf3}
                    f_2(z^{-1}) = e^{-iL(z^{-1})}( 1 +  e^{-\Omega( (\log N)^m)}).
                \end{equation}

                For terms $f_2(z^{-1})$ with $\dH(0,z) \leq d + C,$ we will instead use that their contribution to $\varphi$ is negligible.  By applying the maximum principle to $f_2(z^{-1})z^A,$ we have that for $z \in \D.$
                \[
                    |f_2(z^{-1})| \leq |z|^{-A}\max_{\omega \in T} |f_2(\omega)|.
                \]
                As we have assumed that the pairwise separations between points in $\mathbf{z}$ is $N^{-1},$ we have that for any $S_1,S_2 \subset \mathbf{z},$
                \[
                    \log |c^{\mathbf{y}}(S_1,S_2)
                    \!\cdot\!
                    c_{\mathbf{z}}(S_1,S_2)|
                    \ll \log N.
                \]
                Hence if either $S_1 \subset \mathbf{z}$ or $S_2 \subset \mathbf{z}$ contain a point $z$ with $\dH(0,z) \leq d + C,$ then since $|z|^N |f_2(z^{-1})| \ll e^{-\Omega( (\log N)^m)},$
                \[
                    \biggl|
                        \prod_{S_1} \frac{ z^N f_1(z^{-1})}{f_2(z)}
                    \biggr|
                    \!\cdot\!
                    \biggl| 
                        \prod_{S_2} \frac{ \bar z ^N f_2(\bar z^{-1})}{f_1(\bar z)}
                    \biggr|
                    \!\cdot\!
                    |c^{\mathbf{y}}(S_1,S_2)
                    \!\cdot\!
                    c_{\mathbf{z}}(S_1,S_2)|
                    \ll
                    e^{-\Omega( (\log N)^m)}.
                \]
                The same statement holds for the expression after $f_{1,2}$ have been replaced, using the uniform boundedness of $e^{-iL}.$
                Using \eqref{eq:cf1}, \eqref{eq:cf2}, and \eqref{eq:cf3} to replace terms one-by-one, we arrive at the conclusion of the Lemma.
            \end{proof}

	    Lemma~\ref{lem:cf} reveals that $\varphi$ has the form of a combination of Gaussian characteristic functions, all having the same covariance.  The additional $e^{-iL}$ terms that appear in the sum can be compared to $1$ at this point, but the error incurred will be of order $(\log N)^{-K}.$  This is too expensive a loss to incur at this point.  If we approximate the characteristic function $\varphi$ with a $(\log N)^{-K}$-error, we would only have an $\operatorname{L}^{\infty}$ estimate on the difference of densities of order $(\log N)^{-K}.$  Since the space is $\Theta(\log N)$-dimensional, this is prohibitively large.
	    
	    Instead, we first do Fourier inversion, since this characteristic function can be inverted essentially exactly.  So, we let $\varrho_U$ be the 
	    density of the law of
	    \corO{$\left(\CUEF( \zeta_h \omega_j) \right.$ $\left.+ \WN(\zeta_h\omega_j)\right)_{h,j}$} under the (non-probability) measure biased by ${e^{\Bias{\CUEF}}}.$  That is to say, for any
          \[
              F \in \BSA(\left\{ \zeta_h\omega_j~:~ 1 \leq h \leq d, 1 \leq j \leq R \right\})
          \]
          with $\|F\|_\infty \leq 1,$ we have that
	  \[
              \Exp \left[F(\CUEF + \WN) e^{\Bias{\CUEF}}\right]
	      =
	      \int_{\R^{R\cdot d}} F(x)\varrho_U(x)\,dx.
	  \]
	  We also define $\varrho_G$ to be the density of the law of $\left(\GF( \zeta_h \omega_j) + \WN(\zeta_h\omega_j)\right)_{h,j}$ biased by ${e^{\Bias{\GF}}}.$ These densities exist on account of the extra smoothing provided by the $\WN(\zeta_h\omega_j)$ noise.

	  The density $\varrho_G$ is explicit.  Let $\Sigma$ denote the covariance matrix of $\left(\GF( \zeta_h \omega_j))\right)_{h,j},$ with $1 \leq h \leq d$ and $1 \leq j \leq R.$  Then
	  \[
	    \varrho_G(x)
	    =
	    \frac{
	    	e^{-\frac12 \langle (I+\Sigma)^{-1}(x-\mu), x-\mu\rangle} 
	    }{
	      \sqrt{(2\pi)^{R\cdot d} \det(I+\Sigma)}
	    },
	  \]
	  where $\mu \in \R^{R\cdot d}$ is the vector of means of $
	  \corO{\GF(\zeta_h\omega_j)}$ under the bias $e^{\Bias{\GF}}$ (see \eqref{eq:muhj}).  The bilinear form $\langle \cdot, \cdot\rangle$ is the standard real inner product on $\R^{R\cdot d}.$  We extend it to $\C^{R \cdot d}$ to stay bilinear (so there is no conjugation in the second variable).

	  \begin{lemma}
	    Fix $m >1$ a real number and fix $R \in \mathbb{N}.$
	    Uniformly in $d \leq \log N - m\log\log N,$ in $x \in \R^{R\cdot d}$ uniformly and uniformly in $\mathbf{z} \subset \D$ of a given finite cardinality having pairwise hyperbolic separation
	      $N^{-1},$
                \begin{align*}
                    \varrho_U(x)
                    =&\varrho_G(x)
                    \cdot
                    \!\!\!\!
		    \sum_{\substack{S_1,S_2 \subset \mathbf{z}\\ |S_1|=|S_2|}}
                    \!(-1)^{|S_1|}
                    \!\cdot\!
		    e^{H(S_1,S_2,x)}
                    \!\cdot\!
                    \left[
                        \prod_{S_1} z^N 
                    \right]
                    \!\cdot\!
                    \left[ 
                        \prod_{S_2} 
			  \bar z^N 
                    \right]
                    \!\cdot\!
                    c^{\mathbf{y}}(S_1,S_2)
                    \!\cdot\!
                    c_{\mathbf{z}}(S_1,S_2),
                    \\
                    &+e^{-\Omega((\log N)^m)}\cdot\Exp\left[ e^{\Bias{\GF}} \right],
                \end{align*}
		where
		\[
		  H(S_1,S_2,x)
		  = 
		  \langle (I+\Sigma)^{-1}(\mu-x-\frac{1}{2}v), v\rangle
		\]
		and $v \in \C^{R\cdot d}$ is the vector
		\[
		  v_{h,j} = 
		  \sum_{z \in S_1} 
		  \log\left( 
		  \frac{1-\zeta_h\corO{\bar{\omega}_j} {z}^{-1}}
		  {1-\zeta_h{\omega_j} {z} } \right)
		  +
		  \sum_{z \in S_2} 
		  \log\left( 
		  \frac{1-\zeta_h{\omega_j} \bar{z}^{-1}}
		  {1-\zeta_h\corO{\bar{\omega}_j} \bar{z} } \right).
		\]
	    \label{lem:density_expansion}
	  \end{lemma}
	  \begin{proof}
	    By Fourier inversion, we can write for any $x \in \R^{R\cdot d},$
	    \[
                    \varrho_U(x)
		    =
		    \frac{1}{(2\pi)^{R\cdot d}}
		    \int_{ \R^{R\cdot d}} 
		    e^{-i \langle \xi, x\rangle}
		    \varphi(\xi)
		    e^{-\frac12 \langle \xi, \xi\rangle}
		    \,d\xi.
	    \]
	    Note that by definition of $\varphi,$ we have that
	    \[
		    |\varphi(\xi)| \leq \Exp\left[ e^{\Bias{\CUEF}} \right].
	    \]
	    Hence, truncating the integral according to $\|\xi\|_2^2 \leq (\log N)^m$ or otherwise, we get that
	    \[
                    \varrho_U(x)
		    =
		    \frac{1}{(2\pi)^{R\cdot d}}
		    \int\limits_{ \|\xi\|_2^2 \leq (\log N)^m} 
		    e^{-i \langle \xi, x\rangle}
		    \varphi(\xi)
		    e^{-\frac12 \langle \xi, \xi\rangle}
		    \,d\xi
                    +e^{-\Omega((\log N)^m)}\cdot\Exp\left[ e^{\Bias{\CUEF}} \right].
	    \]
	    The Lebesgue-volume of $\left\{ \xi : \|\xi\|_2^2 \leq (\log N)^m \right\}$ is
	    $e^{ O(\log N \log\log N)},$ and hence by Lemma~\ref{lem:cf}, we get that
	    \begin{equation}
	      \label{eq:cf4}
	    \begin{aligned}
                    \varrho_U(x)
                    =&\varrho_G(x)
                    \cdot
                  \!\!\!  \sum_{\substack{S_1,S_2 \subset \mathbf{z}\\ |S_1|=|S_2|}}
                    \!(-1)^{|S_1|}
                    \!\cdot\!
		    e^{\tilde{H}(S_1,S_2,x)}
                    \!\cdot\!
                    \left[
                        \prod_{S_1} z^N 
                    \right]
                    \!\cdot\!
                    \left[ 
                        \prod_{S_2} 
			  \bar z^N 
                    \right]
                    \!\cdot\!
                    c^{\mathbf{y}}(S_1,S_2)
                    \!\cdot\!
                    c_{\mathbf{z}}(S_1,S_2),
                    \\
                    &+e^{-\Omega((\log N)^m)}\cdot\Exp\left[ e^{\Bias{\CUEF}} \right],
                \end{aligned}
	      \end{equation}
		where
		\begin{multline*}
		  \varrho_G(x)
		  e^{\tilde{H}(S_1,S_2,x)}
		  =
		  \frac{1}{(2\pi)^{R\cdot d}}
		 \!\! \int\limits_{ \|\xi\|_2^2 \leq (\log N)^m} 
		  e^{-i \langle \xi, x\rangle}
		  e^{-\frac12 \langle \xi, \xi\rangle}
		  \Exp\left[ 
                        e^{i \sum_{h,j} \xi_{h,j} \GF(\zeta_h\omega_j) + \Bias{\GF}}
                  \right] \\
		  \cdot
                    \left[
                        \prod_{S_1} \frac{ e^{-i\overline{L(\bar{z}^{-1})}}}{e^{-i{L(z)}}}
                    \right]
                    \!\cdot\!
                    \left[ 
                        \prod_{S_2} 
                        \frac{ e^{-iL(\bar z^{-1})}}
                        {e^{-i\overline{L({z})}}}
                    \right]
		  \,d\xi.
		\end{multline*}
		Using that 
		\[
		  |\Im L(z)| \ll \sum_{h,j} |\xi_{h,j}| \ll (R\log N)^{1/2}\|\xi\|_2,
		\]
		we get that
		\begin{multline*}
		  \int\limits_{ \|\xi\|_2^2 > (\log N)^m} 
		  e^{-\frac12 \langle \xi, \xi\rangle}
		  \Exp\left[ 
                        e^{\Bias{\GF}}
                  \right] 
		  \cdot
                    \left|
                        \prod_{S_1} \frac{ e^{-i\overline{L(\bar{z}^{-1})}}}{e^{-i{L(z)}}}
			\cdot
                        \prod_{S_2} 
                        \frac{ e^{-iL(\bar z^{-1})}}
                        {e^{-i\overline{L({z})}}}
                    \right|
		    \,d\xi \\
		    \leq e^{-\Omega( (\log N)^m)}
		  \Exp\left[ 
                        e^{\Bias{\GF}}
                  \right] 
		    .
		  \end{multline*}
		  Thus if we define $H$ by 
		\begin{multline*}
		  \varrho_G(x)
		  e^{{H}(S_1,S_2,x)}
		  =
		  \frac{1}{(2\pi)^{R\cdot d}}\\
		  \cdot\int\limits_{\R^{R\cdot d}} 
		  e^{-i \langle \xi, x\rangle}
		  e^{-\frac12 \langle \xi, \xi\rangle}
		  \Exp\left[ 
                        e^{i \sum_{h,j} \xi_{h,j} \GF(\zeta_h\omega_j) + \Bias{\GF}}
                  \right] 
		  \cdot
                    \left[
                        \prod_{S_1} \frac{ e^{-i\overline{L(\bar{z}^{-1})}}}{e^{-i{L(z)}}}
                    \right]
                    \!\cdot\!
                    \left[ 
                        \prod_{S_2} 
                        \frac{ e^{-iL(\bar z^{-1})}}
                        {e^{-i\overline{L({z})}}}
                    \right]
		  \,d\xi,
		\end{multline*}
		we can replace $\tilde H$ in \eqref{eq:cf4} by $H$ with the same order of error term. 

		There only remains to show that $H$ can be written as claimed in the Lemma.  Write $v \in \C^{R\cdot d}$ for the vector
		\[
		  v_{h,j} = 
		  \sum_{z \in S_1} 
		  \log\left( 
		  \frac{1-\zeta_h\bar{\omega_j} {z}^{-1}}
		  {1-\zeta_h{\omega_j} {z} } \right)
		  +
		  \sum_{z \in S_2} 
		  \log\left( 
		  \frac{1-\zeta_h{\omega_j} \bar{z}^{-1}}
		  {1-\zeta_h\bar{\omega_j} \bar{z} } \right).
		\]
		In this way, we have that
		\[
		  \left[
                        \prod_{S_1} \frac{ e^{-i\overline{L(\bar{z}^{-1})}}}{e^{-i{L(z)}}}
                    \right]
                    \!\cdot\!
                    \left[ 
                        \prod_{S_2} 
                        \frac{ e^{-iL(\bar z^{-1})}}
                        {e^{-i\overline{L({z})}}}
                    \right]
		    =e^{-i \langle \xi, v \rangle}.
		\]
		This allows us to write
		\[
		  \varrho_G(x)
		  e^{{H}(S_1,S_2,x)}
		  =
		  \frac{1}{(2\pi)^{R\cdot d}}
		  \int\limits_{\R^{R\cdot d}} 
		  e^{i\langle \xi, \mu-x-v\rangle}
		  e^{-\frac12 \langle (I+\Sigma)\xi, \xi\rangle}
		  \Exp\left[ 
                        e^{\Bias{\GF}}
                  \right]
		  \,d\xi.
		\]
		Changing variables by letting $\xi = (1+\Sigma)^{-1}(i(\mu-x-v)+\xi')$
		and using analyticity of the integrand to keep the integration over $\xi
		' \in \R^{R\cdot d},$ this integral reduces to the standard Gaussian integral:
		\[
		  \varrho_G(x)
		  e^{{H}(S_1,S_2,x)}
		  =
		  \frac{
		    e^{-\frac12 \langle (I+\Sigma)^{-1}(x-\mu), x-\mu\rangle} 
		  }{
		    \sqrt{(2\pi)^{R\cdot d} \det(I+\Sigma)}
		  }
		  \Exp\left[ 
                        e^{\Bias{\GF}}
                  \right]
		  e^{\langle (I+\Sigma)^{-1}(\mu-x-\frac{1}{2}v), v\rangle}
		  .
		\]
	  \end{proof}

	  To conclude the proof of the Proposition~\ref{prop:clt}, we would like to show that $H$ can be ignored.  Since Proposition~\ref{prop:clt} follows immediately from Lemma~\ref{lem:density_expansion} and Lemma~\ref{lem:H}, we say nothing more on its proof.
	  \begin{lemma}
                For any $\delta > 0,k>0$, any $m > 2+2\delta$ and any $R \in \mathbb{N},$ the following holds uniformly in all possible choices.
		Let $\mathbf{z} \subset \D$ of cardinality $k$ have the property that all points in $\mathbf{z}$ with distance to $0$ greater than $\log N - (1+\delta)\log\log N$ are contained in a Euclidean ball of radius $N^{-1}(\log N)^m.$  Suppose it has the further property that the pairwise hyperbolic separation between its elements is 
		  $N^{-1}.$
		For all $S_1,S_2 \subseteq \mathbf{z}$ with $|S_1|=|S_2|,$
                \begin{align*}
		  \int_{\R^{R\cdot d}}
		  |e^{{H}(S_1,S_2,x)} - 1|
		  \,
		  \varrho_G(x) 
		  dx
		  &\!\cdot\!
                    \left[
                        \prod_{S_1} { z^N } 
		      \right]
                    \!\cdot\!
                    \left[ 
                        \prod_{S_2} 
                        { \bar z^N }
                    \right]
                    \!\cdot\!
                    c^{\mathbf{y}}(S_1,S_2)
                    \!\cdot\!
                    c_{\mathbf{z}}(S_1,S_2) \\
		    &\ll(\log N)^{\corO{(2+\delta-m)}/2}(1+\Delta)^{2|\mathbf{z}|}
		    \Exp\left[ e^{\Bias{\GF}} \right]
		    .
		\end{align*}
		$\Delta$ is as in Corollary~\ref{cor:exact} or Proposition~\ref{prop:clt}.
                \label{lem:H}
            \end{lemma}
            \begin{proof}
	      If some $z \in S_1 \cup S_2$ has $\dH(0,z) \leq \log N - (1+\delta)\log \log N,$ the whole left-hand-side will be trivial.  This is because $|z|^{N} \leq e^{-\Omega( (\log N )^{1+\delta})}.$  By the separation assumption on $\mathbf{z},$ we get that
	      \[
		\log |c^{\mathbf{y}}(S_1,S_2)
		\!\cdot\!
		c_{\mathbf{z}}(S_1,S_2)|
		\ll \log N.
	      \]
	      We further have that $\int \varrho_G(x)\,dx = \Exp\left[ e^{\Bias{\GF}} \right].$  So we only need to worry about controlling $\int e^{\Re H(S_1,S_2,x)}\,\varrho_G(x)dx.$  To do this, we begin by observing
	      \begin{equation}
		\varrho_G(x)
		e^{\Re{H}(S_1,S_2,x)}
		=
		\frac{
		  e^{-\frac12 \langle (I+\Sigma)^{-1}(x-\mu+\Re v), x-\mu+\Re v\rangle} 
		}{
		  \sqrt{(2\pi)^{R\cdot d} \det(I+\Sigma)}
		}
		\Exp\left[ 
		  e^{\Bias{\GF}}
		\right]
		e^{
		  \frac12 \langle (I+\Sigma)^{-1}\Im v, \Im v\rangle
		}.
		\label{eq:cf5}
	      \end{equation}
	      Hence
	      \[
		\int e^{\Re H(S_1,S_2,x)}\,\varrho_G(x)dx
		=e^{\frac12 \langle (I+\Sigma)^{-1}\Im v, \Im v\rangle}.
	      \]
	      The entries of $\Im v$ are all bounded in absolute value by something that can be chosen only to depend on $|S_1 \cup S_2|.$  The operator norm of $(I+\Sigma)^{-1}$ is at most $1$ as $\Sigma$ is positive definite.  Hence
	      $\langle (I+\Sigma)^{-1}\Im v, \Im v\rangle \ll \log N.$
	      Thus, putting all of this together, we get that when 
some $z \in S_1 \cup S_2$ has $\dH(0,z) \leq \log N - (1+\delta)\log \log N,$ the Lemma holds.

	      So, in what follows, we will assume that all $z \in S_1 \cup S_2$ have $\dH(0,z) > \log N - (1+\delta)\log\log N.$  Our first task will be to truncate the integral.  We will restrict the integral to the $\operatorname{L}^{\infty}$-ball $V$ in $\R^{R\cdot d}$ of radius $(\log N)^{1+\delta}.$  The probability that any $\GF(\zeta_h\omega_j) - \mu_{h,j}$ is larger than $(\log N)^{1+\delta}$ in absolute value is $e^{-\Omega( (\log N)^{1+2\delta})},$ hence
	      \[
		\int_{V^c} \varrho_G(x)dx \leq e^{-\Omega( (\log N)^{1+2\delta})}
		\Exp\left[ e^{\Bias{\GF}} \right].
	      \]
	      We would like to say the same for $e^{\Re H(S_1,S_2,x)}\varrho_G(x),$ for which we will use \eqref{eq:cf5}. 
	      Hence we need to know something about the magnitudes of the real parts of $v.$
	      Suppose $y \in \D$ and $z \in \C$ have $|z-y| \ll (\log N)^{m} N^{-1}$ and $\zeta_d z \in \D,$ then uniformly in $h \in \mathbb{N}$ with $1 \leq h \leq d,$ 
	      \begin{equation}
		\left|
		\log\left( 
		\frac{1-\zeta_h z}
		{1-\zeta_h y}
		\right)
		\right|
		\ll
		\frac{|z-y|}
		{1-\zeta_h}.
		\label{eq:cf6}
	      \end{equation}
	      Thus the real and imaginary parts of the entries of $v$ can be bounded solely in terms of $|S_1 \cup S_2|.$  Hence, the probability that any $\GF(\zeta_h\omega_j) - \mu_{h,j} - v_{h,j}$ is larger than $(\log N)^{1+\delta}$ in absolute value is also $e^{-\Omega( (\log N)^{1+2\delta})},$ so that
	      \[
		\int_{V^c} 
		e^{\Re H(S_1,S_2,x)}\,
		\varrho_G(x)dx \leq e^{-\Omega( (\log N)^{1+2\delta})}
		\Exp\left[ e^{\Bias{\GF}} \right].
	      \]
	      Therefore, it suffices to show that
	      \begin{align*}
		\int_{V}
		|e^{{H}(S_1,S_2,x)} - 1|
		\,
		\varrho_G(x) 
		dx
		&\!\cdot\!
		\left[
		  \prod_{S_1} { z^N } 
		\right]
		\!\cdot\!
		\left[ 
		  \prod_{S_2} 
		  { \bar z^N }
		\right]
		\!\cdot\!
		c^{\mathbf{y}}(S_1,S_2)
		\!\cdot\!
		c_{\mathbf{z}}(S_1,S_2) \\
		&\ll(\log N)^{
		  \corO{(3+2\delta-m)}/2}(1+\Delta)^{2|\mathbf{z}|}
		\Exp\left[ e^{\Bias{\GF}} \right]
		.
	      \end{align*}

	      Since $\|(I+\Sigma)^{-1}\|_{\operatorname{op}} \leq 1,$
	      for $x \in V,$
	      we can estimate
	      \[
		|H(S_1,S_2,x)|
		\leq \|\mu-x-\frac{1}{2}v\|_2\|v\|_2
		\ll (\log N)^{\tfrac 32 +\delta} \|v\|_1.
	      \]
	      Hence, when $\|v\|_1 \ll (\log N)^{-m/2},$ we have shown the desired bound, and it suffices to consider the case that $\|v\|_1 \gg (\log N)^{-m/2}.$

	      Applying \eqref{eq:cf6}, for any bijection $\phi : S_1 \setminus S_2 \to S_2 \setminus S_1,$
	      \[
		\|v\|_1
		\ll 
		\sum_{z \in S_1 \cap S_2} \frac{|z^{-1} - \bar z| + |z - \bar{z}^{-1}|}{1-\zeta_d}
		+\sum_{z \in S_1 \setminus S_2} \frac{
		  |z^{-1} - \overline{\phi(z)}|
		  +|z-{\overline{\phi(z)}}^{-1}|
		}{1-\zeta_d}.
	      \]
	      Since all $z\in S_1 \cup S_2$ have $1-|z| \gg N^{-1}(\log N)^{1+\delta},$ we can replace all inverses in this sum by conjugates, incurring an acceptable error.  Namely:
	      \begin{equation*}
		\|v\|_1
		\ll 
		(\log N)^{1+\delta - m}
		+
		\sum_{z \in S_1 \setminus S_2} \frac{
		  |z - {\phi(z)}|
		}{1-\zeta_d}.
	      \end{equation*}
	      Hence it follows that in the case we are considering, for every matching $\phi : S_1 \setminus S_2 \to S_2 \setminus S_1,$
	      \begin{equation}
		\label{eq:cf7}
		(\log N)^{m/2}
		\ll
		\sum_{z \in S_1 \setminus S_2}
		N\cdot|z - {\phi(z)}|.
	      \end{equation}
	      Estimating the integral by the trivial bound 
	      $
		|e^{{H}(S_1,S_2,x)} - 1|
		\leq
		e^{\Re{H}(S_1,S_2,x)} + 1
		$
		and using \eqref{eq:cf5},
	      \begin{align*}
		\int_{V}
		|e^{{H}(S_1,S_2,x)} - 1|
		\,
		\varrho_G(x) 
		dx
		&\!\cdot\!
		\left[
		  \prod_{S_1} { z^N } 
		\right]
		\!\cdot\!
		\left[ 
		  \prod_{S_2} 
		  { \bar z^N }
		\right]
		\!\cdot\!
		c^{\mathbf{y}}(S_1,S_2)
		\!\cdot\!
		c_{\mathbf{z}}(S_1,S_2) \\
		&\ll (1+\Delta)^{2|\mathbf{z}|}
		\frac{e^{\frac12 \langle (I+\Sigma)^{-1}\Im v, \Im v\rangle}
		}
	      {
		\Exp 
		\exp\left( 
		2 \sum_{S_1 \cap S_2^c} \GF[z]
		-2 \sum_{S_2 \cap S_1^c} \GF[z]
		\right)
	      }.
	    \end{align*}
	    Since $\|v\|_1 \ll 1$ and $\|(I+\Sigma)^{-1}\|_{\operatorname{op}} \leq 1,$ it follows that $e^{\frac12 \langle (I+\Sigma)^{-1}\Im v, \Im v\rangle} \ll 1.$
            We now argue that this Gaussian expectation in the denominator is large.  Specifically, by comparison with branching random walk, we will show the following:
            \begin{lemma}
            \[
                        2
                        \Var
                        \left(
                        \sum_{S_1 \cap S_2^c} \GF[z]
                        -\sum_{S_2 \cap S_1^c} \GF[z]
                        \right)
                        \geq 
                        \min_{\phi}
                        \sum_{z \in S_1 \setminus S_2}\dH(z,\phi(z))-O(1),
            \]
	    where the minimum is over the set of all bijections from $S_1 \setminus S_2 \to S_2 \setminus S_1.$  
            The error can be taken to only depend on $|S_1 \triangle S_2|.$
            \label{lem:varbound}
            \end{lemma}
	    We delay the proof of Lemma \ref{lem:varbound} for the moment.
	    Since all $z$ in question have $\dH(0,z) \leq \log N - (1+\delta)\log \log N,$ we therefore can estimate using Lemma~\ref{lem:branch} for any such $z,$
            \[
                \dH(z,\phi(z))
                = 
                (\log |z-\phi(z)|
                +\log N-(1+\delta)\log N)_+-O(1).
            \]
	    By \eqref{eq:cf7}, at least one of these $z$ has Euclidean distance to $\phi(z)$ order $\Omega( (\log N)^{m/2}),$ and we conclude
            \[
                \Exp 
                        \exp\left( 
                        2 \sum_{S_1 \cap S_2^c} \GF[z]
                        -2 \sum_{S_2 \cap S_1^c} \GF[z]
                        \right)
                        \gg 
			(\log N)^{m/2-1-\delta}.
            \]
	    This completes the proof of the lemma.
            \end{proof}
            We now give the proof of Lemma~\ref{lem:varbound}
            \begin{proof}[Proof of Lemma~\ref{lem:varbound}]
                We prove this by comparison with branching random walk.  First we construct a tree.  Define a family of rays $\left\{ \zeta_h \omega_z : z \in S_1 \triangle S_2, h \in \mathbb{N}_0  \right\},$ where $\omega_z$ is the polar part of $z.$  Let $\Tilde T$ be a graph on this vertex set with edges between every pair of the form $(\zeta_h\omega_z,\zeta_{h+1}\omega_z),$ where $h \in \mathbb{N}_0.$ Now, identify any two points $\zeta_h \omega_z$ and $\zeta_h \omega_y$ with $e^h |\arg\omega_y-\arg\omega_z| < 1.$  This will not be an equivalence relation so extend this to the transitive closure. Let $T$ be the graph that results, discarding multiple edges. 

		We claim $T$ is a tree, which we show by induction on the number of rays.  With a single ray, there is nothing to check.  Suppose that we wish to add a ray $\left\{ \zeta_h \omega_z \right\}_{h=1}^\infty$ to an existing tree $T.$  Suppose that there is some $h \in \N$ and some $y \in S_1 \triangle S_2$ so that $e^h |\arg\omega_y-\arg\omega_z| < 1,$ and so we should identify the two points $\zeta_h\omega_z$ with the entire equivalence class of $\zeta_h\omega_y$ in $T.$  Then for every $j \leq h,$ we also have that $e^j |\arg\omega_y-\arg\omega_z| < 1,$ and hence the entire vector $\left( \zeta_h\omega_z \right)_{j=1}^h$ will be identified with $\left( \zeta_h\omega_y \right)_{j=1}^h.$  This ensures that after the $\omega_z$ ray is added to $T,$ the graph remains a tree.
		
	      Orient $T$ by calling $0$ the top.  On $T,$ we define branching random walk.  That is, we put independent, standard normals on every edge, and let $W(v)$ be the sum of these normals on along the unique path connecting $0$ to $v.$  Then by Lemma~\ref{lem:branch}, up to constants depending only $|S_1 \triangle S_2|,$ we have that
                \[
                    \Exp\left[ W(\zeta_h\omega_z)W(\zeta_j\omega_y) \right]
                    =2\Exp\left[ \GF(\zeta_h\omega_z)\GF(\zeta_j\omega_y) \right] + O(1).
                \]
                Letting $\tilde S_1$ and $\tilde S_2$ be the points in $T$ which are closest to $S_1$ and $S_2,$ we get that
            \[
                        2\Var
                        \left(
                        \sum_{S_1 \cap S_2^c} \GF[z]
                        -\sum_{S_2 \cap S_1^c} \GF[z]
                        \right)
                        =
                        \Var
                        \left(
                        \sum_{\tilde S_1 \cap \tilde S_2^c} W(z)                        -\sum_{\tilde S_2 \cap \tilde S_1^c} W(z)
                        \right)
                        +O(1),
            \]
            using the bounded correlation of increments \eqref{eq:correlationbound}.  The error can again be estimated solely in terms of $|S_1 \triangle S_2|.$

	    Hence, we have reduced the problem to estimating this variance on the tree.  For a matching $\phi : \tilde S_1 \setminus \tilde S_2 \to \tilde S_2 \setminus \tilde S_1,$ each pair $W(z) - W(\phi(z))$ is a sum of 
    independent, standard normals along the geodesic in $T$ that connects $z$ to $\phi(z),$ some having positive signs and some having negative signs.  After summing this over all $z,$ a lower bound for the variance is simply the sum of squares of the signed numbers of times edges are crossed by geodesic paths $(z,\phi(z))$ with $z \in \tilde S_1 \setminus \tilde S_2.$  

            We claim there is a matching $\phi : \tilde S_1 \setminus \tilde S_2 \to \tilde S_2 \setminus \tilde S_1,$ for which there is no cancellation, i.e.\,if $n$ is the number of edges of $T$ contained in the geodesic from $z \to \phi(z)$ and in the geodesic from $y \to \phi(y),$ then
            \[
                \Exp
                \left( W(z) - W(\phi(z)) \right)
                \left( W(y) - W(\phi(y)) \right)
                = n.
            \]
            Note that for such a matching, we therefore have that
            \[
                        \Var
                        \left(
                        \sum_{\tilde S_1 \cap \tilde S_2^c} W(z)                        -\sum_{\tilde S_2 \cap \tilde S_1^c} W(z)
                        \right)
                \geq
                \sum_{z \in \tilde S_1 \cap \tilde S_2^c}
                \dT(z,\phi(z)).
            \]
            Since we have that $\dT(z,\phi(z)) = \dH(z,\phi(z)) + O(1),$ this completes the proof.

            Let $\psi : T^2 \to \mathbb{Z}$ be the coefficient of an edge of $T$ in the sum 
            \[
                \sum_{\tilde S_1 \cap \tilde S_2^c} W(z)       
                -\sum_{\tilde S_2 \cap \tilde S_1^c} W(z).
            \]
            For a point in $z\in\tilde S_1 \triangle \tilde S_2,$ there must be an edge incident to it on which $\psi$ is nonzero, as $W(z)$ includes in it all edges above $z$ while other $W(y)$ would include both the edge above $z$ and an edge below (or neither).  On the other hand, by the same reasoning, at a point $z$ which is incident to a single edge for which $\psi$ is nonzero, we must have $z \in \tilde S_1 \triangle \tilde S_2.$  

            Pick a path $\gamma = z_0z_1\dots z_k$ which is of maximal length subject to the following conditions 
            \begin{enumerate}
                \item $\psi(z_i,z_{i+1}) \neq 0$ for all $0 \leq i \leq k-1.$
                \item $\psi(z_i,z_{i+1})$ changes signs at most once.  If it does, it changes from positive to negative.  
                \item $\psi(z_i,z_{i+1}) > 0$ if and only if $z_{i+1}$ is above $z_i.$
            \end{enumerate}
            Such paths always exist provided $\psi$ is not identically $0.$  This is guaranteed provided $\tilde S_1 \triangle \tilde S_2 \neq \emptyset.$
            We claim that $z_0\in\tilde S_1 \triangle \tilde S_2$ and $z_k\in\tilde S_2 \triangle \tilde S_1.$
            Suppose we are in the case that $\psi(z_0,z_{1}) > 0.$  Then $z_1$ is above $z_0.$  As the path is maximal, it is not possible to extend it to include any points below $z_0.$  Hence it must be that for all $y$ immediately below $z_0,$ $\psi(y,z_0) \leq 0.$  This is only possible if $z_0 \in \tilde S_1 \triangle \tilde S_2.$  
            
           Consider now the case $\psi(z_0,z_{1}) < 0.$  Then we have $z_1$ below $z_0$ by definition.  Due to maximality of the path, there are two types of extensions to the path that must be impossible.  First, it must be there is no $y$ above $z_0$ so that $\psi(y,z_0) < 0.$  Thus there are at least as many $z \in \tilde S_1 \triangle \tilde S_2$ at or below $z_0$ as there are $z \in \tilde S_2 \triangle \tilde S_1$ at or below $z_0.$  Second, for every other $y$ below $z_0,$ we have that $\psi(y,z_0) \leq 0.$  Hence, we conclude that $z_0$ must be in $\tilde S_1 \triangle \tilde S_2.$ 

           A symmetric argument shows that $z_k\in\tilde S_2 \triangle \tilde S_1.$  We let $\phi(z_0) = z_k.$  The geodesic between $z_0$ and $z_k$ is exactly $\gamma,$ as the distance to $0$ decreases monotonically and then increases monotonically.  Hence we remove $z_0$ from $\tilde S_1,$ remove $z_k$ from $\tilde S_2$ and repeat this procedure.   The matching constructed this way has no cancellation as the effect on $\psi$ of deleting $z_0$ and $z_k$ is to increase some edges where $\psi < 0$ by $1$ and decrease edges where $\psi > 0$ by $1.$
      \end{proof}

\appendix
\section{Barrier estimate for certain Gaussian processes}
\label{app-barrier1}
Throughout  this appendix, $\{G_i\}_{i=1}^n$ denotes a centered Gaussian
process with covariance $E(G_iG_j)=:R_G(i,j)$, 
while $\{Y_i\}_{i=1}^n$ denotes the Gaussian
centered process with covariance $R_Y(i,j)=i\wedge j$.
\subsection{The ballot theorem}
Throughout the paper, we 
use the following version of the ballot theorem, 
restricted to Gaussian random variables.
\begin{theorem}
  \label{theo-app-ballot}
  For $\sqrt{n}\geq x,y\geq 1$, we have
  \begin{equation}
    \label{eq-ballot-app}
    \Pr( Y_i\geq -x, i=1,\ldots,n,
    Y_n\in [-x+y,-x+y+1]) \asymp \frac{xy}{n^{3/2}}.
  \end{equation}
  Further, the upper bound in \eqref{eq-ballot-app} holds without restriction on
  the upper bound on $x,y$.
\end{theorem}
For the proof of a much more general version that does not use the
Gaussian assumption, we refer to \cite{Carr,AB}.
\subsection{One ray estimate}
Given a process
$\{Z_i\}_{i=1}^n$, a sequence 
$h=(h_i)$ and a real number $t$, we introduce the barrier event 
$$B_Z(n,t,h)= \{Z_i\leq h_i, i=1,\ldots,n-1; Z_n\in [t,t+1]\}$$
and its probability $p_{B;Z}(n,t,h)=P(B_Z(n,t,h))$.
\begin{proposition}
  \label{prop-comp-bar}
Assume that there exists a constant $C$, independent of $n$, so that
$|R_G(i,j)-R_Y(i,j)|\leq C$ for all $i,j\leq n$. 
Then, for any $\epsilon>0$ there exists $C_1=C_1(C,\epsilon)$,
$j=1,\ldots, $ so that, for any
$t$ with $|t|<n^{1/4}$ and all $n$ large enough,
\begin{eqnarray}
	\label{eq-barrier}
	&&(1-\epsilon) p_{B,Y}(n,t,h-(\log n)^{3/4})-C_1e^{-(\log n)^{3/2-\epsilon}}
	 \nonumber\\ 
&&	
\;\;\;\;\;\;\;
\;\;\;\;\;\;\;
\;\;\;\;\;\;\;
\leq
p_{B,G}(n,t,h)\leq \\
&&	
\;\;\;\;\;\;\;
\;\;\;\;\;\;\;
\;\;\;\;\;\;\;
\;\;\;\;\;\;\;
(1+\epsilon) p_{B,Y}(n,t,h+(\log n)^{3/4})+C_1e^{-(\log n)^{3/2-\epsilon}} .
	\nonumber
\end{eqnarray}
\end{proposition}
\begin{proof}
	Throughout this proof, $c$ denotes a constant that depends 
	only on $C$.

	It will be convenient to consider the processes conditioned
	on their end point. Recall that for a centered Gaussian process
	$\{Z_i\}_{i=1}^n$, setting $\hat Z_i=E(Z_i|Z_n)$ and
	$\tilde Z_i=Z_i-\hat Z_i$, we have
	\begin{equation}
		\label{eq-cond1}
		\hat Z_i=\frac{R_Z(i,n)}{R_Z(n,n)} Z_n,\quad
		R_{\tilde Z}(i,j)=R_Z(i,j)-\frac{R_Z(i,n)R_Z(j,n)}{R_Z(n,n)}\,.
	\end{equation}
	In particular, we have
	\begin{equation}
		\label{eq-cond2}
		\hat Y_i=
		\frac{i}{n} Y_n,\quad
		|\hat G_i-\frac{i}{n} G_n|
		\leq c\frac{|G_n|}{n},\quad
		|R_{\tilde G}(i,j)-R_{\tilde Y}(i,j)|\leq c\,.
	\end{equation}
	Note also that $\{\tilde Y_i\}$ is a Gaussian bridge. 

Let 
$$\phi(x,\sigma^2)=\frac{1}{\sqrt{2\pi\sigma^2}} e^{-x^2/2\sigma^2}$$
denote the centered \corO{standard} Gaussian density.
We rewrite $p_{B,G}(n,t,h)$ in terms of the conditional process $\tilde G$ as
follows:
\begin{equation}
	\label{eq-cond3}
	p_{B,G}(n,t,h)=\int_{[t,t+1]} \phi(x,R_G(n,n))
	P(\tilde G_i\leq h_i-\frac{R_G(i,n)}{R_G(n,n)} x,
	i=1,\ldots,n-1)dx\,.
\end{equation}
In the range $x\in [t,t+1]$ and with the constraint on $|t|<\!<n^{1/2}$
we have that
\begin{equation}
	\label{eq-cond4}
	  (1-\frac{c}{n}) \phi(x,n)
	\leq \phi(x,R_G(n,n)) \leq (1+\frac{c}{n}) \phi(x,n).
\end{equation}
\corO{(It is useful to recall that $\phi(x,n)=\phi(x,R_Y(n,n))$.)}

Our proof of the proposition will use in a crucial way Slepian's lemma.
Toward this end, let $W$ be a centered Gaussian random variable of variance
$1$ and let $W_i$, $i=1,\ldots,n-1$ denote i.i.d.\ copies of $W$,
\corO{with $W$ and the $W_i$'s independent of $G$ and $Y$.}
Introduce the processes
$$G^u_i=\corO{\tilde G_i}+\sigma_i^u W, Y^u_i=\corO{\tilde Y_i}+\eta_i^u W_i,
G^l_i=\corO{\tilde G_i}+\sigma_i^l W_i, Y^l_i=\corO{\tilde Y_i}+
\eta_i^l W\,,$$
where the superscripts $u$ and $l$ refer to upper and lower bounds in
\eqref{eq-barrier}, and $\{\sigma^u_i\}, $ $\{\sigma^l_i\},$ 
$ \{\eta^u_i\},$ $
\{\eta_i^l\}$ are deterministic sequences, chosen as follows.
\begin{enumerate}
	\item\textbf{Upper Bound:} We choose
		\begin{eqnarray}
			\label{eq-constub1}
		R_{\tilde G}(i,i)+(\sigma_i^u)^2&=&
		R_{\tilde Y}(i,i)+(\eta_i^u)^2, \quad i=1,\ldots,n-1\\
			\label{eq-constub2}
		R_{\tilde G}(i,j)+\sigma_i^u \sigma_j^u&\geq &
		R_{\tilde Y}(i,j), \quad 1\leq i<j\leq n-1.
	\end{eqnarray}
	(Such a choice is possible by first choosing $\sigma_i^u>
	(R_{\tilde G}(i,i)-R_{\tilde Y}(i,i))_+$ so 
	that \eqref{eq-constub2} is satisfied, 
	and then adjusting $\eta_i^u\geq 0$ to satisfy 
	\eqref{eq-constub1}.) 
	\item\textbf{Lower Bound:} We choose
		\begin{eqnarray}
			\label{eq-constlb1}
		R_{\tilde G}(i,i)+(\sigma_i^l)^2&=&
		R_{\tilde Y}(i,i)+(\eta_i^l)^2, \quad i=1,\ldots,n-1\\
			\label{eq-constlb2}
			\corO{R_{\tilde G}(i,j)}&\leq &
		R_{\tilde Y}(i,j)+\eta_i^l \eta_j^l, \quad 1\leq i<j\leq n-1.
	\end{eqnarray}
	(Such a choice is possible by first choosing $\eta_i^l>
	(R_{\tilde Y}(i,i)-R_{\tilde G}(i,i))_+$ so 
	that \eqref{eq-constub2} is satisfied, 
	and then adjusting $\sigma_i^l\geq 0$ to satisfy 
	\eqref{eq-constlb1}.)
\end{enumerate}
	With these choices, we note that
	\begin{equation}
		\label{eq-comp1}
		\max_i (\sigma_i^u+\sigma_i^l+\eta_i^u+\eta_i^l)<c.
	\end{equation}

	We now prove the upper bound in \eqref{eq-barrier}.
	By Slepian's inequality, we obtain that for any deterministic sequence 
	$\{g_i\}$,
	$$
	P(G_i^u\leq g_i,
	i=1,\ldots,n-1)
	\leq
	P(Y_i^u\leq g_i,
	i=1,\ldots,n-1).$$
	In particular, 
	\begin{eqnarray*}
		&&P(\cap_{i=1}^{n-1}
		\{\tilde G_i\leq h_i-\frac{R_G(i,n)}{R_G(n,n)} x\}
	)\\
	&=&
	P(\cap_{i=1}^{n-1}
	\{G_i^u\leq h_i-\frac{R_G(i,n)}{R_G(n,n)} x\corO{+}\sigma_i^u W\})
	\nonumber\\
	&\leq & P(|W|\geq (\log n)^{3/4-\epsilon/2})+
	P(\cap_{i=1}^{n-1}
	\{G_i^u\leq h_i-\frac{R_Y(i,n)}{R_Y(n,n)} x+(\log n)^{3/4-\epsilon/2}
+c\})\nonumber\\
	&\leq & P(|W|\geq (\log n)^{3/4-\epsilon/2})+
	P(\cap_{i=1}^{n-1}
	\{Y_i^u\leq h_i-\frac{R_Y(i,n)}{R_Y(n,n)} x+(\log n)^{3/4-\epsilon/2}
+c\})\nonumber\\
	&\leq & P(|W|\geq (\log n)^{3/4-\epsilon/2})+
	P(\cap_{i=1}^{n-1}
	\{\tilde Y_i\leq h_i-\frac{R_Y(i,n)}{R_Y(n,n)} x
	-\eta_i^u W_i+(\log n)^{3/4-\epsilon/2}
+c\})\nonumber\\
	&\leq & 
	nP(|W|\geq (\log n)^{3/4-\epsilon/2})+
	P(\cap_{i=1}^{n-1}
	\{\tilde Y_i\leq h_i-\frac{R_Y(i,n)}{R_Y(n,n)} x
	+(\log n)^{3/4}
\}).\nonumber
\end{eqnarray*}
Combined with \eqref{eq-cond4} and changing the value of $\epsilon$ if
necessary we obtain
\begin{eqnarray*}
	&&p_{B,G}(n,t,h)
	\leq  C_1 e^{-(\log n)^{3/4-\epsilon}}\\
	&&+
	(1+\epsilon)\int_{[t,t+1]} \phi(x,R_Y(n,n))
	P(\tilde Y_i\leq h_i-\frac{R_Y(i,n)}{R_Y(n,n)} x+(\log n)^{3/4},
	i=1,\ldots,n-1)dx\\
	&=&
	C_1e^{-(\log n)^{3/4-\epsilon}}+ (1+\epsilon) P(Y_i\leq h_i+(\log n)^{3/4},
	i=1,\ldots,n-1, Y_n\in [t,t+1])\,.
\end{eqnarray*}
The proof of the lower bound in \eqref{eq-barrier} is similar, using 
now the processes $\tilde Y^l$ and $\tilde G^l$ together with
\eqref{eq-constlb1} and \eqref{eq-constlb2},
instead of the processes $\tilde Y^u$ and $\tilde G^u$ together with
\eqref{eq-constub1} and \eqref{eq-constub2}. Further details are omitted.
\end{proof}
The following lemma,
\corO{dealing with the case that $h_i$ depends on $n$ but not on $i$,}
will be used in conjunction with 
Proposition \ref{prop-comp-bar}.
\begin{lemma}
  \label{lem-ratio}
  Fix $\epsilon<1/2$ small.
  In the notation of Proposition \ref{prop-comp-bar}, assume that
  $h_i=h(n)\geq (\log n)^{1-\epsilon}$ and $t=t(n)$ is such that
  $-\log n/10>t>-n^{1/2-\epsilon}$. Let $t'=t'(n)$ be such that 
  $|t'-t|/t\to 0$. Then,
  \begin{equation}
    \label{eq-ratio-comp}
    \lim_{n\to\infty} \frac{p_{B,Y}(n,t,h+(\log n)^{3/4})}
    {p_{B,Y}(n,t',h-(\log n)^{3/4})}=1\,,
  \end{equation}
  and the convergence is uniform in such choices.
  \end{lemma}
  \begin{proof}
    Let $W_t$ be standard Brownian motion. 
    Note that for any real numbers $s<0,g>0$,
    \begin{eqnarray*} 
    &&P(W_t\leq g, t\leq n; W_n\in [s,s+1])\leq p_{B,Y}(n,s,g)\\
    & &\leq P(W_t\leq (1+\epsilon)g, t\leq n; W_n\in [s,s+1])+
    nP(\max_{t\in [0,1]} W_t>\epsilon g)\,.
  \end{eqnarray*}
    For $g\geq (\log n)^{1-\epsilon}/2$, we have
    $$P(\max_{t\in [0,1]} W_t>\epsilon g)\leq e^{-c(\log n)^{2(1-\epsilon)}}
    \leq n^{-10}\,.$$
    One the other hand, with $p_u(x)=\frac{1}{\sqrt{2\pi u}}e^{-x^2/2u}$
    denoting the standard heat kernel we have \corO{from the reflection principle}
    that
    \begin{equation}
      \label{eq-dec6a}
      P(W_t\leq g, t\leq n; W_n\in [s,s+1])=
    \int_{[s,s+1]}
  [p_n(x)-p_n(x-2g) ]dx\,.
\end{equation}
  Combining the last three displays and using that for 
  $\corO{0<g<<n^{1/2}}$ and $0>x>>-n^{1/2}$,
  $[p_n(x)-p_n(x-2g) ]\sim - C \frac{\corO{g(g-x)}}{n^{3/2}}$,
   we conclude that for $g,s$ satisfying the above constraints,
  $$\lim_{n\to \infty}
  \frac{ P(W_t\leq g, t\leq n; W_n\in [s,s+1])}
  {P(W_t\leq (1+\epsilon)g, t\leq n; W_n\in [s,s+1])}=1+O(\epsilon)\,.$$
  Combining this with \eqref{eq-dec6a} 
   (to handle the difference between $t$ and $t'$)
  yields the lemma after some elementary manipulations.
\end{proof}
The following corollary specializes Proposition \ref{prop-comp-bar}
to the case of almost linear barriers.
\begin{corollary}
	\label{cor-UBbarrier}
	Assume that there exists a constant $C$, independent of $n$, so that
	$|R_G(i,j)-R_Y(i,j)|\leq C$
	and $|h_i|\leq C \log n$
	for all $i,j\leq n$.
	Then there exists $C_1=C_1(C)$ so that, for any $t<h_{n-1}$ with 
	$|t|\leq n^{1/4}$, and all $n$ large enough,
	\begin{equation}
		\label{directbar}
		p_{B,G}(n,t,h)\leq C\frac{t\log n}{n^{3/2}}.
	\end{equation}
\end{corollary}
\begin{proof}
	From Proposition \ref{prop-comp-bar} we obtain that
	$$p_{B,G}(n,t,h)\leq
	2p_{B,Y}(n,t,\corO{2C}\log n)+C_1e^{-(\log n)^{3/2-\epsilon}}.$$
The conclusion follows from the Ballot theorem for random walk,
see e.g. \cite{AB,Carr}.
\end{proof}

\subsection{Two rays estimate} We will need two different two ray estimates, depending on the separation between the rays.
\begin{proposition}[Separated rays]
    \label{prop-separated}
  Let $\{G_i^{(j)}\}_{i=1}^n$, $j=1,2$ two centered Gaussian processes, each 
  satisfying the assumptions of Proposition \ref{prop-comp-bar}. Assume 
  further that $\corO{|E(G_i^{(1)} G_j^{(2)})|}\leq C$. 
  Define
  $p_{B;G^{(1)},G^{(2)}}(n,t,h)=P(B_{G^{(1)}}(n,t,h)\cap \corO{B_{G^{(2)}}(n,t,h) })$.
  Then, there exists 
  $C_1=C_1(\epsilon,C)$ so that
  \begin{align}
    \label{eq-tworay-1}
    & p_{B;G^{(1)},G^{(2)}}(n,t,h) \\
     &\leq  
     (1+\epsilon) p_{B;Y}(n,t,h+(\log n)^{3/4})^2
  +C_1 e^{-(\log n)^{3/2-\epsilon}}.
  \nonumber
\end{align}
\end{proposition}
\begin{proof}

  The argument is very similar to the proof of Proposition 
  \ref{prop-comp-bar}. The only difference is that instead of
  conditioning on $G_n$, one needs to condition on $(G_n^{(1)},G_n^{(2)})$,
  where the latter vector has
  covariance matrix $nI_2+ B$ where $B$ is a $2\times 2$ matrix which 
  has bounded norm (and with inverse $I_2/n +B'/n^2$ with $B'$ having bounded
  norm).
  Performing the conditioning, in the comparison with (independent) 
  Gaussian bridges, instead of the variables 
  $W, W_i$ one needs to use variables $W^{(j)}, W_i^{(j)}$ ($j=1,2$), with
  $E (W_i^{(j)})^2\leq C$ and $E(W^{(j)})^2\leq C$, but otherwise may be 
  correlated. Since the estimate 
  $$P(\max_{j=1,2} \max_{i=1}^n 
  |W^{(j)}_i|>(\log n)^{3/4-\epsilon/2})\leq 2n P(|W|>(\log n)^{3/4-\epsilon/2})$$
  holds regardless of the correlation, the argument  of
  Proposition \ref{prop-comp-bar} carries through.
  We omit further details.
\end{proof}

We next consider rays with a common trunk. We introduce some notation. 
Given two processes $\{Z^{(1)}_i\}_{i=1}^n,\{Z^{(2)}_i\}_{i=1}^n$, a sequence 
$h=(h_i)$, an integer $0<k<n$ and real numbers $z,t$, define the event
\begin{align*}
  &A_{Z^{(1)},Z^{(2)}}(n,t,h,k,z) \\
  &=\{Z_i^{(\ell)}\leq h_i, i=1,\ldots,n-1, \ell=1,2; 
Z_n^{(\ell)}\in [t,t+1], \ell=1,2; Z_k^{(1)}\in [z,z+1]\},
\end{align*}
and its probability 
$p_{A;Z^{(1)},Z^{(2)}}(n,t,h,k,z)=P(A_{Z^{(1)},Z^{(2)}}(n,t,h,k,z))$.
In addition, for any $0<k<n$,
introduce the centered Gaussian processes $\{Y_i^{(1),k}\}_{i=1}^n$,
$\{Y_i^{(2),k}\}_{i=1}^n, $
such that
$$R_{Y^{(\ell),k}}(i,j)=i\wedge j,\quad 
E(Y^{(1),k}_i Y^{(2),k}_j)=:R_{Y^{(1),k},Y^{(2),k}}=i\wedge j\wedge k\,.$$
Note that the processes $\{Y^{(\ell),k}\}$ possess a common Gaussian random walk
part up to time $k$ (i.e., $Y^{(1),k}_i=Y^{(2),k}_i, i=1, \ldots,k$), 
and from $k$ onward increments in both processes
are independent of each other. 
\begin{proposition}[Overlapping rays]
    \label{prop-overlap}
    Let $\{G_i^{(j)}\}_{i=1}^n$, $j=1,2$ \corO{be}
    two centered Gaussian processes, each 
  satisfying the assumptions of Proposition \ref{prop-comp-bar}. Assume 
  further that there exist $C>10$ and  $C^2<k<n-C^2$ so that
  $|z|\leq (\log k)^2,$  $|t-z|\leq (\log (n-k))^2$ 
  and
  \begin{equation}
	  \label{eq-compare1}
  |E(G_i^{(1)} G_j^{(2)})- \min(i,j,k)|\leq C.
  \end{equation}
  Then, for any $\epsilon>0$ there exists $C_1=C_1(\epsilon,C)$ so that
  \begin{align}
    \label{eq-common1}
&p_{A;Z^{(1)},Z^{(2)}}(n,t,h,k,z)\\
    \nonumber
&\leq (1+\epsilon)
P(Y_i^{(\ell)}\leq h_i+(\log n)^{3/4}, i=C_1,\ldots,n-C_1, \ell=1,2;\\
\nonumber
&\quad
Y_n^{(\ell)}\in [t,t+1], \ell=1,2; Y_k^{(1)}\in [z,z+1])\\
&\quad +C_1e^{-(\log n)^{3/2-\epsilon}}\,.
\nonumber
\end{align}
  \end{proposition}
  \begin{proof}
As in the proof of Proposition \ref{prop-comp-bar}, the first step is 
a conditioning. Here we condition on the vector $W:=
(G_k^{(1)}, G_n^{(1)}-G_k^{(1)},G_n^{(2)}-G_k^{(2)})$, which 
possesses the covariance matrix
$$R_W=\begin{pmatrix}
  k&0&0\\
  0&(n-k)&0\\
  0&0&(n-k)
\end{pmatrix}+ {\bf A} $$
where the matrix ${\bf A}$ has norm bounded by $3C$. In particular (because 
$k,n-k\geq C^2>3C$), the matrix $R_W$ is invertible and 
$$R_W^{-1}=\begin{pmatrix}
  1/k&0&0\\
  0&1/(n-k)&0\\
  0&0&1/(n-k)
\end{pmatrix}+ \widehat{\bf A} $$
where the norm of $\widehat{\bf A}$ is bounded. 
We obtain
$$\hat G^{(\ell)}_i:=E(G^{(\ell)}_i|W)=
\sum_{m,m'=1}^3 
E(G^{(\ell)}_i W_m) R_W^{-1}(m,m') W_{m'}.$$
Defining the analogous quantities for the processes $Y^{(\ell),k}$, 
\corO{writing  $W^Y$ instead of  $W$}, we have that
on the event $ \corO{ \cap_{j=1}^3 \{|W_j-W_j^Y|\leq 1\}}$ one has
$$\sup_{i,\ell} |\hat G^{(\ell)}_i-\hat Y^{(\ell),k}_i|\leq c,$$
for some $c=c(C)$. Similarly, with
$\tilde G^{(\ell)}=G^{(\ell)}-\hat G^{(\ell)}$, we have that
with $ G=(G^{(1)}, G^{(2)})$ and
$\tilde G=(\tilde G^{(1)},\tilde G^{(2)})$,
$$ R_{\tilde G}(i,j)=
 R_{G}(i,j)-
\sum_{m,m'=1}^3 
E(G_i W_m) R_W^{-1} (m,m')
E(G_j W_{m'}) \,.$$
We conclude that under our assumptions,
$$\max_{i,j}\|R_{\tilde G}(i,j)-R_{\tilde Y}(i,j)\|\leq c\,.$$
The conclusion now follows by using Slepian's lemma as in
the proof of Proposition \ref{prop-comp-bar}.
  \end{proof}

  \corO{
    \noindent {\bf Acknowledgment} We thank an anonymous referee for an 
  extremely useful and detailed referee report, that led to a significant improvement in the paper.}
      \bibliographystyle{alpha}
\bibliography{CUE}
\end{document}